\pdfoutput=1
\documentclass[11pt]{article}

\usepackage[letterpaper,left=1in,right=1in,top=1in,bottom=1in]{geometry}
\usepackage{amssymb,amsmath,graphicx,amsfonts}
\usepackage{amsthm}
\usepackage{mathrsfs}
\usepackage{multirow}
\usepackage{graphicx,epsfig}
\usepackage{color}
\usepackage{empheq}
\usepackage{cases}
\usepackage{cite}
\usepackage{amsgen}
\usepackage{amscd}
\usepackage{booktabs}
\usepackage{latexsym}
\usepackage{euscript}
\usepackage{bm,hyperref,subfigure}
\usepackage{enumerate}

\theoremstyle{plain}
\newtheorem{theorem}{Theorem}[section]
\newtheorem{lemma}[theorem]{Lemma}

\newtheorem{proposition}[theorem]{Proposition}
\theoremstyle{definition}

\newtheorem{remark}{Remark}[section]

\newenvironment{keywords}%
  {\par\medskip\noindent{\bf Key words. }\ignorespaces}{\par\medskip}
\newenvironment{AMS}%
  {\par\medskip\noindent{\bf AMS subject classifications. }\ignorespaces}{\par\medskip}

\allowdisplaybreaks
\numberwithin{equation}{section}
\numberwithin{figure}{section}

\def\u{{\bf u}}
\def\v{{\bf v}}
\def\x{{\bf x}}
\def\y{{\bf y}}
\def\z{{\bf z}}

\def\e{{\bf e}}
\def\dof{{N}}
\def\K{{\bf K}}
\def\M{{\bf M}}
\def\A{{\bf A}}
\def\n{{\bf n}}
\def\H{{\bf H}}

\def\f{{\bf f}}
\def\w{{\bf w}}

\def\d{{\mathrm d}}

\def\R {{\mathbb R}}

\def\le{\leqslant}
\def\ge{\geqslant}
\def\Omega{\varOmega}
\def\Delta{\varDelta}

\title{Dynamic Ritz projection of mean curvature flow\\ and optimal ${\bf L^2}$ convergence of parametric FEM\thanks{This work is supported in part by AMSS-PolyU Joint Laboratory and the Research Grants Council of Hong Kong (PolyU/GRF15303022 and PolyU/RFS2324-5S03).}} 

\author{
Buyang Li\thanks{Department of Applied Mathematics, The Hong Kong Polytechnic University, Hong Kong. E-mail address: buyang.li@polyu.edu.hk and claire.tang@polyu.edu.hk} and Rong Tang\footnotemark[\value{footnote}]
}

\date{}

\begin{document}

\maketitle

\begin{abstract} 
A new approach is developed to study the convergence of parametric finite element approximations to the mean curvature flow of closed surfaces in three-dimensional space. In this approach, the error analysis is conducted by comparing the numerical solution to a dynamic Ritz projection of the mean curvature flow introduced in this paper, rather than an interpolation of the mean curvature flow, as commonly used in the literature. The errors associated with the dynamic Ritz projection in approximating the mean curvature flow are established in the $L^2$ and $W^{1,p}$ norms. Leveraging these results, optimal-order convergence of parametric finite element methods for mean curvature flow of closed surfaces in the $L^\infty(0,T;L^2)$ norm is proved, including the convergence of parametric finite element methods with piecewise linear finite elements. 
\end{abstract} 




\setlength\abovedisplayskip{4pt}
\setlength\belowdisplayskip{4pt}

\section{Introduction}\label{Se:1}

The numerical approximation of surface evolution under geometric flows, including mean curvature flow, Willmore flow and surface diffusion, have been intensively studied in the past. The most well-known example of geometric flows is mean curvature flow, which describe the evolution of a surface $\Gamma(t)\subset\R^3$, with certain initial condition $\Gamma(0)=\Gamma^0$, moving with velocity 
\begin{align*}
v = - H n ,
\end{align*} 
where $H$ and $n$ denote the mean curvature and the normal vector of the surface. 
By using identity $Hn=- \Delta_{\Gamma}{\rm id}_{\Gamma}$, where $\Delta_{\Gamma}$ and ${\rm id}_{\Gamma}$ denote the Laplace--Beltrami operator and identity map on surface $\Gamma$, mean curvature flow can also be written as 
\begin{equation}
\label{MCF-2}
v = \Delta_{\Gamma}{\rm id}_{\Gamma} .
\end{equation} 
By utilizing the formulation in \eqref{MCF-2}, Dziuk introduced the following type of finite element method (FEM) in \cite{Dziuk1990} for approximating surface evolution under mean curvature flow: assuming that $\Gamma(t_{m-1})$ is already approximated by a piecewise triangular surface $\Gamma_h^{m-1}$, find a parametrization of surface $\Gamma_h^m$ through a finite element function $u^m_h: \Gamma_h^{m-1}\rightarrow \R^3$ which is determined by some weak formulation of \eqref{MCF-2}. Such methods are referred to as parametric FEMs. 

Since 1990s, parametric FEMs have been widely used for approximating surface evolution under various geometric flows and interface evolution in various related problems. Many novel numerical methods were developed to address the challenges (such as prevention of mesh distortion and preservation of energy stability) in approximating surface evolution; see the artificial tangential motion constructed by Barrett, Garcke \& N\"urnberg's (BGN) \cite{BGN07, BGN08,BGN08b}, Elliott \& Fritz \cite{Elliott-Fritz-2017,Elliott-Fritz-2016}, Hu \& Li \cite{HL22}, Duan \& Li \cite{Duan-Li-2024}, and the structure-preserving parametric FEMs \cite{Bao-2020-SISC,Bao-Li-2024,Zhao-Jiang-Bao-2020}. These techniques have significantly improved the performance of parametric FEMs in approximating surface evolution under geometric flows. However, proving convergence of these methods remains challenging. 

Convergence of parametric FEMs for geometric flows has been addressed for mean curvature flow and Willmore flow of curves in \cite{Dziuk94,DeckelnickDziuk,Deckelnick-Dziuk-2009,Bartels-2013,Elliott-Fritz-2017,Li-2020-SINUM,Ye-Cui-SINUM}, and for graph surfaces and axisymmetric surfaces in \cite{deckelnick1995convergence,Deckelnick-Dziuk-2006,DN2021SJNA,BDN2021IJNA,Deckelnick-Styles-2022}. However, the techniques developed in the proofs are not applicable to the analysis of geometric flow of general closed surfaces. We also refer to \cite{Elliott-Fritz-2017,BDN2021IJNA,DN2021SJNA,Mierswa2020,Deckelnick-Styles-2022,Barrett-Deckelnick-Styles-2017} for the proofs of convergence of parametric FEMs for mean curvature flow of closed curves, graph surfaces, axisymmetric surfaces, and surfaces of torus type, with Elliott \& Fritz's tangential motion \cite{Elliott-Fritz-2017,Elliott-Fritz-2016} generated by a tangential transformation on $\Gamma^0$. 

Convergence of the parametric FEMs for mean curvature flow, Willmore flow and surface diffusion of closed surfaces was firstly {{proved in \cite{KLL19,KLL-2021,HL22,Elliott-Garcke-Kovacs,MR4454930} for some}} equivalent formulations of the geometric flows which couple the velocity equation of geometric flows with the geometric evolution equations of mean curvature $H$ and normal vector $n$, by formulating the algorithms into evolving FEMs as \cite{Dziuk-Elliott-2007} and utilizing the matrix-vector formulation of evolving FEMs introduced in \cite{KLLP2017}. For example, convergence of parametric FEMs for mean curvature flow of closed surfaces is proved for the following equivalent formulation of mean curvature flow: 
\begin{align}
\label{KLL}
\begin{aligned}
\partial_tX&=v\circ X &&\mbox{on}\,\,\,\Gamma^0 \\
v &= - H n &&\mbox{on}\,\,\,\Gamma(t) \\
\partial_t^\bullet H-\Delta_{\Gamma(t)} H &= |\nabla_{\Gamma(t)} n|^2 H &&\mbox{on}\,\,\,\Gamma(t)\\
\partial_t^\bullet n-\Delta_{\Gamma(t)} n &= |\nabla_{\Gamma(t)} n|^2 n &&\mbox{on}\,\,\,\Gamma(t) ,
\end{aligned}
\end{align} 
where $X(\cdot,t):\Gamma^0\rightarrow\R^3$ is the flow map which determines $\Gamma(t)=\{X(p,t):p\in\Gamma^0\}$, and $\partial_t^\bullet $ denotes material derivative along the particle trajectories of the flow map, i.e., 
$$
\partial_t^\bullet u(x,t) = \frac{\d}{\d t} u (X(p,t),t) \,\,\,\mbox{at point}\,\,\, x=X(p,t)\,\,\,\mbox{on}\,\,\,\Gamma(t) . 
$$
The analyses in \cite{KLL19,KLL-2021,HL22} are restricted to finite elements of degree $k\ge 2$. This condition is required for proving convergence of numerical solutions in the $W^{1,\infty}$ norm in order to control the nonlinear terms and to utilize the equivalence of $L^p$ and $W^{1,p}$ norms of functions on the numerical-solution surface and interpolated surface. 

The convergence of other parametric FEMs, designed for approximating \eqref{MCF-2} instead of \eqref{KLL}, is more challenging due to the degeneracy of the nonlinear Laplacian--Beltrami operator $\Delta_{\Gamma}{\rm id}_{\Gamma}$ acting on surface $\Gamma$; see the discussions in \cite{Bai-Li-2023,Li21}. In the spatially semi-discretization setting, the convergence of parametric FEMs for mean curvature flow of surfaces with formulation \eqref{MCF-2} was proved in \cite{BL22A,Li21} for finite elements of degree $k\ge 6$ based on a discovery that the nonlinear Laplacian--Beltrami operator is $H^1$ elliptic with respect to the normal component of the trajectory error (error between the exact and numerical flow maps). The restriction to finite elements of degree $k\ge 6$ is needed to control the nonlinearities in error analysis by using the very weak estimates obtained based on the partial $H^1$ ellipticity in the normal direction. In the fully discrete setting, the convergence of Dziuk's semi-implicit parametric FEM for mean curvature flow of surfaces was proved in \cite{Bai-Li-2023} for finite elements of degree $k\ge 3$ based on a discovery that the nonlinear Laplacian--Beltrami operator is $H^1$ elliptic with respect to the distance error (distance between the exact and numerical surfaces) --- the strong $H^1$ ellipticity in both normal and tangential directions leads to stronger estimates of the errors for controlling the nonlinearities and therefore reduces the requirements of finite elements degree from $k\ge 6$ to $k\ge 3$. 

In summary, the convergence of some fundamental algorithms for geometric flows of surfaces still remains open. The existing proofs of convergence of parametric FEMs for mean curvature flow and other geometric flows of closed surfaces are all based on optimal-order $H^1$-norm error estimates that require using finite elements of degree $k\ge 2$ to control the $W^{1,\infty}$ boundedness of numerical solutions of surface position $X$, mean curvature $H$ and normal vector $n$. The following two questions remain open:\vspace{3pt} 
\begin{itemize}
\item
Convergence of parametric FEMs for mean curvature flow and other geometric flows of closed surfaces with piecewise linear finite elements still remains open.\vspace{3pt} 

\item
Optimal-order convergence of parametric FEMs in the $L^\infty(0,T;L^2)$ norm for these geometric flows remains open.\vspace{3pt}  
\end{itemize}
The two questions are addressed simultaneously in the current paper for approximating formulation \eqref{KLL} of mean curvature flow. It turns out that the two questions are closely related such that our answer to the second question (by introducing a dynamic Ritz projection which reduces the remainders in the error equations) also addresses the first question. 
In particular, the new framework developed in this paper, by defining and utilizing a dynamic Ritz projection of geometric flow in error analysis, is promising for proving convergence of parametric FEMs for geometric flows with optimal-order convergence and lower-degree finite elements. 

\section{Dynamic Ritz projection and main results}

Let $\Gamma_h^0$ be a piecewise {polynomial} surface that interpolates the smooth surface $\Gamma^0$, with each piece being the image of the reference triangle under a polynomial map of degree $k\ge 1$, and assume that the curved triangles are shape-regular and quasi-uniform with mesh size $h$; see \cite{Demlow-2009,Kovacs2018} and Lenoir's isoparametric approximation of a surface\cite{Lenoir-1986}. 

Let $\x^0=(p_1,\cdots,p_\dof)\in \mathbb{R}^{3\dof}$ be the nodal vector that collects all the nodes $p_j\in \mathbb{R}^3$, $j=1,\ldots,\dof$, in $\Gamma_h^0$. We evolve $\x^0$ in time and denote its position at time $t$ by $\x(t)=(x_1(t),\cdots,x_\dof(t))$, which determines a piecewise (possibly curved) triangular surface $\Gamma_h[\x(t)]$ via piecewise polynomial interpolation on a plane reference triangle, and denote by $S_h[\x(t)]$ the finite element space of polynomial degree $k$ on the piecewise triangular surface $\Gamma_h[\x(t)]$. 

There exists a unique finite element function $X_h(\cdot,t)$ of piecewise polynomial degree $k$ defined on $\Gamma_h[\x^0]$ satisfying 
$$
X_h(p_j,t)=x_j(t)\quad\mbox{for}\quad j=1,\dots,\dof.
$$ 
This is the discrete flow map which maps $\Gamma_h[\x^0]$ to $\Gamma_h[\x(t)]$. The semidiscrete parametric FEM for \eqref{KLL} is to find 
\begin{equation*}
(X_h(\cdot,t),v_h(\cdot,t),H_h(\cdot,t),n_h(\cdot,t))
\in S_h[\x^0]^3 \times S_h[\x(t)]^3  \times S_h[\x(t)] \times S_h[\x(t)]^3  
\end{equation*}
such that the following weak formulation holds for all $(\chi_H,\chi_n) \in S_h(\Gamma_h[\x])  \times S_h(\Gamma_h[\x])^3 $: \\[-10pt]
\begin{subequations}\label{weak-form}
\begin{align}
\label{BFormX}
&\partial_{t}X_h = v_h\circ X_h \quad\hspace{5pt} \mbox{on}\,\,\, \Gamma_h[\x^0] \\ 
\label{BFormb}
&v_h =  -I_h(H_hn_h) \quad\mbox{on}\,\,\, \Gamma_h[\x] \\
&\int_{\Gamma_h[\x]} \partial_{t,h}^\bullet H_h  \chi_H 
+ \int_{\Gamma_h[\x]} \nabla_{\Gamma_h[\x]} H_h \cdot \nabla_{\Gamma_h[\x]} \chi_H 
= \int_{\Gamma_h[\x]} |\nabla_{\Gamma_h[\x]}n_h |^2 H_h \chi_H 
\label{H-Fem} \\
&\int_{\Gamma_h[\x]} \partial_{t,h}^\bullet n_h \cdot \chi_n
+\int_{\Gamma_h[\x]} \nabla_{\Gamma_h[\x]} n_h \cdot \nabla_{\Gamma_h[\x]} \chi_n 
= \int_{\Gamma_h[\x]}  |\nabla_{\Gamma_h[\x]}n_h|^2 n_h \cdot \chi_n ,
\label{n-Fem}  
\end{align}
\end{subequations}
where $I_h$ denotes the Lagrange interpolation onto the finite element space $S_h[\x(t)]$ {and $\partial_{t,h}^\bullet$ denotes the material derivative on $\Gamma_h[\x(t)]$ with respect to the discrete flow map $X_h(\cdot,t)$.}
The initial value for  \eqref{weak-form} can be chosen as follows: 
$X_h(\cdot,0)={\rm id}$ on $\Gamma_h[\x^0]$; $H_h(\cdot,0)$ and $n_h(\cdot,0)$ are the Lagrange interpolations of $H(\cdot,0)$ and $n(\cdot,0)$, respectively. 

Let $\x^*=\x^*(t)$ be the nodal vector which collects the nodes evolving according to the exact flow map $X(\cdot,t):\Gamma^0\rightarrow \R^3$, and denote by $\Gamma_h[\x^*]$ the piecewise curved triangular surface that interpolates $\Gamma(t)$ at the nodes in $\x^*$. The finite element space on $\Gamma_h[\x^*]$ is denoted by $S_h(\Gamma_h[\x^*])$. 

The analyses in \cite{KLL19,KLL-2021,HL22} are based on estimating the error between numerical solution $(X_h,v_h,H_h,n_h)$ and $(\hat I_h^*X,\hat I_h^*v,\hat R_h^*H, \hat R_h^*n)$, where $\hat I_h^*X$ and $\hat I_h^*v$ are Lagrange interpolations of $X$ and $v$ onto $\Gamma_h[\x^0]$ and $\Gamma_h[\x^*]$, respectively, and $\hat R_h^*H, \hat R_h^*n\in S_h(\Gamma_h[\x^*])$ are the linear Ritz projections of $H,n$ onto surface $\Gamma_h[\x^*]$, respectively, defined by 
\begin{subequations}\label{Ritz-old} 
\begin{align}
\int_{\Gamma_h[\x^*]}
(\hat R_h^*H\,\varphi_h+\nabla_{\Gamma_h[\x^*]} \hat R_h^*H\cdot \nabla_{\Gamma_h[\x^*]} \varphi_h) &=
\int_{\Gamma}
(H\varphi_h^l+\nabla_{\Gamma} H \cdot \nabla_{\Gamma}  \varphi_h^l)  &&\notag\\
&\hspace{44pt}\forall\,\varphi_h\in S_h(\Gamma_h[\x^*]) , \\[5pt]
\int_{\Gamma_h[\x^*]}
(\hat R_h^*n\,\cdot \phi_h+\nabla_{\Gamma_h[\x^*]} \hat R_h^*n\cdot \nabla_{\Gamma_h[\x^*]} \phi_h) &=
\int_{\Gamma}
(n \cdot \phi_h^l+\nabla_{\Gamma} n \cdot \nabla_{\Gamma}  \phi_h^l)  &&\notag\\
&\hspace{38pt}\forall\,\phi_h\in S_h(\Gamma_h[\x^*])^3 ,
\end{align}
\end{subequations}
In \eqref{Ritz-old}, $\varphi_h^l$ and $\phi_h^l$ denote the lifts of functions $\varphi_h\in S_h(\Gamma_h[\x^*])$ and $\phi_h\in S_h(\Gamma_h[\x^*])^3$ to the exact surface $\Gamma=\Gamma(t)$, respectively; see Section \ref{section:lifts}. By this definition of Ritz projection, $(\hat I_h^*X,\hat I_h^*v,\hat R_h^*H, \hat R_h^*n)$ satisfies numerical scheme \eqref{weak-form} up to some remainders, i.e.,
{\small 
\begin{subequations}\label{weak-interpolated}
\begin{align}
\label{interpolated-X}
\partial_{t}\hat I_h^*X &= \hat I_h^*v\circ \hat I_h^*X &&\mbox{on}\,\,\, \Gamma_h[\x^0] \\ 
\label{interpolated-v}
\hat I_h^*v &=  - \hat I_h^*(\hat R_h^*H \, \hat R_h^*n) + d_v &&\mbox{on}\,\,\, \Gamma_h[\x^*] \\
\notag\int_{\Gamma_h[\x^*]} \partial_{t,h}^\bullet \hat R_h^*H \chi_H 
&+ \int_{\Gamma_h[\x^*]} \nabla_{\Gamma_h[\x^*]} \hat R_h^*H \cdot \nabla_{\Gamma_h[\x^*]} \chi_H && \forall\,\chi_H \in S_h(\Gamma_h[\x^*])\\
\label{interpolated-H}&= \int_{\Gamma_h[\x^*]} |\nabla_{\Gamma_h[\x^*]}\hat R_h^*n |^2 \hat R_h^*H \chi_H + \int_{\Gamma_h[\x^*]} d_H \chi_H \\  
\int_{\Gamma_h[\x^*]} \partial_{t,h}^\bullet \hat R_h^*n \cdot \chi_n
&+ \int_{\Gamma_h[\x^*]} \nabla_{\Gamma_h[\x^*]} \hat R_h^*n \cdot \nabla_{\Gamma_h[\x^*]} \chi_n && \forall\,\chi_n \in S_h(\Gamma_h[\x^*])^3 \notag \\
&= \int_{\Gamma_h[\x^*]}  |\nabla_{\Gamma_h[\x^*]}\hat R_h^*n|^2 \hat R_h^*n \cdot \chi_n 
+ \int_{\Gamma_h[\x^*]} d_n \cdot\chi_n , 
\label{interpolated-n}  
\end{align}
\end{subequations}
where $ d_v = \hat I_h^*(\hat R_h^*H \, \hat R_h^*n) - \hat I_h^*(\hat I_h^*H \, \hat I_h^*n) $, $d_H$ and $d_n$ are remainders which satisfy the following estimates (for some constant $C$ which is independent of mesh size $h$): 
\begin{align}
&\|d_v\|_{L^2(\Gamma_h[\x^*])} + h\|d_v\|_{H^1(\Gamma_h[\x^*])} \le Ch^{k+1} , \\ 
&\Big| \int_{\Gamma_h[\x^*]} d_H \chi_H  \Big | 
+ \Big| \int_{\Gamma_h[\x^*]} d_n \chi_n \Big| 
\le Ch^{k+1} 
(\|\chi_H \|_{H^1(\Gamma_h[\x^*])} + \|\chi_n \|_{H^1(\Gamma_h[\x^*])} ) . 
\end{align} 

We denote by $\hat x_h$, $\hat v_h$, $\hat H_h$ and $\hat n_h$ the lift of $X_h$, $v_h$, $H_h$ and $n_h$ to $\Gamma_h[\x^*]$, i.e., finite element functions on $\Gamma_h[\x^*]$ with the same nodal vectors as $X_h$, $v_h$, $H_h$ and $n_h$, respectively. 
By using this notation, we can define the following finite element error functions on $\Gamma_h[\x^*]$: 
\begin{align}
\begin{aligned}
\hat e_{x,h} &= \hat x_h - {\rm id}_{\Gamma_h[\x^*]} ,
&& \hat e_{v,h}= \hat v_h - \hat I_h^*v, \\
\hat e_{H,h} &= \hat H_h - \hat R_h^*H, 
&&\hat  e_{n,h}= \hat n_h - \hat R_h^*n ,  
\end{aligned}
\end{align}
where $\hat e_{x,h}$ represents the error between surfaces $\Gamma_h[\x]$ and $\Gamma_h[\x^*]$, and $\hat e_{v,h}$, $\hat e_{H,h}$ and $\hat e_{n,h}$ represent the errors in the numerical approximations of velocity, mean curvature, and normal vector, respectively. 
Such definitions of error functions are used in \cite{KLL19,KLL-2021,HL22} in proving convergence of parametric FEMs for mean curvature flow and Willmore flow. However, the presence of remainder $d_v$ in \eqref{interpolated-v} hinders people from proving optimal-order convergence in the $L^\infty(0,T;L^2)$ norm for the reason that $\|\hat e_{x,h}\|_{H^1(\Gamma_h[\x^*])}$ frequently appears due to surface location errors, and this needs to be controlled by using $\|d_v\|_{H^1(\Gamma_h[\x^*])} $ instead of $\|d_v\|_{L^2(\Gamma_h[\x^*])} $. Therefore, optimal-order convergence of parametric FEMs for mean curvature flow was only proved in the $L^\infty(0,T;H^1)$ norm in the literature, and such $L^\infty(0,T;H^1)$ error analysis typically requires $W^{1,\infty}$ boundedness of the numerical solutions in order to bound the nonlinear terms in the error analysis. This requires the convergence order to be sufficiently high in order to apply the inverse inequality of finite element functions to prove $W^{1,\infty}$ boundedness of the numerical solutions, and this limits the error analyses to high-order finite elements of degree $k\ge 2$. 

Our solution to the above-mentioned two questions, i.e., optimal-order convergence in the $L^\infty(0,T;L^2)$ norm and convergence of parametric FEM for mean curvature flow with piecewise linear finite elements, is based on the following two observations: 
\begin{enumerate}
\item
The evolution equations of $H$ and $n$ in \eqref{KLL} have similar nonlinear structures as the harmonic map heat flow studied in \cite{Gui-Li-Wang-2022}, where the proof of optimal-order convergence of FEMs in the $L^\infty(0,T;L^2)$ norm only requires utilizing $W^{1,4}$ boundedness of $e_{H,h}$ and $e_{n,h}$ to bound the nonlinear terms appearing in the error analysis. This motivates us to consider $L^\infty(0,T;L^2)$ error estimates of parametric FEMs for \eqref{KLL} based on boundedness of $e_{H,h}$ and $e_{n,h}$ in a norm weaker than the $W^{1,\infty}$ norm, instead of the $L^\infty(0,T;H^1)$ error estimates considered in \cite{KLL19,KLL-2021,HL22}. The latter approach requires $W^{1,\infty}$ boundedness of $e_{H,h}$ and $e_{n,h}$ (this cannot be proved for piecewise linear finite elements so far) to bound the nonlinear terms in error analysis. 

\item
However, convergence of parametric FEMs for mean curvature flow of closed surfaces requires $W^{1,\infty}$ boundedness of $\hat e_{x,h}$ to guarantee that the norm equivalence of finite element functions on $\Gamma_h[\x]$ and $\Gamma_h[\x^*]$ associated to a common nodal vector. Such norm equivalence is frequently used and can hardly be relaxed in analyzing the error of parametric FEMs for surface evolution under geometric flows. However, an optimal-order error estimate such as 
\begin{align}\label{ex-LinftyL2-k+1}
\|\hat e_{x,h}\|_{L^\infty(0,T;L^2(\Gamma_h[\x^*]))}\le Ch^{k+1}
\end{align}
is not enough to guarantee the $W^{1,\infty}$ boundedness of $\hat e_{x,h}$ in the case $k=1$ (i.e., for piecewise linear finite elements) --- we actually need $k>1$ (quadratic or higher-order finite elements) to have some extra convergence rate in order to prove the $W^{1,\infty}$ boundedness of $\hat e_{x,h}$ from \eqref{ex-LinftyL2-k+1}; see the arguments in \cite{KLL19,KLL-2021,HL22}. 

This difficulty can be overcome if we can prove the following super-convergence rate in $L^2(0,T;H^{1})$ norm: 
\begin{align}\label{ev-L2H1}
& 
\|\hat e_{v,h}\|_{L^2(0,T;H^{1}(\Gamma_h[\x^*]))} \le Ch^{k+1} ,
\end{align}
which would imply the following result (position is the time integral of velocity): 
\begin{align}\label{ex-LinftyH1}
& 
 \|\hat e_{x,h}\|_{L^\infty(0,T;H^{1}(\Gamma_h[\x^*]))} 
\le Ch^{k+1} .
\end{align}
This would further imply (via inverse inequality in two dimensions)
\begin{align}\label{ex-LinftyW1infty-h}
& 
\|\hat e_{x,h}\|_{L^\infty(0,T;W^{1,\infty}(\Gamma_h[\x^*]))} \le Ch^k . 
\end{align}
This can be used to prove the uniform $W^{1,\infty}$ boundedness of $\hat e_{x,h}$ for $k\ge 1$. 
However, the main difficulty of this approach is that \eqref{ev-L2H1} cannot be shown with the presence of remainder $d_v$ in \eqref{interpolated-v}. To overcome this difficulty,  we {\it redefine} the error functions using a modified Ritz projection of the mean curvature flow that could exclude the remainder term $d_v$ in \eqref{interpolated-v}.

\end{enumerate}

Our idea is to define a dynamic Ritz projection of $(X,v,H,n)$ as the finite element solution $(Y_h^*, v_h^*, H_h^*, n_h^*)$ of a nonlinearly coupled surface evolution problem in \eqref{def-Ritz}, with $Y_h^*(\cdot,t):\Gamma_h^0\rightarrow\R^3$ being a finite element flow map with nodal vector $\y^*=\y^*(t)$, and $v_h^*$, $H_h^*$ and $n_h^*$ being functions on $\Gamma_h[\y^*]$ determined by the equations:   
{\small
\begin{align}\label{def-Ritz} 
\partial_t Y_h^*&= v_h^*\circ Y_h^* &&\mbox{on}\,\,\,\Gamma_h^0 \\[5pt]
v_h^* &= -I_h^*( H_h^* n_h^*) &&\mbox{on}\,\,\,\Gamma_h[\y^*] \notag\\ 
\int_{\Gamma_h[\y^*]}
(H_h^*\varphi_h^*+\nabla_{\Gamma_h[\y^*]} H_h^*\cdot \nabla_{\Gamma_h[\y^*]} \varphi_h^*)  &= 
\int_{\Gamma}
(H\varphi_h^l+\nabla_{\Gamma} H \cdot \nabla_{\Gamma} \varphi_h^l)  
&&\forall\,\varphi_h\in S_h(\Gamma_h[\x^*]) \notag\\[-2pt]
\int_{\Gamma_h[\y^*]}
(n_h^*\cdot\phi_h^*+\nabla_{\Gamma_h[\y^*]} n_h^*\cdot \nabla_{\Gamma_h[\y^*]} \phi_h^*) &=
\int_{\Gamma}
(n\cdot\phi_h^l+\nabla_{\Gamma} n \cdot \nabla_{\Gamma}  \phi_h^l)  
&&\forall\,\phi_h\in S_h(\Gamma_h[\x^*])^3 , 
\notag 
\end{align}
\!\!}
where $I_h^*$ denotes the Lagrange interpolation operator onto $S_h(\Gamma_h[\y^*]) $, and $\varphi_h^l$ denotes the lift of a function $\varphi_h\in S_h(\Gamma_h[\x^*])$ to surface $\Gamma=\Gamma(t)$, i.e., $\varphi_h^l(x^l) = \varphi_h(x)$ for $x\in \Gamma_h[\x^*]$, with $x^l$ denoting the lift of point $x$ from $\Gamma_h[\x^*]$ to $\Gamma$, and $\varphi_h^*$ denotes the finite element function on $\Gamma_h[\y^*]$ with the same nodal vector as $\varphi_h\in S_h(\Gamma_h[\x^*])$. {The initial value for the system \eqref{def-Ritz} is given by $Y_h^*(\cdot, 0) = \rm id$ on $\Gamma_h^0$.}
In this definition, $H_h^*$ and $n_h^*$ are Ritz projections of $H$ and $n$ onto an unknown surface $\Gamma_h[\y^*]$ which evolves with velocity $v_h^* = - I_h^*( H_h^* n_h^*)$ determined by this Ritz projection. 

{Note that the idea of using Ritz projection to achieve $H^1$ superconvergence and subsequently obtain optimal $L^2$ error estimates was first introduced in \cite{MaryPriori} for a class of nonlinear parabolic equations. Nonlinear types of Ritz projections were employed to ensure uniform control over the gradient of the height function for the mean curvature flow, Willmore flow and surface diffusion of graphs; see \cite{2000_DECKELNICK,Deckelnick-Dziuk-2006,Deckelnick-Styles-2022}. The dynamic Ritz projection introduced in this paper distinguishes itself from the classical Ritz projection primarily through modifications to the first two equations in \eqref{def-Ritz}. These alterations enable the surface $\Gamma_h[\y^*]$ to evolve according to an evolution equation (thus earning the name ``dynamic"), thereby differentiating it from the interpolated surface $\Gamma_h[\x^*]$.}

The first main result of this paper is the following theorem about optimal-order approximation properties for this dynamic Ritz projection of mean curvature flow. 

\begin{theorem}\label{THM1}
{\it 
Assume that the exact solution $(X, v, H, n)$ of the mean curvature flow is sufficiently smooth and the flow map $X(\cdot,t): \Gamma^0\to \Gamma(t)$ is a diffeomorphism for $t\in [0,T]$. Let $(\hat y_h^*,\hat v_h^*,\hat H_h^*,\hat n_h^*)$ be the lift of the dynamic Ritz projection $(Y_h^*, v_h^*, H_h^*, n_h^*)$ defined in \eqref{def-Ritz}, with finite elements of degree $k\ge 1$. 
Then there exists a constant $h_0>0$ such that the following error bound holds for mesh size $h\le h_0$: 
\begin{align}\label{Error-of-Ritz}
&\|\hat y_h^* - {\rm id}_{\Gamma_h[\x^*]} \|_{{{L^\infty(0,T;L^2(\Gamma_h[\x^*]))}}} 
+ \| \hat v_h^* - \hat I_h^*v \|_{{{L^\infty(0,T;L^2(\Gamma_h[\x^*]))}}} \notag\\
&\quad\, 
+ \| \hat H_h^* - \hat I_h^*H \|_{{{L^\infty(0,T;L^2(\Gamma_h[\x^*]))}}}
+ \| \hat n_h^* - \hat I_h^*n \|_{{{L^\infty(0,T;L^2(\Gamma_h[\x^*]))}}} \notag\\
&\quad\, + \| \partial^\bullet_{t,h}(\hat H_h^* - \hat I_h^*H) \|_{{{L^\infty(0,T;L^2(\Gamma_h[\x^*]))}}} 
+ \| \partial^\bullet_{{t,h}}(\hat n_h^* - \hat I_h^*n) \|_{{{L^\infty(0,T;L^2(\Gamma_h[\x^*]))}}} 
\le Ch^{k+1} ,
\end{align}
{where $\hat I_h^*X$ and $\hat I_h^*v$ denote the Lagrange interpolations of $X$ and $v$ onto $\Gamma_h[\x^0]$ and $\Gamma_h[\x^*]$, respectively.} The constant $C$ is independent of $h$ and $t\in[0,T]$ (though it may depend on $T$). 
}
\end{theorem}

In view of the results in Theorem \ref{THM1}, we compare the numerical solution $(X_h, v_h,H_h,n_h)$ with the dynamic Ritz projection $(Y_h^*, v_h^*,H_h^*,n_h^*)$, which satisfies \eqref{weak-form} up to some remainders $d_H^* \in S_h[\bf y^*]$ and $d_n^* \in S^3_h[\bf y^*]$, i.e., 
{\small
\begin{subequations}\label{weak-Ritz}
\begin{align}
\label{weak-Ritz-X}
\partial_t Y_h^* &= v_h^*\circ Y_h^* \hspace{18pt} &&\mbox{on}\,\,\,\Gamma_h^0 \\[5pt]
\label{weak-Ritz-v}
v_h^* &= -I_h^*( H_h^* n_h^*) &&\mbox{on}\,\,\,\Gamma_h[\y^*] \\ 
\int_{\Gamma_h[\y^*]} \partial_{t,h}^\bullet H_h^* \chi_H 
&+ \int_{\Gamma_h[\y^*]} \nabla_{\Gamma_h[\y^*]} H_h^* \cdot \nabla_{\Gamma_h[\y^*]} \chi_H && \forall \, \chi_H \in S_h(\Gamma_h[\y^*])\notag\\
&
= \int_{\Gamma_h[\y^*]} |\nabla_{\Gamma_h[\y^*]}n_h^* |^2 H_h^* \chi_H 
+ \int_{\Gamma_h[\y^*]} d_H^* \chi_H 
\label{weak-Ritz-H} \\
\int_{\Gamma_h[\y^*]} \partial_{t,h}^\bullet n_h^* \cdot \chi_n
&+\int_{\Gamma_h[\y^*]} \nabla_{\Gamma_h[\y^*]} n_h^* \cdot \nabla_{\Gamma_h[\y^*]} \chi_n && \forall \, \chi_n \in S_h^3(\Gamma_h[\y^*])\notag\\
&
= \int_{\Gamma_h[\y^*]}  |\nabla_{\Gamma_h[\y^*]}n_h^* |^2 n_h^* \cdot \chi_n 
+ \int_{\Gamma_h[\y^*]} d_n^* \cdot \chi_n.
\label{weak-Ritz-n} 
\end{align}
\end{subequations}
\hspace{-3.5pt}}
By introducing $u_h^* :=(n_h^*,H_h^*)$, the weak formulation in \eqref{weak-Ritz} can be rewritten as follows:
{\small
\begin{subequations}\label{weak-Ritz-u-version}
\begin{align}
\label{weak-Ritz-X-u}
\partial_t Y_h^* &= v_h^*\circ Y_h^* \hspace{18pt} &&\mbox{on}\,\,\,\Gamma_h^0 \\[5pt]
\label{weak-Ritz-v-u}
v_h^* &= -I_h^*( H_h^* n_h^*) &&\mbox{on}\,\,\,\Gamma_h[\y^*] \\ 
\int_{\Gamma_h[\y^*]} \partial_{t,h}^\bullet u_h^* \cdot \chi_u
&+ \int_{\Gamma_h[\y^*]} \nabla_{\Gamma_h[\y^*]} u_h^* \cdot \nabla_{\Gamma_h[\y^*]} \chi_u && \forall \, \chi_u \in S_h^4(\Gamma_h[\y^*])\notag\\
&
= \int_{\Gamma_h[\y^*]} |\nabla_{\Gamma_h[\y^*]}n_h^* |^2 u_h^* {{\cdot}}\chi_u
+ \int_{\Gamma_h[\y^*]} d_u^* {{\cdot}} \chi_u 
\label{weak-Ritz-u}.
\end{align}
\end{subequations}
\hspace{-3.5pt}}
where $d_u^*$ is remainder which satisfies the following estimate (this can be shown by using the result of Theorem \ref{THM1}; see Lemma \ref{Lemma:defects}): 
\begin{align*}
&\Big| \int_{\Gamma_h[\y^*]} d_u^* \cdot\chi_u  \Big |   
\le Ch^{k+1}  \|\chi_u \|_{H^1(\Gamma_h[\y^*])} .
\end{align*} 
In particular, compared with \eqref{interpolated-v}, no remainder appears in equation \eqref{weak-Ritz-v-u}. This makes it possible to prove the estimates in \eqref{ev-L2H1}--\eqref{ex-LinftyH1} by redefining the error functions as follows:
\begin{align}\label{error-functions}
\hat e_{x,h} := \hat x_h - \hat y_h^*, 
\quad
\hat e_{v,h} := \hat v_h - \hat v_h^*,
\quad
\hat e_{H,h} := \hat H_h - \hat H_h^* 
\quad\mbox{and}\quad
\hat e_{n,h} := \hat n_h - \hat n_h^* ,
\end{align} 
where $(\hat x_h,\hat v_h,\hat H_h,\hat n_h)$ and  $(\hat y_h^*,\hat v_h^*,\hat H_h^*,\hat n_h^*)$ are the lifts of $(X_h,v_h,H_h,n_h)$ and $(Y_h^*, v_h^*, H_h^*, n_h^*)$, respectively, i.e., the numerical solution of mean curvature flow defined in \eqref{weak-form} and the dynamic Ritz projection of mean curvature flow defined in \eqref{def-Ritz}, respectively. This leads to second main result of this paper, which is presented in the following theorem. 

\begin{theorem}\label{THM2} 
{\it 
Under the assumptions of Theorem {\rm\ref{THM1}}, the following error bound holds: 
\begin{align}\label{THM2-result} 
&\| \hat e_{x,h} \|_{L^\infty(0,T;H^1(\Gamma_h[\x^*]))}
+ \|\hat e_{H,h} \|_{L^\infty(0,T;L^2(\Gamma_h[\x^*]))}
+ \|\hat e_{n,h} \|_{L^\infty(0,T;L^2(\Gamma_h[\x^*]))} \notag\\
&\quad\, 
+ \|\hat e_{v,h} \|_{L^2(0,T;H^1(\Gamma_h[\x^*]))} 
+ \|\hat e_{H,h} \|_{L^2(0,T;H^1(\Gamma_h[\x^*]))} 
+ \|\hat e_{n,h} \|_{L^2(0,T;H^1(\Gamma_h[\x^*]))} \notag \\
&\le Ch^{k+1} ,
\end{align}
where $ C$ is a constant independent of $h$ and $t\in[0,T]$, but may depend on $T$.
}
\end{theorem}

The error estimates in Theorem \ref{THM1} and Theorem \ref{THM2}, and the estimates of the Lagrange interpolation error, lead to the third main result of this paper, i.e., optimal-order convergence of FEMs for mean curvature flow in the $L^\infty(0,T;L^2)$ norm with finite elements of degree $k\ge 1$. Before stating this result, we list some basic notations for finite element functions on $\Gamma_h[\x]$, $\Gamma_h[\x^*]$ and $\Gamma$. 

For any given finite element function $w_h\in S_h(\Gamma_h[\x])$ on the approximate surface $\Gamma_h[\x]$, 
we denote its nodal vector by $\w$, which collects all the values of $w_h$ at the nodes of $\Gamma_h[\x]$. The finite element function on the interpolated surface $\Gamma_h[\x^*]$ with the same nodal vector $\w$ is denoted by $\hat w_h$. The function $\hat w_h$ can be further lifted to $\Gamma[X]$ as $(\hat w_h)^l$, see details in Section \ref{section:lifts}. 
The lift from $S_h(\Gamma_h[\x])$ to $\Gamma[X]$ is denoted by $w_h^L = (\hat w_h)^l$. 
\begin{theorem}\label{THM3} 
  {\it 
Under the assumptions of Theorem {\rm\ref{THM1}}, the numerical solution of mean curvature flow defined in \eqref{weak-form}, with finite elements of degree $k\ge 1$, satisfies the following error bound{\rm:}
  \begin{align}
  &\|\hat x_h - {\rm id}_{\Gamma_h[\x^*]} \|_{{{L^\infty(0,T;L^2(\Gamma_h[\x^*]))}}}
+ \| \hat v_h - \hat I_h^*v \|_{{{L^\infty(0,T;L^2(\Gamma_h[\x^*]))}}} \notag\\
&\quad\, 
+ \| \hat H_h - \hat I_h^*H \|_{{{L^\infty(0,T;L^2(\Gamma_h[\x^*]))}}}
+ \| \hat n_h - \hat I_h^*n \|_{{{L^\infty(0,T;L^2(\Gamma_h[\x^*]))}}} \le Ch^{k+1} , \label{THM3-result} \\[5pt]  
  &\| X_h^L - {\rm id} \|_{L^\infty(0,T;L^2(\Gamma[X]))}
  + \|H_h^L - H \|_{L^\infty(0,T;L^2(\Gamma[X]))} \notag\\
  &\quad
  + \|n_h^L - n \|_{L^\infty(0,T;L^2(\Gamma[X]))} + \|v_h^L - v \|_{L^2(0,T;L^2(\Gamma[X]))} 
  \le Ch^{k+1} , \label{THM3-result2} 
  \end{align}
  where $ C$ is a constant independent of $h$ and $t\in[0,T]$ (but may depend on $T$).
  }
\end{theorem}

The rest of this paper is devoted to the proofs of Theorem \ref{THM1} and Theorem \ref{THM2}. The proof of Theorem \ref{THM3} is standard and therefore omitted (in fact, \eqref{THM3-result} follows from Theorem \ref{THM1} and Theorem \ref{THM2} with the application of the triangle inequality, and \eqref{THM3-result2} follows from an additional estimate for the interpolation errors). {The appendices in the supplementary material provide essential results and detailed proofs of Lemmas \ref{Lemma:L2-dtH}--\ref{Lemma:defects}, which are used for proving Theorems \ref{THM1} and \ref{THM2}. These proofs, following similar techniques as those in Lemmas \ref{Lemma:W1p}--\ref{Lemma:L2}, have been omitted from the main paper.}

\section{Proof of Theorem \ref{THM1}}\label{Proof-THM1}

\subsection{Lifts}\label{section:lifts} 

Throughout this article, we denote by $C$ and $h_0$ two generic positive constants which are different at different occurrences, possibly depending on the norms of the exact solution and $T$, but are independent of mesh size $h$. 

Given a smooth surface $\Gamma\subset \mathbb{R}^3$, the surface tangential gradient of a scalar function $u: \Gamma\to \mathbb{R}$ is a column vector denoted by $ \nabla_ \Gamma u$.
For a vector-valued function $u= (u_1,u_2,u_3)^\top: \Gamma\to \mathbb{R}^3$, we define 
$$
\nabla_ \Gamma u 
:= (\nabla_ \Gamma u_1,\nabla_ \Gamma u_2,\nabla_ \Gamma u_3) . 
$$
We denote by ${\rm id}$ the identity function on $\R^3$, i.e., ${\rm id}(x)=x$ for $x\in\R^3$. Its domain of definition can be restricted to any surface in $\R^3$. 

The finite element basis functions on $\Gamma_h[\x]$ are denoted by $\phi_j[\x]$, $j=1,\dots,\dof $, which are polynomials of degree $k$ after being pulled back to the reference plane triangle, and satisfy the following identities: 
\begin{equation*}
	\phi_j[\x](x_i) = \delta_{ij}, \quad \ i,j = 1,  \dotsc, \dof .
\end{equation*}
This definition of basis functions implies the following transport property (see \cite{Demlow-2009}):
\begin{align}\label{transport-phi}
\partial_{t,h}^{\bullet} \phi_j[\x(t)] = 0 \quad\mbox{on}\,\,\,\Gamma_h[\x(t)],\,\,\, j=1,\dots,N.
\end{align}
The finite element space on $\Gamma_h[\x]$ is defined as 
$
S_h(\Gamma_h[\x]) 
:= \big\{ \sum_{j=1}^N c_j \phi_j[\x]: 
	c_j\in \mathbb{R} \big\}.
$

From \cite[Lemma 7.1]{KLLP2017} or \cite[(2.15)-(2.16)]{Demlow-2009} we know that, there exists $h_0>0$ such that for $h\le h_0$ and $t\in [0,T]$, any point $x\in \Gamma_h[\x^*(t)]$ can be lifted to a unique point $x^l\in \Gamma(t)$ satisfying relation 
$$
x^l - x = \pm |x^l - x| n(x^l) .
$$
Then, any function $\varphi$ on $\Gamma_h[\x^*(t)]$ can be lifted to a function $\varphi^l$ on $\Gamma(t)$, defined as  
$$
\varphi^l(x^l)=\varphi(x) \quad\forall\, x\in  \Gamma_h[\x^*(t)]. 
$$
The lifted functions satisfy the following estimates uniformly for $h$ and $t$:
\begin{align} \label{LiftBound}
\begin{aligned}
C^{-1}\| \phi\|_{L^{2}\left(\Gamma_h[\x^*]\right)} 
& \leq\| \phi^{l}\|_{L^{2}(\Gamma[X])} 
\le C\| \phi\|_{L^{2}\left(\Gamma_h[\x^*]\right)} \\
C^{-1}\| \nabla_{\Gamma_h[\x^*]}\varphi\|_{L^2\left(\Gamma_h[\x^*]\right)} & 
\leq\| \nabla_{\Gamma[X]} \varphi^{l}\|_{L^2(\Gamma[X])} 
\le C\| \nabla_{\Gamma_h[\x^*]} \varphi\|_{L^2\left(\Gamma_h[\x^*]\right)} 
\end{aligned}
\end{align}
which hold for all $ \phi \in L^2( \Gamma_h[\x^*])$ and $\varphi\in H^1( \Gamma_h[\x^*])$, respectively.

\subsection{Matrix-vector formulation}
\label{Sec_MvForm}

The matrix-vector notations of \cite{KLL-2021,KLL19,HL22} will be used in this paper. In particular, we define $\K(\x) = \M(\x)+\A(\x)$, with ${\bf M}(\x)\in \mathbb{R}^{N\times N}$ and ${\bf A}(\x )\in \mathbb{R}^{N\times N}$ denoting the mass matrix and stiffness matrix associated to finite element space $S_h(\Gamma_h[\x])$ on surface $\Gamma_h[\x]$, respectively, and define  
{\begin{align*}
{\bf M}^{[d]}(\x ) =   I_d \otimes {\bf M}(\x) 
\quad\mbox{and}\quad
{\bf A}^{[d]}(\x ) =  I_d \otimes {\bf A}(\x)  , 
\end{align*}}
where $I_d$ is the $d\times d$ identity matrix. We denote by $\v$, $\n$ and $\H$ the nodal vectors of $v_h$, $n_h$ and $H_h$, respectively, and denote by $\f_1(\x,\n)\in \R^{3\dof}$ and $\f_2(\x,\n,\H)\in \R^{\dof}$ the nonlinear terms associated to the right-hand sides of (\ref{n-Fem}) and (\ref{H-Fem}), respectively, defined by 
\begin{subequations}
\begin{align*}
\f_1(\x,\n)_{{j+(m-1)N}}
&=\int_{\Gamma_h[\x]} |\nabla_{\Gamma_h[\x]} n_h|^2 (n_h)_m \phi_j \\
\f_2(\x,\n,\H)_{j}
&=\int_{\Gamma_h[\x]} |\nabla_{\Gamma_h[\x]} n_h|^2 H_h \phi_j , 
\end{align*}
\end{subequations}
with $j = 1,\dots, N$ and $m = 1,2,3$, where $(n_h)_m$ denotes the $m$th component of $n_h\in \mathbb{R}^3$. 

By introducing $\u :=(\n,\H)^\top$, the spatially semidiscrete parametric surface FEM in \eqref{weak-form} can be rewritten into the matrix-vector form:
\begin{subequations} \label{matrix-vector-form}
\begin{align}
\dot \x &= \v \label{Mform-x}\\ 
\v &= -{\bf I}_h(\H\bullet\n) \label{Mform-v}\\ 
\M^{[4]}(\x) \dot \u + \A^{[4]}(\x) \u &= \f(\x,\u), 
\label{Mform-nH}
\end{align}
\end{subequations} 
where 
\begin{align} 
\label{Mf}
\f(\x,\u) = 
\begin{pmatrix} 
\f_1(\x,\n)\\
\f_2(\x,\n,\H)
\end{pmatrix}\in \mathbb{R}^{3N+N} , 
\end{align} 
and ${\bf I}_h(\H\bullet\n)$ denotes the nodal vector of the finite element function $I_h(H_hn_h)$. 

We denote by $\y^*$, $\v^*$ and $\u^*=((\n^*)^\top,(\H^*)^\top)^\top$ the nodal vectors of $Y_h^*$, $v_h^*$ and $u_h^*=((n_h^*)^\top,H_h^*)^\top$, respectively. The latter are defined in \eqref{def-Ritz}, which can be written into the following matrix-vector form:
\begin{subequations}\label{def-yHn}
\begin{align}\label{def-y*}
\dot\y^*&= - {\bf I}_h( {\bf H}^*\bullet {\bf n}^*) \\
{\bf K}^{[4]}(\y^*){\bf u}^*\cdot{\bm \varphi}&=
\int_{\Gamma}
(u \cdot \varphi_h^l+\nabla_{\Gamma} u \cdot \nabla_{\Gamma}  \varphi_h^l)  
&&\forall\,\varphi_h\in S_h(\Gamma_h[\x^*])^4 ,
\label{def-Rhu} 
\end{align}
\end{subequations}
where ${\bm\varphi}$ is the nodal vector of finite element function $\varphi_h$. The existence of $\y^*$ as a sufficiently good approximation to $\x^*$ will be proved. Then \eqref{weak-Ritz} can be written into the following matrix-vector form: 
\begin{align}\label{ritz-matrix-form}
\dot\y^*&= {\bf v}^* \\
{\bf v}^*&= -{\bf I}_h({\bf H}^*\bullet {\bf n}^* ) \\
{\bf M}^{[4]}(\y^*)\dot {\bf u}^*+{\bf A}^{[4]}(\y^*){\bf u}^* &= {\bf f}(\y^*,{\bf u}^*) + {\bf M}^{[4]}(\y^*) {\bf d}_{\bf u}^* ,
\end{align}
where ${\bf d}_{\bf u}^*$ denotes the nodal vector of the finite element function $d_u^* = ((d_n^*)^\top,d_H^*)^\top$, with $d_H^*$ and $d_n^*$ being the remainders defined in \eqref{weak-Ritz}. In the rest of this paper, we omit the superscripts in $\M^{[4]}(\x)$, $\A^{[4]}(\x)$ and $\K^{[4]}(\x)$ for the simplicity of notations. 

\subsection{Perturbation of mass matrix and stiffness matrix}
\label{section:def-hat-ey}

For $\e_\y = \y^*-\x^*$, which is the nodal vector of the finite element function $\hat e_y=\hat y_h^*-{\rm id}$ on $\Gamma_h[\x^*]$, we consider the following intermediate surfaces 
$$
\Gamma_h[\y^\theta] \quad\mbox{with}\quad 
\y^\theta=(1-\theta)\x^*+\theta\y^* = \x^* + \theta\e_\y \quad\mbox{for}\,\,\, \theta\in [0,1] .
$$
The finite element functions on $\Gamma_h[\y^\theta] $ with nodal vectors $\e_\y$, $\z$ and $\w$ are denoted by $\hat e_y^\theta$, $\hat z_h^\theta $ and  $\hat w_h^\theta $, respectively  (thus $\hat e_y^0=\hat e_y$). 
The following result was proved in \cite[Lemma 4.3]{KLLP2017} and \cite[Lemma 7.2]{KLL19}.

\begin{lemma}
\label{CBI}
If $\|  \nabla_{ \Gamma_h[\x^*]}\hat e_y\|_{L^\infty( \Gamma_h[\x^*])}\le 1/2$, 
then the following inequalities hold for $ \theta\in [0,1]$ and $1\le p\le \infty${\rm:}  
\begin{align}
\label{L2-equiv}
& \|\hat w_h^ \theta\|_{L^p( \Gamma_h[\y^\theta])} \le c_p \| \hat w_h^0\|_{L^ p( \Gamma_h[\x^*])},\\
\label{grad-equiv}
& \|  \nabla_{ \Gamma_h[\y^\theta]} \hat w_h^ \theta\|_{L^p( \Gamma_h[\y^\theta])} \le c_p \|  \nabla_{ \Gamma_h[\x^*]} \hat w_h^0\|_{L^p(\Gamma_h[\x^*])},
\end{align}
where $c_p$ is a constant independent of $ \theta$ and $h$, and $c_\infty = 2$. 
\end{lemma}

The following lemma was proved in \cite[ Lemma 7.1]{KLL19}. 
\begin{lemma}\label{BilinearMA} 
If $\|  \nabla_{ \Gamma_h[\x^*]}\hat e_y\|_{L^\infty( \Gamma_h[\x^*])}\le 1/2$, then the following relations hold{\rm:} 
\begin{align}
\label{M-Diff-0}
& (\M(\y^*)-\M(\x^*))\z \cdot \w = 
\int_0^1 \int_{ \Gamma_h[\y^\theta]} \hat w_h^ \theta (  \nabla_{ \Gamma_h[\y^\theta]}\cdot \hat e_y^ \theta) \hat z_h^ \theta,\\[-2pt]
\label{A-Diff-0}
& (\A(\y^*)-\A(\x^*))\z \cdot \w = 
\int_0^1 \int_{ \Gamma_h[\y^\theta]}  \nabla_{ \Gamma_h[\y^\theta]} \hat w_h^ \theta \cdot \big( D_{ \Gamma_h[\y^\theta]}\hat e_y^ \theta \big)  \nabla_{ \Gamma_h[\y^\theta]}\hat z_h^ \theta ,
\end{align}
where $D_{\Gamma_h[\y^\theta]}\hat e_y^{\theta}=\operatorname{tr}\big(E^{\theta}\big) I_{3}-\big(E^{\theta}+(E^{\theta})^{\top} \big)$ and $E^{\theta}=\nabla_{\Gamma_h[\y^\theta]} \hat e_y^{\theta} \in \mathbb{R}^{3 \times 3}$.
\end{lemma}

Lemma \ref{CBI} and Lemma \ref{BilinearMA} imply the following result: If $ \|  \nabla_{ \Gamma_h[\x^*]} \hat e_y\|_{ L^\infty( \Gamma_h[\x^*])}\le 1/2$ then $ \|  \nabla_{ \Gamma_h[\y^\theta]} \hat e_y^ \theta\|_{ L^\infty( \Gamma_h[\y^\theta])}\le 1$ for $ \theta\in [0,1]$ and therefore 
\begin{align}\label{equiv-norm}
&\begin{aligned}
&\textrm{the norms } \| \cdot\|_{\M(\x^*+ \theta \e_\y)} 
\textrm{ are $h$-uniformly equivalent for } \theta\in [0,1], \\[-2pt] 
&\textrm{the norms } \| \cdot\|_{\A(\x^*+ \theta \e_\y)} 
\textrm{ are $h$-uniformly equivalent for } \theta\in [0,1], \\ 
\end{aligned} \\
\label{M-Diff}
&\hspace{-5pt} (\M(\y^*)-\M(\x^*))\z \cdot \w \le 
C \| \hat w_h^0 \|_{L^p(\Gamma_h[\x^*])} \| \nabla_{\Gamma_h[\x^*]}\hat e_y^0 \|_{L^{p'}(\Gamma_h[\x^*])} \| \hat z_h^0\|_{L^\infty(\Gamma_h[\x^*])},\\
\label{A-Diff}
&\hspace{-5pt} (\A(\y^*)-\A(\x^*))\z \cdot \w \le 
C\| \hat w_h^0 \|_{W^{1,p}(\Gamma_h[\x^*])} \| \hat e_y^0 \|_{W^{1,p'}(\Gamma_h[\x^*])} \| \hat z_h^0\|_{W^{1,\infty}( \Gamma_h[\x^*])},\\
\label{A-Diff-2}
&\hspace{-5pt} (\A(\y^*)-\A(\x^*))\z \cdot \w \le 
C \| \hat e_y^0\|_{W^{1,\infty}(\Gamma_h[\x^*])} \| \hat z_h^0\|_{W^{1,p}( \Gamma_h[\x^*])} \| \hat w_h^0\|_{W^{1,p'}( \Gamma_h[\x^*])},
\end{align}
given that $p$ and $p'$ satisfy the relation $\frac{1}{p} + \frac{1}{p'} = 1$. 

In addition to Lemma \ref{CBI} and Lemma \ref{BilinearMA}, the following relation will be used: If $K\subset\Gamma$ or $K\subset\Gamma_h[\x^*]$ is a smooth piece of surface which evolves under the velocity field $w$, and $\partial_t^\bullet$ denotes the material derivative with respect to $w$, then   
\begin{align}\label{dt-grad-f}
\partial_t^\bullet \nabla_K f = \nabla_K \partial_t^\bullet f  - (\nabla_K w - n_Kn_K^\top(\nabla_K w)^\top)\nabla_K f ,
\end{align}
where $n_K$ denotes the unit normal vector of $K$.

\subsection{${\bf W^{1,p}}$ error estimates for the dynamic Ritz projection} 

Let $\hat y_h^*$, $\hat H_h^*$ and $\hat n_h^*$ be the finite element functions in $S_h(\Gamma_h[\x^*])$ with the nodal vectors $\y^*$, ${\bf H}^*$ and ${\bf n}^*$ defined in \eqref{def-yHn}, respectively. In this subsection we prove the following lemma. 

\begin{lemma}\label{Lemma:W1p}
{\it 
Under the assumptions of Theorem {\rm\ref{THM1}}, there exists $h_0>0$ such that for mesh size $h\le h_0$ there exists a unique solution $(\y^*,{\bf H}^*,{\bf n}^*)$ of \eqref{def-yHn} satisfying the following estimates for all $2\le p <\infty${\rm:} 
\begin{align}\label{Ritz-W1p}
&\|\hat y_h^* - {{\rm id}_{\Gamma_h[\x^*]}}\|_{W^{1,p}(\Gamma_h[\x^*])} 
+ \|\hat v_h^*-  \hat I_h^*v\|_{W^{1,p}(\Gamma_h[\x^*])}  \notag\\
&+ \|\hat H_h^*-  \hat I_h^*H\|_{W^{1,p}(\Gamma_h[\x^*])} 
+ \|\hat n_h^*-  \hat I_h^*n\|_{W^{1,p}(\Gamma_h[\x^*])}  \le C_p h^k .
\end{align}
}
\end{lemma}
\begin{proof}
Problem \eqref{def-yHn} is essentially a system of ordinary differential equations (ODEs). We assume that $t_*\in(0,T]$ is the maximal time such that the solution of problem \eqref{def-yHn} exists and satisfies the following estimates hold for $t\in[0,t_*]$: 
\begin{subequations}\label{ind-hat-yHn}
\begin{align}
&\|\hat y_h^*-{{{\rm id}_{\Gamma_h[\x^*]}}}\|_{W^{1,\infty}(\Gamma_h[\x^*])} \le 1/2, \label{ind-hat-y}\\
&\|\hat H_h^*-  \hat I_h^*H\|_{W^{1,\infty}(\Gamma_h[\x^*])} \le 1/2, \label{ind-hat-H}\\
&\|\hat n_h^*- \hat  I_h^*n\|_{W^{1,\infty}(\Gamma_h[\x^*])} \le 1/2 . \label{ind-hat-n}
\end{align}
\end{subequations} 
Under this condition, we shall prove that \eqref{Ritz-W1p} holds for $t\in[0,t_*]$ and all $2\le p< \infty$ (with some constants $h_0$ and $C_p$ that are independent of $t_*$). In particular, for $p=4$, the local-in-time existence and uniqueness of solutions to ODE system \eqref{def-Ritz} (and the continuity in time of solutions to the ODE system) guarantee that its solution extends to $t\in[0,t_*+\delta_h]$ for some $\delta_h>0$ and satisfies \eqref{Ritz-W1p} for $t\in[0,t_*+\delta_h]$ with $C_4$ replaced by $2C_4$. For sufficiently small $h$ (smaller than some constant independent of $t_*$) such that $2C_4h^{k-2/4} \le 1/2$, this implies that \eqref{ind-hat-yHn} holds for $t\in[0,t_*+\delta_h]$. This will prove that $t_*=T$ (otherwise $t_*\in(0,T]$ is not the maximal time for \eqref{ind-hat-yHn} to hold). 

{
Since \(u_h^*\) is defined on the surface \(\Gamma_h[\y^*]\) via \eqref{def-Rhu}, we must bridge the gap between the discrete surfaces \(\Gamma_h[\y^*]\) and \(\Gamma_h[\x^*]\) to estimate \(\|\hat u_h^* - \hat I_h^* u \|_{W^{1,p}(\Gamma_h[\x^*])}\). 
Under condition \eqref{ind-hat-yHn}, we can rewrite \eqref{def-Rhu} as
\begin{align}\label{def-varphi-modified}
{\bf K}(\x^*){\bf u}^*\cdot{\bm \varphi} = ({\bf K}(\x^*)-{\bf K}(\y^*)){\bf u}^*\cdot{\bm \varphi} + \int_{\Gamma} \Big( u\,\varphi_h^l + \nabla_{\Gamma} u \cdot \nabla_{\Gamma}\varphi_h^l \Big).
\end{align}
To characterize the gap term \(
({\bf K}(\x^*)-{\bf K}(\y^*)){\bf u}^*\cdot{\bm \varphi}
\), we define \(w\in H^1(\Gamma)^4\) as the solution of the following weak formulation:
\begin{align}\label{def-w}
\int_{\Gamma} \Big( w \cdot \varphi^l + \nabla_{\Gamma} w \cdot \nabla_{\Gamma}\varphi^l \Big)
= ({\bf K}(\x^*)-{\bf K}(\y^*)){\bf u}^*\cdot ({\bf P}_h\varphi) \quad\forall\,\varphi^l\in H^1(\Gamma)^4,
\end{align}
where \(\varphi\) denotes the inverse lift of \(\varphi^l\) onto \(\Gamma_h[\x^*]\) and \({\bf P}_h\varphi\) is the nodal vector of \(P_h\varphi\), i.e., the \(L^2\) projection of \(\varphi\in H^1(\Gamma_h[\x^*])^4\) onto \(S_h(\Gamma_h[\x^*])^4\).
}
 Since the $L^2$ projection operator $P_h$ is bounded in the $L^p$ norm for $1\le p\le \infty$ (see Appendix \ref{appendix_B} in the supplementary material) and the $L^p$ norms of $\varphi$ and $\varphi^l$ are equivalent for $1\le p\le \infty$, it follows that 
$$
\|P_h\varphi\|_{L^p(\Gamma_h[\x^*])}
\le C\|\varphi\|_{L^p(\Gamma_h[\x^*])}
\le C\|\varphi^l\|_{L^p(\Gamma)} 
\quad\forall\, 1\le p\le \infty. 
$$
Since ${\bf P}_h\varphi_h= {\bm \varphi}$ for all $\varphi_h\in S_h(\Gamma_h[\x^*])$,  it follows from \eqref{def-varphi-modified}--\eqref{def-w} that 
\begin{align}\label{KH.varphi}
{\bf K}(\x^*){\bf u}^*\cdot{\bm \varphi} =
\int_{\Gamma} ((w+u) \cdot \varphi_h^l+\nabla_{\Gamma} (w+u) \cdot \nabla_{\Gamma}  \varphi_h^l) 
\quad\forall\, \varphi_h\in S_h(\Gamma_h[\x^*]) .
\end{align}
This means that $\hat u_h^*$ is the linear Ritz projection of $w+u$ onto $S_h(\Gamma_h[\x^*])$. If we further define $\hat w_h^*$ as the linear Ritz projection of $w$ onto $S_h(\Gamma_h[\x^*])$, then $\hat u_h^*-\hat w_h^*$ is the linear Ritz projection of $u$ onto $S_h(\Gamma_h[\x^*])$. 
{\paragraph{Estimate for $\|w\|_{W^{1,p}(\Gamma)}$}
To derive an estimate for \(\|w\|_{W^{1,p}(\Gamma)}\), we consider the PDE problem \eqref{def-w} on the continuous surface, which can be reformulated as} 
\begin{align}\label{PDE-w-2}
 \int_{\Gamma} (w\varphi^l+\nabla_{\Gamma} w \cdot \nabla_{\Gamma}  \varphi^l) 
 = \ell(\varphi^l), 
\end{align}
where $\ell(\varphi^l)= ({\bf K}(\x^*)-{\bf K}(\y^*)){\bf u}^*\cdot {\bf P}_h \varphi $ 
is a linear functional on $\varphi^l$. Under condition \eqref{ind-hat-yHn}, the following estimate follows from inequalities \eqref{M-Diff}--\eqref{A-Diff}:
\begin{align*}
  |\ell(\varphi^l)| 
&\le
C\|\hat u_h^*\|_{W^{1,\infty}(\Gamma_h[\bf x^*])} \|{{{\rm id}_{\Gamma_h[\x^*]}}}-\hat y_h^*\|_{W^{1,p}(\Gamma_h[\x^*])}
 \|\varphi\|_{W^{1,p'}(\Gamma_h[\x^*])}
 \\
 &\le
C\|{{{\rm id}_{\Gamma_h[\x^*]}}}-\hat y_h^*\|_{W^{1,p}(\Gamma_h[\x^*])}
 \|\varphi^l\|_{W^{1,p'}(\Gamma)}. 
\end{align*}
This means that
$$
\|\ell\|_{W^{-1,p}(\Gamma)} := \|\ell\|_{(W^{1,p'}(\Gamma))'}
\le 
C\|{{{\rm id}_{\Gamma_h[\x^*]}}}-\hat y_h^*\|_{W^{1,p}(\Gamma_h[\x^*])} . 
$$
Then the standard $W^{1,p}$ estimates for the elliptic PDE problem in \eqref{PDE-w-2} implies that (cf. {{\cite[Theorem 1]{Meyers1963}}}, which extends to PDEs on surfaces via estimates on local {{coordinate}} charts)
\begin{align}\label{W1p-Hh-2}
\|w\|_{W^{1,p}(\Gamma)}
\le 
C\|\ell\|_{W^{-1,p}(\Gamma)}
\le 
C\|{{{\rm id}_{\Gamma_h[\x^*]}}}-\hat y_h^*\|_{W^{1,p}(\Gamma_h[\x^*])} 
\quad\mbox{for}\,\,\, 2\le p<\infty. 
\end{align}

{{\paragraph{Estimates for \(L^2\) and \(W^{1,p}\) norms of \((\hat u_h^*)^l - (\hat w_h^*)^l - u\)}
The following \(W^{1,p}\) and \(H^1\) estimates for the linear Ritz projection onto the interpolated surface \(\Gamma_h[\x^*]\) were established in \cite[Corollaries 4.2 and 4.5]{Demlow-2009}:}}
\begin{align}
\label{W1p-wh-H1} 
&\|(\hat w_h^*)^l-w\|_{L^{2}(\Gamma)} + h\|(\hat w_h^*)^l-w\|_{H^{1}(\Gamma)}
\le 
Ch^{k+1} \|w\|_{H^{k+1}(\Gamma)} , \\
&\|(\hat w_h^*)^l-w\|_{W^{1,\infty}(\Gamma)}
\le 
C\|w-(\hat I_h^*w)^l\|_{W^{1,\infty}(\Gamma)} +Ch^{k+1}|\ln h| \|w\|_{W^{1,\infty}(\Gamma)} \notag\\
&\hspace{92.5pt} \le C h^{k} \|w\|_{W^{k+1,\infty}(\Gamma)}. 
\label{W1p-wh-W1inf} 
\end{align} 
Since the complex interpolation spaces between $H^{k+1}(\Gamma)$ and $W^{k+1,\infty}(\Gamma)$ are  $W^{k+1,p}(\Gamma)$ for $2\le p\le \infty$ (cf. \cite[Theorem 6.4.5]{BL76}), the complex interpolation (cf. \cite[Theorem 4.1.2]{BL76}) of the above two estimates yields, for $2\le p\le \infty$,  
\begin{align}\label{W_1pW_2}
    \|(\hat w_h^*)^l - w\|_{W^{1,p}(\Gamma)} \le Ch^k \|w\|_{W^{k+1,p}(\Gamma)} . 
\end{align}
Moreover, the linear Ritz projection onto the interpolated surface is naturally stable in $H^1$ norm, i.e., 
\begin{align}\label{H1_stability}
    \|(\hat w_h^*)^l\|_{H^1(\Gamma)} \le C  \|w\|_{H^1(\Gamma)}.
\end{align}
By utilizing \eqref{W1p-wh-W1inf} and the $W^{1,\infty}$ stability of Lagrange interpolation, we derive the following result:
\begin{equation}\label{W1infty_stablity}
  \begin{aligned}
   \|(\hat w_h^*)^l\|_{W^{1,\infty}(\Gamma)}
   &\le \|(\hat w_h^*)^l - w\|_{W^{1,\infty}(\Gamma)} + \|w\|_{W^{1,\infty}(\Gamma)}\\
   &\le C\|w-(\hat I_h^*w)^l\|_{W^{1,\infty}(\Gamma)} +C \|w\|_{W^{1,\infty}(\Gamma)} \le C  \|w\|_{W^{1,\infty}(\Gamma)}. 
  \end{aligned}
\end{equation}
The complex interpolation between \eqref{H1_stability} and \eqref{W1infty_stablity} yields the following result (cf. \cite[Theorem 4.1.2]{BL76}) for $2\le p \le \infty$: 
\begin{equation}\label{W1p_stability}
  \begin{aligned}
    \|(\hat w_h^*)^l\|_{W^{1,p}(\Gamma)} \le C \|w\|_{W^{1,p}(\Gamma)} . 
  \end{aligned}
\end{equation}
Since $\hat u_h^*-\hat w_h^*$ is the linear Ritz projection of $u$ onto $S_h(\Gamma_h[\x^*])$, replacing $w$ by $u$ in \eqref{W1p-wh-H1} and \eqref{W_1pW_2} yields that 
\begin{align}
&\|(\hat u_h^*)^l - (\hat w_h^*)^l-u\|_{L^{2}(\Gamma)}
+h\|(\hat u_h^*)^l - (\hat w_h^*)^l-u\|_{H^{1}(\Gamma)}
\le  
C\|u\|_{H^{k+1}(\Gamma)} h^{k+1}  , \label{H-w-H-L2}  \\
&\|(\hat u_h^*)^l-(\hat w_h^*)^l-u\|_{W^{1,p}(\Gamma)}
\le 
C\|u\|_{W^{k+1,p}(\Gamma)} h^{k}  \quad \mbox{for}\,\,\, 2\le p \le  \infty .  
\label{H-w-H-W1p-1}
\end{align}
{{\paragraph{Estimates for \(L^2\) and \(W^{1,p}\) norms of \((\hat u_h^*)^l  - u\)}
As a result of \eqref{W1p-wh-H1} and \eqref{H-w-H-L2}, by using the triangle inequality we have }}
\begin{align}
\|(\hat u_h^*)^l-u\|_{L^2(\Gamma)} 
&\le 
\|(\hat w_h^*)^l\|_{L^2(\Gamma)}+
C\|u\|_{H^{k+1}(\Gamma)} h^{k+1} \notag \\
&\le 
\|w\|_{L^2(\Gamma)}+\|(\hat w_h^*)^l-w\|_{L^2(\Gamma)}
+C \|u\|_{H^{k+1}(\Gamma)}  h^{k+1} \notag \\
&\le 
\|w\|_{L^2(\Gamma)} + Ch^2\|w\|_{H^2(\Gamma)}
+C \|u\|_{H^{k+1}(\Gamma)} h^{k+1} , \label{H-H-L2-1} \\[5pt] 
\|(\hat u_h^*)^l-u\|_{W^{1,p}(\Gamma)}
&\le 
\|(\hat w_h^*)^l\|_{W^{1,p}(\Gamma)}+
C\|u\|_{W^{k+1,p}(\Gamma)} h^k  . \label{H-H-W1p-11}
\end{align} 
Then, substituting inequality \eqref{W1p_stability} into \eqref{H-H-W1p-11}, we obtain  
\begin{align}
\label{W1p-Hh-1}
&\|(\hat u_h^*)^l-u\|_{W^{1,p}(\Gamma)}
\le 
C\|w\|_{W^{1,p}(\Gamma)}  +
C\|u\|_{W^{k+1,p}(\Gamma)} h^k . 
\end{align}

From \eqref{W1p-Hh-1} and \eqref{W1p-Hh-2} we obtain the following result for $2\le p<\infty$: 
\begin{align}\label{W1p-Hh-H}
\|(\hat u_h^*)^l-u\|_{W^{1,p}(\Gamma)}
\le 
C\|{{{\rm id}_{\Gamma_h[\x^*]}}}-\hat y_h^*\|_{W^{1,p}(\Gamma_h[\x^*])} 
+ C \|u\|_{W^{k+1,p}(\Gamma)} h^k . 
\end{align}

In order to establish an estimate for $\|{{{\rm id}_{\Gamma_h[\x^*]}}}-\hat y_h^*\|_{W^{1,p}(\Gamma_h[\x^*])} $ on the right-hand side of \eqref{W1p-Hh-H}, we consider the flow maps $X_h^*:\Gamma_h[\x^0]\rightarrow \Gamma_h[\x^*]$ and $Y_h^*=\hat y_h^*\circ X_h^*$, which satisfy the following relation: 
$$
\frac{\d}{\d t}(X_h^* - Y_h^*) = - \hat I_h^*(\hat I_h^*H\hat I_h^*n-\hat H_h^*\hat n_h^*)\circ X_h^* . 
$$
which can be written into the integral form:
\begin{align}\label{X-Y-expr}
X_h^*(s)  - Y_h^*(s) = - \int_0^s \hat I_h^*(\hat I_h^*H \hat I_h^*n-\hat H_h^*\hat n_h^*)\circ X_h^* \, \d t .
\end{align}
Then, applying gradient to \eqref{X-Y-expr} and using the chain rule of partial differentiation, we have 
$$
\nabla_{\Gamma_h^0}(X_h^*(s)  - Y_h^*(s) )
= -\int_0^s \nabla_{\Gamma_h^0} X_h^* \, [\nabla_{\Gamma_h[\x^*]} \hat I_h^*(\hat I_h^*H\hat I_h^*n-\hat H_h^*\hat n_h^*)]\circ X_h^* \,\, \d t .
$$
Then, by considering the $L^p$ norm of both sides of this relation, we obtain the following result for $s\in(0,t_*]$:  
\begin{align}
&\|\nabla_{\Gamma_h^0}(X_h^*(s)  - Y_h^*(s)  )\|_{L^p(\Gamma_h^0)} \notag \\ 
&\le
C\int_0^s \|[\nabla_{\Gamma_h[\x^*]} \hat I_h^*(\hat I_h^*H\hat I_h^*n-\hat H_h^*\hat n_h^*)]\circ X_h^*\|_{L^p(\Gamma_h^0)} \d t 
\quad\mbox{(since $\|X_h^*\|_{W^{1,\infty}(\Gamma_h^0)}$ is bounded)} \notag \\ 
&\le
C\int_0^s (\| \hat I_h^*H-\hat H_h^*\|_{W^{1,p}(\Gamma_h[\x^*])} \| \hat I_h^*n \|_{W^{1,\infty}(\Gamma_h[\x^*])} + \| \hat H_h^*\|_{W^{1,\infty}(\Gamma_h[\x^*])}\|\hat I_h^*n-\hat n_h^*\|_{W^{1,p}(\Gamma_h[\x^*])}) \d t \notag \\
&\hspace{30pt} \mbox{($W^{1,p}$ stability of $\hat I_h^*$ is used; see Appendix \ref{appendix_A} in the supplementary material)} \notag \\ 
&\le
C\int_0^s (\| \hat I_h^*H-\hat H_h^*\|_{W^{1,p}(\Gamma_h[\x^*])} + \|\hat I_h^*n-\hat n_h^*\|_{W^{1,p}(\Gamma_h[\x^*])}) \d t ,
\end{align}
where the last inequality uses the boundedness of $\| \hat H_h^*\|_{W^{1,\infty}(\Gamma_h[\x^*])} $, which follows from \eqref{ind-hat-yHn}. Similarly, by considering the $L^p$ norm of \eqref{X-Y-expr} directly, we can obtain the following result: 
\begin{align}\label{id-yh-Lp}
\|X_h^*(s) - Y_h^*(s)\|_{L^p(\Gamma_h^0)} 
&\le
C\int_0^s (\| \hat I_h^*H-\hat H_h^*\|_{L^p(\Gamma_h[\x^*])} + \|\hat I_h^*n-\hat n_h^*\|_{L^p(\Gamma_h[\x^*])}) \d t .
\end{align}
Since ${{{\rm id}_{\Gamma_h[\x^*]}}}-\hat y_h^*= (X_h^* - Y_h^*)\circ (X_h^*)^{-1}$ and therefore $\|{{{\rm id}_{\Gamma_h[\x^*]}}}-\hat y_h^*\|_{W^{1,p}(\Gamma[\x^*(s)])} \sim \|X_h^*(s) - Y_h^*(s)\|_{W^{1,p}(\Gamma_h^0)} $, the last two estimates and Lagrange interpolation error estimates imply that 
\begin{align}\label{id-yh-W1p-expr}
\|{{{\rm id}_{\Gamma_h[\x^*(s)]}}}-\hat y_h^*\|_{W^{1,p}(\Gamma[\x^*(s)])} 
&\le
C\int_0^s (\|\hat I_h^*H-\hat H_h^*\|_{W^{1,p}(\Gamma_h[\x^*])} + \| \hat I_h^*n-\hat n_h^*\|_{W^{1,p}(\Gamma_h[\x^*])}) \d t \notag \\
&\le
C\int_0^s (\| H-(\hat H_h^*)^l\|_{W^{1,p}(\Gamma)} + \| n-(\hat n_h^*)^l\|_{W^{1,p}(\Gamma)}) \d t 
+Ch^k \notag\\
&\le
C\int_0^s \| u-(\hat u_h^*)^l\|_{W^{1,p}(\Gamma)} \d t 
+Ch^k \quad\mbox{for}\,\,\, s\in(0,t_*] . 
\end{align}
Substituting this into \eqref{W1p-Hh-H} and using Gronwall's inequality, we obtain 
\begin{align*}
\|{{{\rm id}_{\Gamma_h[\x^*]}}}-\hat y_h^*\|_{W^{1,p}(\Gamma_h[\x^*])} 
+\|(\hat u_h^*)^l-u\|_{W^{1,p}(\Gamma)}
\le 
Ch^k
\quad\mbox{for}\,\,\, 2\le p<\infty.  
\end{align*}
Since $\|(\hat I_h^*u)^l - (\hat u_h^*)^l\|_{W^{1,p}(\Gamma)} \le \|(\hat I_h^*u)^l - u\|_{W^{1,p}(\Gamma)} + \|(\hat u_h^*)^l-u\|_{W^{1,p}(\Gamma)} \le Ch^k$ and 
$\|(\hat I_h^*u)^l - (\hat u_h^*)^l\|_{W^{1,p}(\Gamma)} \sim \|\hat I_h^*u - \hat u_h^*\|_{W^{1,p}(\Gamma_h[\x^*])}$, it follows that 
\begin{align}\label{W1p-Hh-H-nh-n}
\|{{{\rm id}_{\Gamma_h[\x^*]}}}-\hat y_h^*\|_{W^{1,p}(\Gamma_h[\x^*])} + \|\hat I_h^*u - \hat u_h^*\|_{W^{1,p}(\Gamma_h[\x^*])}
\le 
Ch^k
\quad\mbox{for}\,\,\, 2\le p<\infty.  
\end{align}
Moreover, we can express $\hat I_h^*v - \hat v_h^* $ as 
\begin{align}
\hat I_h^*v - \hat v_h^* 
& = - \hat I_h^*(\hat I_h^*H \hat I_h^*n) + \hat I_h^*(\hat H_h^*\hat n_h^*) \notag\\
& = - \hat I_h^*[(\hat I_h^*H - \hat H_h^*) \hat I_h^*n] - \hat I_h^*[\hat H_h^* (\hat I_h^*n - \hat n_h^*)] . 
\end{align}
and apply the $W^{1,p}$ stability of Lagrange interpolation (as shown in Appendix \ref{appendix_A} in the supplementary material). This leads to the following result for $t\in[0,t_*]$: 
\begin{align}
\| \hat I_h^*v - \hat v_h^* \|_{W^{1,p}(\Gamma_h[\x^*])}
& \le C\| \hat I_h^*H - \hat H_h^*\|_{W^{1,p}(\Gamma_h[\x^*])} \| \hat I_h^*n\|_{W^{1,\infty}(\Gamma_h[\x^*])}  \notag\\
&\quad\, 
+ C\|\hat H_h^*\|_{W^{1,\infty}(\Gamma_h[\x^*])} \| \hat I_h^*n - \hat n_h^*\|_{W^{1,p}(\Gamma_h[\x^*])} \notag\\
& \le C\| \hat I_h^*u - \hat u_h^*\|_{W^{1,p}(\Gamma_h[\x^*])} 
\le Ch^k , 
\end{align}
where the boundedness of $\|\hat H_h^*\|_{W^{1,\infty}(\Gamma_h[\x^*])}$ follows from \eqref{ind-hat-yHn}, and the last inequality follows from \eqref{W1p-Hh-H-nh-n}. 
This proves Lemma \ref{Lemma:W1p} according to the discussions in the text below \eqref{ind-hat-yHn}. 
\hfill\end{proof}

\subsection{${\bf L^2}$ error estimates for the dynamic Ritz projection} 
$\,$

\begin{lemma}\label{Lemma:L2}
{\it  
Under the assumptions of Theorem {\rm\ref{THM1}}, there exists $h_0>0$ such that, for mesh size $h\le h_0$, \eqref{def-yHn} has a unique solution $(\y^*,{\bf H}^*,{\bf n}^*)$, which satisfies the following estimate{\rm:} 
\begin{align}
&\|\hat y_h^*-{{{\rm id}_{\Gamma_h[\x^*]}}}\|_{L^2(\Gamma_h[\x^*])} 
+ \|\hat H_h^*- \hat I_h^*H\|_{L^2(\Gamma_h[\x^*])} 
+ \|\hat n_h^*- \hat I_h^*n\|_{L^2(\Gamma_h[\x^*])}  \le C h^{k+1} . \label{L2-y*} 
\end{align}
}
\end{lemma}
\begin{proof}
From \eqref{H-H-L2-1} we see that, in order to estimate $\|(\hat u_h^*)^l-u\|_{L^2(\Gamma)} $, it suffices to  estimate $\|w\|_{L^2(\Gamma)}$ and $\|w\|_{H^2(\Gamma)}$. 
By estimating the right-hand side of \eqref{def-w} using \eqref{M-Diff}--\eqref{A-Diff}, along with the results in Lemma \ref{Lemma:W1p} and the inverse inequality for finite element functions, we immediately get the following estimate: 
\begin{align*}
\bigg| \int_{\Gamma} (w\varphi^l+\nabla_{\Gamma} w \cdot \nabla_{\Gamma}  \varphi^l) \bigg| 
\le
&C\|{{{\rm id}_{\Gamma_h[\x^*]}}}-\hat y_h^*\|_{H^1(\Gamma_h[\x^*])} \|\hat u_h^*\|_{W^{1,\infty}(\Gamma_h[\x^*])} \|P_h\varphi\|_{H^1(\Gamma_h[\x^*])} \\
\le
&Ch^{-2}\|{{{\rm id}_{\Gamma_h[\x^*]}}}-\hat y_h^*\|_{L^2(\Gamma_h[\x^*])} \|\hat u_h^*\|_{W^{1,\infty}(\Gamma_h[\x^*])} 
\|\varphi\|_{L^2(\Gamma_h[\x^*])} 
\\
\le
&Ch^{-2}\|{{{\rm id}_{\Gamma_h[\x^*]}}}-\hat y_h^*\|_{L^2(\Gamma_h[\x^*])} 
\|\varphi^l\|_{L^2(\Gamma)} , 
\end{align*}
which implies that (via duality) 
$
\int_{\Gamma} (w\varphi^l+\nabla_{\Gamma} w \cdot \nabla_{\Gamma}  \varphi^l) 
= \int_{\Gamma} f \varphi^l 
$
for some function $f$ satisfying the following inequality:
$$
\|f\|_{L^2(\Gamma)} \le Ch^{-2}\|{{{\rm id}_{\Gamma_h[\x^*]}}}-\hat y_h^*\|_{L^2(\Gamma_h[\x^*])}  .
$$
The standard $H^2$ regularity estimate of elliptic equations says that 
$$
\|w\|_{H^2(\Gamma)} \le
C\|f\|_{L^2(\Gamma)} \le Ch^{-2}\|{{{\rm id}_{\Gamma_h[\x^*]}}}-\hat y_h^*\|_{L^2(\Gamma_h[\x^*])}  .
$$
Substituting this into \eqref{H-H-L2-1} yields 
\begin{align}\label{H-H-L2-2}
\|(\hat u_h^*)^l-u\|_{L^2(\Gamma)} 
&\le 
\|w\|_{L^2(\Gamma)} 
+ 
C\|{{{\rm id}_{\Gamma_h[\x^*]}}}-\hat y_h^*\|_{L^2(\Gamma_h[\x^*])} + C h^{k+1} . 
\end{align} 

It remains to estimate $\|w\|_{L^2(\Gamma)}$ and $\|{{{\rm id}_{\Gamma_h[\x^*]}}}-\hat y_h^*\|_{L^2(\Gamma_h[\x^*])}$. This can be done as follows, by using \eqref{def-w} with a duality argument. We define $\varphi^l\in H^2(\Gamma)$ to be the solution of 
\begin{align}\label{def-varphi}
\varphi^l - \Delta_{\Gamma}  \varphi^l = w \,\,\,\mbox{on}\,\,\, \Gamma ,
\end{align}
and let $\varphi$ be the inverse lift of $\varphi^l$ onto $\Gamma_h[\x^*]$. 
Then 
\begin{align}\label{H2-varphi-ww}
\|\varphi^l\|_{H^2(\Gamma)} \le C \|w\|_{L^2(\Gamma)} . 
\end{align}
For $\y^\theta:=(1-\theta)\x^*+\theta \y^*$, let $\hat u_h^{*,\theta}$, $\hat \varphi_h^{\theta}$ and $\hat e_y^\theta$ be the finite element functions on $\Gamma_h[\y^\theta]$ with nodal vectors ${\bf u}^*$, ${\bf P}_h\varphi$ and ${\bf e_y}:=\y^*-\x^*$, respectively. In particular, $\hat u_h^{*,0} = \hat u_h^* $. Then surface $\Gamma_h[\y^\theta]$ moves with velocity $\hat e_y^{\theta} $ as parameter $\theta\in[0,1]$ changes, and the following relation holds: 
$$
\partial_\theta^{\bullet} \hat u_h^{*,\theta} = \partial_\theta^{\bullet} \hat \varphi_h^{\theta} = 0 .
$$
Testing \eqref{def-varphi} by $w$ and using \eqref{def-w}, we have 
\begin{align} \label{w-L2}
\|w\|_{L^2(\Gamma)}^2 
&=\int_{\Gamma} (w\varphi^l+\nabla_{\Gamma} w \cdot \nabla_{\Gamma}  \varphi^l)  \notag \\ 
&=  ({\bf K}(\x^*)-{\bf K}(\y^*)){\bf u}^*\cdot ({\bf P}_h\varphi) \notag \\
&=
\bigg| \int_0^1 \frac{\d}{\d\theta} \int_{\Gamma_h[\y^\theta]} (\hat u_h^{*,\theta}\cdot \hat\varphi_h^\theta + \nabla_{\Gamma_h[\y^\theta]}\hat u_h^{*,\theta} \cdot \nabla_{\Gamma_h[\y^\theta]} \hat\varphi_h^\theta ) \bigg| \notag \\
&=
\bigg|  \int_0^1 \int_{\Gamma_h[\y^\theta]} (\hat u_h^{*,\theta}\cdot \hat\varphi_h^{\theta} \nabla_{\Gamma_h[\y^\theta]}\cdot \hat e_y^{\theta} 
+ \nabla_{\Gamma_h[\y^\theta]}\hat u_h^{*,\theta}\cdot D_{\Gamma_h[\y^\theta]} \hat e_y^{\theta}\nabla_{\Gamma_h[\y^\theta]} \hat\varphi_h^{\theta} ) \, \d\theta \bigg| \notag\\
&\hspace{255pt} \mbox{(see Lemma \ref{BilinearMA})} \notag\\
&= 
\big| B_0 + B_ 1 + B_2 + B_3 \big| ,
\end{align} 
where 
\begin{align}\label{def-B1-B2}
B_0:= & \int_{\Gamma} (\hat u_h^{*,l}\cdot \hat\varphi_h^{0,l} \nabla_{\Gamma}\cdot \hat e_y^{0,l} 
+ \nabla_{\Gamma} u \cdot D_{\Gamma} \hat e_y^{0,l}\nabla_{\Gamma} \varphi^l ) \\ 
B_1 :=& \int_{\Gamma} (\nabla_{\Gamma}\hat u_h^{*,l}\cdot D_{\Gamma} \hat e_y^{0,l}\nabla_{\Gamma} \hat\varphi_h^{0,l} - \nabla_{\Gamma} u \cdot D_{\Gamma} \hat e_y^{0,l}\nabla_{\Gamma} \varphi^l ) \\
B_2 :=& \int_{\Gamma_h[\x^*]} (\hat u_h^{*}\cdot \hat\varphi_h^{0} \nabla_{\Gamma_h[\x^*]}\cdot \hat e_y^{0}+ \nabla_{\Gamma_h[\x^*]}\hat u_h^{*}\cdot D_{\Gamma_h[\x^*]} \hat e_y^{0}\nabla_{\Gamma_h[\x^*]} \hat\varphi_h^{0} ) \notag\\
& - \int_{\Gamma} (\hat u_h^{*,l}\cdot \hat\varphi_h^{0,l} \nabla_{\Gamma}\cdot \hat e_y^{0,l} 
+ \nabla_{\Gamma}\hat u_h^{*,l}\cdot D_{\Gamma} \hat e_y^{0,l}\nabla_{\Gamma} \hat\varphi_h^{0,l} )  
\\
B_3 :=& \int_0^1 \int_{\Gamma_h[\y^\theta]} (\hat u_h^{*,\theta} \cdot \hat\varphi_h^{\theta} \nabla_{\Gamma_h[\y^\theta]}\cdot \hat e_y^{\theta} 
+ \nabla_{\Gamma_h[\y^\theta]} \hat u_h^{*,\theta} \cdot D_{\Gamma_h[\y^\theta]} \hat e_y^{\theta}\nabla_{\Gamma_h[\y^\theta]} \hat\varphi_h^{\theta} ) \, \d\theta \notag\\
&- \int_0^1 \int_{\Gamma_h[\x^*]} (\hat u_h^{*}\cdot \hat\varphi_h^{0} \nabla_{\Gamma_h[\x^*]}\cdot \hat e_y^{0}+ \nabla_{\Gamma_h[\x^*]}\hat u_h^{*}\cdot D_{\Gamma_h[\x^*]} \hat e_y^{0}\nabla_{\Gamma_h[\x^*]} \hat\varphi_h^{0} )   \, \d\theta 
\end{align}
In the expression of $B_0$, we can remove the gradient acting on $\hat e_y^{0,l}$ by utilizing integration by parts. This will yield the following result: 
\begin{align}\label{estimate-B0}
|B_0| 
\le C\|\hat e_y^{0,l} \|_{L^2(\Gamma)}  \|\varphi^l\|_{H^2(\Gamma)} 
\le C\|\hat e_y^0 \|_{L^2(\Gamma_h[\x^*])}  \|\varphi^l\|_{H^2(\Gamma)} . 
\end{align}
Since $\hat\varphi_h^{0,l} = (P_h\varphi)^l$, it follows that 
$\|\nabla_\Gamma(\varphi^l - \hat\varphi_h^{0,l} ) \|_{L^2(\Gamma)}  \le Ch\|\varphi^l\|_{H^2(\Gamma)}$. This implies 
\begin{align}\label{estimate-B1}
|B_1| 
\le&
C\| (\hat u_h^*)^l - u \|_{W^{1,p}(\Gamma)} \|\hat e_y^{0,l} \|_{H^1(\Gamma)} \|\varphi^l\|_{{W^{1,\frac{2p}{p-2}}(\Gamma)}} \notag\\
&
+\| (\hat u_h^*)^l \|_{W^{1,\infty}(\Gamma_h[\x^*])} \|\hat e_y^{0,l} \|_{H^1(\Gamma)} Ch\|\varphi^l\|_{H^2(\Gamma)} 
\quad\mbox{(for some $p>2$)}\notag \\
\le&
Ch \|\hat e_y^{0,l} \|_{H^1(\Gamma)} \|\varphi^l\|_{H^2(\Gamma)} , 
\end{align}
where the last inequality follows from the $W^{1,p}$ error estimate in Lemma \ref{Lemma:W1p}. Similarly, by estimating $B_2$ with the error between the interpolated surface $\Gamma_h[\bf x^*]$ and the exact surface $\Gamma[X]$ (cf. \cite[Proposition 2.3]{Demlow-2009}),  we can derive the following result:  
\begin{align}\label{estimate-B2-B3}
|B_2|
\le&
Ch \|\hat e_y^{0,l} \|_{H^1(\Gamma)} \|\varphi^l\|_{H^2(\Gamma)} . 
\end{align}
{We consider the intermediate surfaces between $\Gamma_h[\y^\theta]$ and $\Gamma_h[\x^*]$ defined by
$$
\Gamma_h[\y^{\theta,\alpha}], \quad \text{with} \quad \y^{\theta,\alpha} = (1-\alpha)\x^* + \alpha\y^\theta = \x^* + \alpha\theta\e_\y,\quad \alpha \in [0,1],
$$
for fixed $\theta \in [0,1]$. As $\alpha$ varies, the intermediate surface $\Gamma_h[\y^{\theta,\alpha}]$ moves with velocity $\theta \hat e_y^\theta$. By employing these intermediate surfaces along with the estimates in \eqref{M-Diff}--\eqref{A-Diff-2} and the bound $\|\hat e_y^0\|_{W^{1,p}(\Gamma_h[\x^*])} \le Ch$ from \eqref{Ritz-W1p} in Lemma \ref{Lemma:W1p}, we can obtain (details are omitted)
\begin{align}\label{estimate-B3}
  |B_3|
  \le&
  Ch \|\hat e_y^{0,l} \|_{H^1(\Gamma)} \|\varphi^l\|_{H^2(\Gamma)}.
\end{align}
}
Therefore, by substituting the estimates of $B_0$, $B_1$, $B_2$ and $B_3$ into \eqref{w-L2} and using the estimate of $\|\varphi^l\|_{H^2(\Gamma)} $ in \eqref{H2-varphi-ww}, we obtain 
\begin{align*}
\|w\|_{L^2(\Gamma)}^2
&\le
C \|\hat e_y^{0}\|_{L^2(\Gamma_h[\bf x^*])} \|w\|_{L^2(\Gamma)} 
+ Ch \|\hat e_y^{0} \|_{H^1(\Gamma_h[\bf x^*])} \|w\|_{L^2(\Gamma)} \notag\\
&\le
C \|\hat e_y^{0}\|_{L^2(\Gamma_h[\bf x^*])} \|w\|_{L^2(\Gamma)} 
+ Ch^{k+1} \|w\|_{L^2(\Gamma)} , 
\end{align*} 
where the last inequality follows from the $W^{1,p}$ estimate of $\hat e_y^{0} = {{{\rm id}_{\Gamma_h[\x^*]}}}-\hat y_h^*$ in Lemma \ref{Lemma:W1p} with $p=2$. This implies that 
\begin{align}\label{w_L2_bound}
\|w\|_{L^2(\Gamma)} 
&\le
C \| {{{\rm id}_{\Gamma_h[\x^*]}}}-\hat y_h^*\|_{L^2(\Gamma_h[\bf x^*])} + Ch^{k+1} .
\end{align} 
Then, substituting this result into \eqref{H-H-L2-2}, we obtain 
\begin{align}\label{Hh-H-L2-1-y-x}
\|(\hat u_h^*)^l-u\|_{L^2(\Gamma)} 
&\le  
C\|{{{\rm id}_{\Gamma_h[\x^*]}}}-\hat y_h^*\|_{L^2(\Gamma_h[\x^*])} + C h^{k+1}. 
\end{align} 
The first term on the right-hand side of \eqref{Hh-H-L2-1-y-x} can be estimated similarly as \eqref{id-yh-W1p-expr}, by choosing $p=2$ in \eqref{id-yh-Lp} and rewriting it equivalently as follows: 
\begin{align}\label{id-yh-L2-expr}
\|{{{\rm id}_{\Gamma_h[\x^*(s)]}}}-\hat y_h^*\|_{L^2(\Gamma[\x^*(s)])} 
&\le
C\int_0^s \| (\hat u_h^*)^l(t) - u(t) \|_{L^2(\Gamma)} \d t + Ch^{k+1}, ~~~\rm{for}~s\in[0,T].
\end{align}
Then, substituting \eqref{Hh-H-L2-1-y-x} into \eqref{id-yh-L2-expr} and using Gronwall's inequality, we obtain an optimal-order estimate of $\|{{{\rm id}_{\Gamma_h[\x^*]}}}-\hat y_h^*\|_{L^2(\Gamma_h[\x^*])} $, i.e., 
\begin{align}
\|{{{\rm id}_{\Gamma_h[\x^*]}}}-\hat y_h^*\|_{L^2(\Gamma_h[\x^*])} 
\le Ch^{k+1}, ~~~\rm{for}~s\in[0,T].
\end{align}
Substituting this estimate back into \eqref{Hh-H-L2-1-y-x} yields the following result: 
\begin{align}\label{L2-u-u}
\|(\hat u_h^*)^l-u\|_{L^2(\Gamma)} &\le Ch^{k+1} , ~~~\rm{for}~s\in[0,T].
\end{align} 
Since $\|(\hat I_h^*u)^l-u\|_{L^2(\Gamma)} \le Ch^{k+1} $, by using \eqref{L2-u-u} and the triangle inequality
\begin{align*}
\|(\hat u_h^*)^l-(\hat I_h^*u)^l\|_{L^2(\Gamma)}
\le 
\|(\hat u_h^*)^l-u\|_{L^2(\Gamma)}  + \|(\hat I_h^*u)^l-u\|_{L^2(\Gamma)} ,
\end{align*} 
as well as the norm equivalence $\|(\hat u_h^*)^l-(\hat I_h^*u)^l\|_{L^2(\Gamma)} \sim \|\hat u_h^*-\hat I_h^*u\|_{L^2(\Gamma_h[\x^*])}  $, 
we obtain the $L^2$ error estimate of the dynamic Ritz projection in Lemma \ref{Lemma:L2}. 
\hfill\end{proof}

Substituting Lemma \ref{Lemma:W1p} into \eqref{W1p-Hh-2} with $p=2$, and substituting Lemma \ref{Lemma:L2} into \eqref{w_L2_bound}, we obtain  
\begin{align}\label{L2-H1-w} 
\|w\|_{L^2(\Gamma)} + h\|w\|_{H^1(\Gamma)} \le Ch^{k+1} . 
\end{align} 
This result will be used in the next subsection. 


In addition to the $W^{1,p}$ and $L^2$ error estimates of the dynamic Ritz projection, we can also differentiate equation \eqref{KH.varphi} in time and, by using the resulting derivative equation, prove the following $L^2$ error estimates for the material derivative of the dynamic Ritz projection. 

\begin{lemma}\label{Lemma:L2-dtH}
{\it 
Under the assumptions of Theorem {\rm\ref{THM1}}, there exists $h_0>0$ such that, for mesh size $h\le h_0$, the solution $(\y^*,{\bf H}^*,{\bf n}^*)$ of \eqref{def-yHn} satisfies the following estimates{\rm:} 
\begin{align}
\| (\partial_{t,h}^\bullet H_h^*)^\wedge -  \hat I_h^*\partial_t^{\bullet}H\|_{L^2(\Gamma_h[\x^*])}
+h \| (\partial_{t,h}^\bullet H_h^*)^\wedge - \hat  I_h^*\partial_t^{\bullet}H\|_{H^1(\Gamma_h[\x^*])} 
&\le C h^{k+1}, \label{L2-dtH*} \\
\| (\partial_{t,h}^\bullet n_h^*)^\wedge - \hat  I_h^*\partial_t^{\bullet}n\|_{L^2(\Gamma_h[\x^*])}
+h \| (\partial_{t,h}^\bullet n_h^*)^\wedge - \hat  I_h^*\partial_t^{\bullet}n\|_{H^1(\Gamma_h[\x^*])} 
&\le C h^{k+1} , \label{L2-dtn*} 
\end{align}
where $\partial_{t,h}^\bullet H_h^*$ and $(\partial_{t,h}^\bullet H_h^*)^\wedge$ are finite element functions on $\Gamma_h[\y^*]$ and $\Gamma_h[\x^*]$, respectively, with a common nodal vector $\dot{\bf H}^*$. 
}
\end{lemma} 

The proof of Lemma \ref{Lemma:L2-dtH} is based on differentiating \eqref{KH.varphi} in time, which leads to complicated expressions. However, the techniques for proving Lemma \ref{Lemma:L2-dtH} are similar as those for proving Lemma \ref{Lemma:W1p} and Lemma \ref{Lemma:L2}. Therefore, we omit the proof here and refer the readers to Appendix \ref{appendix_C} in the supplementary material.

\subsection{Estimates of the remainders}

The remainders $d_H^*$ and $d_n^*$ defined in \eqref{weak-Ritz} can be estimated by using the approximation properties of the dynamic Ritz projection in Lemma \ref{Lemma:W1p}--\ref{Lemma:L2-dtH}. This result is presented in the following lemma and the proof is omitted. We refer the readers to Appendix \ref{appendix_D} in the supplementary material for more details. 
\begin{lemma}\label{Lemma:defects}
Under the assumptions of Theorem {\rm\ref{THM1}}, there exists $h_0>0$ such that, for mesh size $h\le h_0$, the remainder $d_{u}^*$ defined in \eqref{weak-Ritz-u-version} satisfies the following estimate{\rm:} 
\begin{align*}
\Big| \int_{\Gamma_h[\y^*]} d_u^* \cdot \chi_u \Big | 
\le Ch^{k+1}  \|\chi_u \|_{H^1(\Gamma_h[\y^*])^4}  . 
\end{align*}
\end{lemma}

\section{Proof of Theorem \ref{THM2} and Theorem \ref{THM3}}\label{Proof-THM2}

In this section we prove Theorem \ref{THM2} on the optimal-order convergence of the parametric FEM for mean curvature flow in the $L^\infty(0,T;L^2)$ norm, by utilizing the estimates of the dynamic Ritz projection in Lemma \ref{Lemma:W1p}--\ref{Lemma:L2-dtH} and the estimates of the remainders in Lemma \ref{Lemma:defects}. 

\subsection{Basic settings}\label{section:THM2-basic}

The numerical solution $(\x,\v,\u)$ and the dynamic Ritz projection $(\x^*,\v^*,\u^*)$ satisfy equations \eqref{matrix-vector-form} and \eqref{ritz-matrix-form}, respectively. 
By subtracting \eqref{ritz-matrix-form} from \eqref{matrix-vector-form}, we find that the errors functions
$$
\mathbf{e}_{\mathbf{x}} = \mathbf{x} - \y^*, \quad \mathbf{e}_{\mathbf{v}} = \mathbf{v} - \mathbf{v}^* 
\quad\mbox{and}\quad \mathbf{e}_{\mathbf{u}} = \mathbf{u} - \mathbf{u}^* 
$$
satisfy the following equations:
\begin{subequations}\label{error-copy}
    \begin{align}
        \dot {{\bf e}}_\x & = \bf e_{\bf v} \label{error-x-copy}\\
        \bf e_{\bf v} & = -({\bf I}_h({\bf H}\bullet {\bf n} )-{\bf I}_h({\bf H}^*\bullet {\bf n}^* )) \label{error-v-copy}\\
        {\bf M}(\bf x) \dot{\bf e}_{\bf u} + {\bf A}(\bf x)\bf e_{\bf u} & =  - (\bf M(x) - \bf M(\y^*))\dot{\bf u}^* 
         - (\bf A(x) - \bf A(\y^*))\bf u^* \notag \\
        & \quad + (\bf f(\bf x,\bf u) -\bf f(\y^*,\bf u^*) ) - {\bf M}(\y^*) {\bf d}_{\bf u}^* . 
        \label{error-u-copy}
    \end{align}
\end{subequations}
Let $\x^\theta=(1-\theta)\y^*+\theta\x$ for $\theta\in[0,1]$, which defines an intermediate surface $\Gamma_h[\mathbf{x}^\theta]$ moving with the velocity $e_x^\theta$ as parameter $\theta\in[0,1]$ changes, and denote by $e_x^\theta$, $e_v^\theta$ and $e_u^\theta$ the finite element functions on $\Gamma_h[\x^\theta]$ with nodal vectors $\e_\x$, $\e_\v$ and $\e_\u$, respectively. In particular, we denote $e_x=e_x^0$, $e_v=e_v^0$ and $e_u=e_u^0$, which are finite element functions on $\Gamma_h[\y^*]$. 
On the intermediate surface $\Gamma_h[\mathbf{x}^\theta]$ we also define finite element functions 
$$
v_h^\theta \quad\mbox{and}\quad u_h^\theta=((n_h^\theta)^\top, H_h^\theta)^\top 
$$ 
with nodal vectors $\v^\theta = (1-\theta)\v^* + \theta\v$ and $\u^\theta = (1-\theta)\u^* + \theta\u$, respectively. We also denote by $u_h^{*,\theta}=((n_h^{*,\theta})^\top,H_h^{*,\theta})^\top$ the finite element function on $\Gamma_h[\mathbf{x}^\theta]$ with nodal vector $\u^*$. 

Similar as the proof of {Theorem \ref{THM1}}, we define $t^* \in [0, T]$ as the maximal time such that the numerical solution exists and the following inequalities are satisfied:
\begin{subequations}\label{error-induction-assumption} 
    \begin{align}
        \|e_x(\cdot, t)\|_{W^{1,\infty}(\Gamma_h[\y^*])} & \leq h^{k-0.1}\label{error-induction-x}\\
        \|e_u(\cdot, t)\|_{L^\infty(\Gamma_h[\y^*])} & \leq h^{k-0.1} \label{error-induction-u-infty} \quad \text{for} \quad t \in [0, t^*].
    \end{align}
\end{subequations}
At time $t=0$ we have $e_x(\cdot, 0) = 0$ and $e_u=\hat I_h^*u - \hat u_h^*$ on $\Gamma_h[\x^0]$. Therefore, by using the inverse inequality of finite element functions and the $L^2$ estimates of $\hat I_h^*u - \hat u_h^*$ in Theorem {\rm\ref{THM1}}, we have 
\begin{equation*}
    \begin{aligned}
        \|e_u(\cdot, 0)\|_{L^\infty(\Gamma_h[\y^*])} & \leq Ch^{-1} \|e_u(\cdot, 0)\|_{L^2(\Gamma_h[\y^*])} \le Ch^{k} . 
    \end{aligned}
\end{equation*}
For sufficiently small $h$ such that $Ch^{k}\le h^{k-0.1}$, the inequality above implies $t^*>0$. 

Note that, under condition \eqref{error-induction-assumption}, the $L^p$ and $W^{1,p}$ norms on surfaces $\Gamma_h[\mathbf{x}]$ and $\Gamma_h[\y^*]$ are equivalent for $1 \leq p \leq \infty$ (as shown in Lemma \ref{CBI}). This norm equivalence will be used frequently in the following subsections. In particular, under condition \eqref{error-induction-assumption}, we shall prove the following proposition (with some constants $h_0$ and $C$ that are independent of $t^*$). 

\begin{proposition}\label{Proposition:error-final}
{\it 
Under the assumptions in Theorem {\rm\ref{THM1}} and \eqref{error-induction-assumption} there exists $h_0>0$ such that, when mesh size $h\le h_0$, the following estimate holds{\rm:} 
        \begin{align}\label{error-final}
         \|e_x\|_{L^\infty(0,t^*;H^1(\Gamma_h[\y^*]))} 
         + \|e_u\|_{L^\infty(0,t^*;L^2(\Gamma_h[\y^*]))} 
         + \|e_u\|_{L^2(0,t^*;H^1(\Gamma_h[\y^*]))}  
         \le Ch^{k+1}.
        \end{align}
}
\end{proposition}
{{\begin{remark}\label{maximal-time-T}\upshape
By the local-in-time existence and uniqueness, as well as the continuity of solutions, to the ODE system \eqref{matrix-vector-form}, there exists $\delta_h>0$ such that the numerical solution and the error estimate in \eqref{error-final} holds for $t\in[0,t^*+\delta_h]$, with $C$ replaced by $2C$ therein. This would imply that, when $h$ is smaller than some constant (which is independent of $t^*$), \eqref{error-induction-assumption} holds for $t\in[0,t^*+\delta_h]$. This would imply that $t^*=T$ (otherwise $t^*\in(0,T]$ would not be the maximal time satisfying the condition), and therefore, the error estimate in \eqref{error-final} holds for $t\in[0,T]$. Then, by the norm equivalence between $\Gamma_h[\y^*]$ and $\Gamma_h[\x^*]$, the error estimate in \eqref{error-final} can be equivalently written into \eqref{THM2-result}. This would complete the proof of Theorem \ref{THM2}. 
\end{remark}
}}

\subsection{Estimates of $\|e_u\|_{L^\infty(0,t;L^2(\Gamma_h[\y^*]))}$ and $\|e_u\|_{L^2(0,t;H^1(\Gamma_h[\y^*]))}$}

In this subsection, we establish an estimate of $\|e_u\|_{L^\infty(0,t;L^2(\Gamma_h[\y^*]))}$ and $\|e_u\|_{L^2(0,t;H^1(\Gamma_h[\y^*]))}$ in terms of $\|e_x\|_{L^2(0,t;H^1(\Gamma_h[\y^*]))}$ and $\|e_u\|_{L^2(0,t;L^2(\Gamma_h[\y^*]))}$. 

\begin{lemma}\label{Lemma:e_u}
{\it 
Under the assumptions in Theorem {\rm\ref{THM1}} and \eqref{error-induction-assumption}, there exists $h_0>0$ such that when mesh size $h\le h_0$ the following estimate holds for $t\in[0,t^*]${\rm:}
        \begin{equation}\label{e_u}
        \begin{aligned}
            & \quad \|e_u(t)\|_{L^2(\Gamma_h[\y^*(t)])}^2 + \int_0^{t} \|\nabla_{\Gamma_h[\y^*]} e_u(s)\|_{L^2(\Gamma_h[\y^*(s)])}^2 \, \d s \\
            & \leq C \int_{0}^{t} \big(\|e_x(s)\|_{H^1(\Gamma_h[\y^*(s)])}^2 + \|e_u(s)\|_{L^2(\Gamma_h[\y^*(s)])}^2\big)\,\d s  + Ch^{2k+2}.
        \end{aligned}
        \end{equation}
}
\end{lemma}

\begin{proof}
Testing equation \eqref{error-u-copy} with $\mathbf{e}_{\mathbf{u}}$, we obtain the following relation:
    \begin{equation}\label{u-weak-test-eu}
        \begin{aligned}
            & \quad \frac{1}{2} \frac{\d}{\d t}\|e_u^1\|_{L^2(\Gamma_h[\bf x])}^2 + \|\nabla_{\Gamma_h[\bf x]} e_u^1\|_{L^2(\Gamma_h[\bf x])}^2 \\
            & = -\bf e_{\bf u}^\top (\bf M(\bf x) - \bf M(\y^*)) \dot{\bf u}^* - \bf e_{\bf u}^\top (\bf A(\bf x) - \bf A(\y^*)) {\bf u}^*\\
            & \quad\, + {\bf e}_{\bf u}^\top({\bf f}(\x,{\bf u}) - {\bf f}({\y^*}, {\bf u^*})) - {\bf e}_{\bf u}^\top {\bf M}({\y^*}){\bf d}_{\bf u}^* \\
            &=: I_1 + I_2 + I_3 + I_4.
        \end{aligned}
    \end{equation}
where $e_u^1$ denotes the finite element function on $\Gamma_h[\x^\theta]$ with $\theta=1$. Additionally, $\partial_{\theta}^\bullet e_x^\theta =\partial_{\theta}^\bullet e_v^\theta = 0$ and $\partial_{\theta}^\bullet e_u^\theta = 0$. Since $u_h^*\in\Gamma_h[\y^*]$ and $\hat u_h^*\in\Gamma_h[\x^*]$ have the same nodal vectors, by using the equivalence of $L^p$ and $W^{1,p}$ norms on $\Gamma_h[\mathbf{x}^\theta]$ and $\Gamma_h[\y^*]$, and Lemma \ref{Lemma:L2}--\ref{Lemma:L2-dtH}, we have 
\begin{align}
\label{u_h-w-infty-bound}
            \|u_h^*\|_{W^{1,\infty}(\Gamma_h[\y^*])} &\leq C\|\hat{u}_h^*\|_{W^{1,\infty}(\Gamma_h[\mathbf{x}^*])} \notag\\
            &\leq C \|\hat{u}_h^* - \hat{I}_h^* u\|_{W^{1,\infty}(\Gamma_h[\mathbf{x}^*])} + C \|\hat{I}_h^* u\|_{W^{1,\infty}(\Gamma_h[\mathbf{x}^*])} \notag\\
            &\leq C h^{-2} h^{k+1} + C \leq C, \\ 
\label{t-u_h-w-infty-bound}
        \|\partial_{t,h}^\bullet u_h^*\|_{W^{1,\infty}(\Gamma_h[\y^*])} 
        &\leq C\|(\partial_t^\bullet {u}_h^*)^\wedge\|_{W^{1,\infty}(\Gamma_h[\mathbf{x}^*])} \notag\\
        &\leq C \|(\partial_t^\bullet {u}_h^*)^\wedge - \hat{I}_h^*\partial_t^\bullet  u\|_{W^{1,\infty}(\Gamma_h[\mathbf{x}^*])} + C \| \hat{I}_h^*\partial_t^\bullet u\|_{W^{1,\infty}(\Gamma_h[\mathbf{x}^*])} \notag\\
        &\leq C h^{-2} h^{k+1} + C \leq C,
\end{align}
where Lemma \ref{Lemma:L2-dtH} and inverse inequality are used in the last inequalities of \eqref{u_h-w-infty-bound} and \eqref{t-u_h-w-infty-bound}. Then, with inequalities \eqref{u_h-w-infty-bound}--\eqref{t-u_h-w-infty-bound} and \eqref{M-Diff}--\eqref{A-Diff-2}, we can estimate $|I_1|$ and $|I_2|$ as follows:
    \begin{align}\label{I-1-estimtes}
            |I_1| &\le C \|\nabla_{\Gamma_h[\y^*]} e_x \|_{L^2(\Gamma_h[\y^*])} \|e_u\|_{L^2(\Gamma_h[\y^*])} \\
            \label{I-2-estimates}
            |I_2| &\le C \|\nabla_{\Gamma_h[\y^*]} e_x \|_{L^2(\Gamma_h[\y^*])} \|\nabla_{\Gamma_h[\y^*]}e_u\|_{L^2(\Gamma_h[\y^*])}.
    \end{align}
Lemma \ref{Lemma:defects} guarantees that  
    \begin{align}\label{I-4-estimates}
            |I_4| \le Ch^{k+1} \|e_u\|_{H^1(\Gamma_h[\y^*])}.
    \end{align}
    {{It remains to estimate $|I_3|$. By employing the identity \eqref{dt-grad-f}, we can bound $|I_3|$ by the sum of five terms as follows:}}
{\small
    \begin{align}\label{I-3-formulation}
            |I_3| &= |{\bf e}_{\bf u}^\top ({\bf f} (\x,{\bf u}) - {\bf f}({\y^*}, {\bf u^*}))| = \bigg|\int_{\Gamma_h[\bf x]} \big|\nabla _{\Gamma_h[\bf x]} n_h\big|^2 u_h \cdot e_u^1 - \int_{\Gamma_h[\y^*] }\big|\nabla _{\Gamma_h[\y^*]} n_h^*\big|^2 u_h^* \cdot e_u\bigg| \notag\\
            &=\bigg| \int_0^1 \frac{\d}{\d\theta} \int_{\Gamma_h[\x^\theta]} \big|\nabla _{\Gamma_h[\x^\theta]} n_h^\theta\big|^2 u_h^\theta \cdot e_u^\theta \d\theta \bigg| \notag\\
            & \le \bigg| \int_0^1 \int_{\Gamma_h[\x^\theta]} \big|\nabla _{\Gamma_h[\x^\theta]}n_h^\theta\big|^2 u_h^\theta\cdot e_u^\theta \nabla_{\Gamma_h[\x^\theta]} \cdot e_x^\theta \,\d \theta \bigg| 
            + \bigg| \int_0^1 \int_{\Gamma_h[\x^\theta]}  \big|\nabla _{\Gamma_h[\x^\theta]}n_h^\theta\big|^2 e_u^\theta \cdot e_u^\theta \,\d \theta \bigg| \notag\\
            & \quad \, + \bigg| \int_0^1 \int_{\Gamma_h[\x^\theta]} 2\nabla_{\Gamma_h[\x^\theta]} n_h^\theta \Big(\nabla_{\Gamma_h[\x^\theta]} \partial_{\theta}^\bullet n_h^\theta - \nabla_{\Gamma_h[\x^\theta]} e_x^\theta \nabla_{\Gamma_h[\x^\theta]} n_h^\theta  \notag\\
            &\qquad\qquad\qquad\qquad\qquad\qquad\quad + n_{\Gamma_h[\x^\theta]} n_{\Gamma_h[\x^\theta]}^\top (\nabla_{\Gamma_h[\x^\theta]} e_x^\theta)^\top \nabla_{\Gamma_h[\x^\theta]} n_h^\theta \Big) u_h^\theta\cdot e_u^\theta \,\d\theta\bigg| \notag\\
            & =:I_{31} + I_{32} + I_{33} + I_{34} + I_{35} , 
    \end{align}
\!}
where $n_{\Gamma_h[\x^\theta]}$ is the unit normal vector of $\Gamma_h[\x^\theta]$, while $n_h^\theta$ is the finite element function with nodal vector ${\bf n}^\theta = (1-\theta){\bf n}^* + \theta{\bf n}$, with ${\bf n}^*$ and ${\bf n}$ being the nodal vectors of the dynamic Ritz projection $n_h^*\in S_h(\Gamma_h[\y^*])$ and numerical solution $n_h\in S_h(\Gamma_h[\x])$, respectively. 
Since $n_h^\theta = n_h^{*,\theta} + \theta e_n^\theta$ and $u_h^\theta = u_h^{*,\theta} + \theta e_u^\theta$, $I_{31}$ can be bounded as follows:
    \begin{align}\label{I-31-estimate}
            I_{31}  &= \bigg|\int_0^1 \int_{\Gamma_h[\x^\theta]} \big|\nabla_{\Gamma_h[\x^\theta]} (n_h^{*,\theta} + \theta e_n^\theta)\big|^2 (u_h^{*,\theta} + \theta e_u^\theta) \cdot e_u^\theta \nabla_{\Gamma_h[\x^\theta]}\cdot e_x^\theta \,\d\theta\bigg| \notag\\
            & \le \bigg| \int_0^1\int_{\Gamma_h[\x^\theta]} \big|\nabla_{\Gamma_h[\x^\theta]} n_h^{*,\theta}\big|^2 (u_h^{*,\theta}\cdot e_u^\theta)\nabla_{\Gamma_h[\x^\theta]}\cdot e_x^\theta \,\d\theta\bigg| \notag\\
            & \quad \, + \bigg| \int_0^1 \int_{\Gamma_h[\x^\theta]} \theta\big|\nabla_{\Gamma_h[\x^\theta]} n_h^{*,\theta}\big|^2 \big|e_u^\theta\big|^2 \nabla_{\Gamma_h[\x^\theta]}\cdot e_x^\theta \,\d\theta\bigg| \notag\\
            & \quad \,+ \bigg| \int_0^1 \int_{\Gamma_h[\x^\theta]} \theta^2\big|\nabla_{\Gamma_h[\x^\theta]} e_n^\theta\big|^2 (u_h^{*,\theta} \cdot e_u^\theta)\nabla_{\Gamma_h[\x^\theta]}\cdot e_x^\theta\,\d\theta\bigg| \notag\\
            & \quad \,+ \bigg|\int_{0}^1\int_{\Gamma_h[\x^\theta]} \theta^3\big|\nabla_{\Gamma_h[\x^\theta]} e_n^\theta\big|^2 \big|e_u^\theta\big|^2 \nabla_{\Gamma_h[\x^\theta]}\cdot e_x^\theta\,\d\theta\bigg| \notag\\
            & \quad \, + \bigg|\int_0^1 \int_{\Gamma_h[\x^\theta]} 2\theta (\nabla_{\Gamma_h[\x^\theta]}n_h^{*,\theta} \cdot \nabla_{\Gamma_h[\x^\theta]}e_n^\theta)( u_h^{*,\theta} \cdot e_u^\theta)\nabla_{\Gamma_h[\x^\theta]}\cdot e_x^\theta\,\d\theta \bigg| \notag\\
            & \quad\,+ \bigg| \int_0^1 \int_{\Gamma_h[\x^\theta]} 2\theta^2 \nabla_{\Gamma_h[\x^\theta]}n_h^{*,\theta} \cdot \nabla_{\Gamma_h[\x^\theta]}e_n^\theta \,|e_u^\theta|^2\, \nabla_{\Gamma_h[\x^\theta]} \cdot e_x^\theta \,\d\theta  \bigg| \notag\\
            & =: I_{311} + I_{312} + I_{313} + I_{314} + I_{315} + I_{316}.
    \end{align}
Then, with norm equivalence of $L^p$ and $W^{1,p}$ norms on $\Gamma_h[\x^\theta]$ and $\Gamma_h[\y^*]$, as well as estimates \eqref{error-induction-assumption} and \eqref{u_h-w-infty-bound}, the following estimates of $I_{31j}$, $j=1,\dots,6$, can be derived: 
    \begin{align}\label{I-311-316}
            I_{311} & \le C \|e_u\|_{L^2(\Gamma_h[\y^*])} \|\nabla_{\Gamma_h[\y^*]} e_x\|_{L^2(\Gamma_h[\y^*])}, \notag\\
            I_{312} & \le C \|e_u\|_{L^2(\Gamma_h[\y^*])} \|\nabla_{\Gamma_h[\y^*]} e_x\|_{L^2(\Gamma_h[\y^*])}\|e_u\|_{L^\infty(\Gamma_h[\y^*])} \notag\\
            &\le C \|e_u\|_{L^2(\Gamma_h[\y^*])} \|\nabla_{\Gamma_h[\y^*]} e_x\|_{L^2(\Gamma_h[\y^*])}, \notag\\
            I_{313} &\le C \|\nabla_{\Gamma_h[\y^*]} e_n\|_{L^2(\Gamma_h[\y^*])}^2 \|\nabla_{\Gamma_h[\y^*]} e_x\|_{L^\infty(\Gamma_h[\y^*])} \|e_u\|_{L^\infty(\Gamma_h[\y^*])} \notag\\
            & \le C h^{2k-0.2} \|\nabla_{\Gamma_h[\y^*]} e_u\|_{L^2(\Gamma_h[\y^*])}^2, \notag\\
            I_{314} &\le C \|\nabla_{\Gamma_h[\y^*]} e_n\|_{L^2(\Gamma_h[\y^*])}^2 \|\nabla_{\Gamma_h[\y^*]} e_x\|_{L^\infty(\Gamma_h[\y^*])} \|e_u\|_{L^\infty(\Gamma_h[\y^*])}^2 \notag\\
            & \le C h^{3k-0.3} \|\nabla_{\Gamma_h[\y^*]} e_u\|_{L^2(\Gamma_h[\y^*])}^2, \notag\\
            I_{315} &\le C\|\nabla_{\Gamma_h[\y^*]} e_n\|_{L^2(\Gamma_h[\y^*])} \|e_u\|_{L^2(\Gamma_h[\y^*])} \|\nabla_{\Gamma_h[\y^*]} e_x\|_{L^\infty(\Gamma_h[\y^*])} \notag\\
            & \le Ch^{k-0.1} \|e_u\|_{L^2(\Gamma_h[\y^*])}\|\nabla_{\Gamma_h[\y^*]} e_u\|_{L^2(\Gamma_h[\y^*])}, \notag\\
            I_{316} &\le C \|\nabla_{\Gamma_h[\y^*]} e_n\|_{L^2(\Gamma_h[\y^*])} \|e_u\|_{L^2(\Gamma_h[\y^*])} \|e_u\|_{L^\infty(\Gamma_h[\y^*])}\|\nabla_{\Gamma_h[\y^*]} e_x\|_{L^\infty(\Gamma_h[\y^*])} \notag\\
            & \le Ch^{2k-0.2} \|e_u\|_{L^2(\Gamma_h[\y^*])}\|\nabla_{\Gamma_h[\y^*]} e_u\|_{L^2(\Gamma_h[\y^*])}. 
    \end{align} 
By summing up the estimates of $I_{31j}$, $j=1,\dots,6$, in \eqref{I-311-316}, we obtain
    \begin{align}\label{I-31-final}
            I_{31} &\le C\|e_u\|_{L^2(\Gamma_h[\y^*])}^2 + C\|e_x\|_{H^1(\Gamma_h[\y^*])}^2  +Ch^{2k-0.2} \|\nabla_{\Gamma_h[\y^*]}e_u\|_{L^2(\Gamma_h[\y^*])}^2.
    \end{align}
In the same way, the following estimates of $I_{3j}$, $j=2,\dots,5$, can be obtained:  
    \begin{align}
            I_{32}  + I_{34} + I_{35}&\le C\|e_u\|_{L^2(\Gamma_h[\y^*])}^2 + C\|\nabla_{\Gamma_h[\y^*]}e_x\|_{L^2(\Gamma_h[\y^*])}^2 \notag\\
            & \quad\, +Ch^{2k-0.2} \|\nabla_{\Gamma_h[\y^*]}e_u\|_{L^2(\Gamma_h[\y^*])}^2 . 
    \end{align}
    {{Furthermore, by applying Young's inequality, the term \(I_{33}\) can be bounded as follows:
    \begin{align}
      I_{33} \le \epsilon \|\nabla_{\Gamma_h[\y^*]}e_u\|_{L^2(\Gamma_h[\y^*])}^2 + C(\epsilon)\|e_u\|_{L^2(\Gamma_h[\y^*])}^2.
    \end{align}
    These estimates lead to  }}
    \begin{align}\label{I-3-final}
            |I_3| \le C\|e_u\|_{L^2(\Gamma_h[\y^*])}^2 + C\|e_x\|_{H^1(\Gamma_h[\y^*])}^2 +{(Ch^{2k-0.2} + \epsilon)} \|\nabla_{\Gamma_h[\y^*]}e_u\|_{L^2(\Gamma_h[\y^*])}^2.
    \end{align}
Then, substituting estimates \eqref{I-1-estimtes}--\eqref{I-4-estimates} and \eqref{I-3-final} into \eqref{u-weak-test-eu} yields the following result:
    \begin{align}\label{error-u-intermediate-0}
            &\frac{1}{2} \frac{\d}{\d t}\|e_u^1\|_{L^2(\Gamma_h[\bf x])}^2 + \|\nabla_{\Gamma_h[\bf x]} e_u^1\|_{L^2(\Gamma_h[\bf x])}^2 \\
            & \le C\|e_x\|_{H^1(\Gamma_h[\y^*])}^2 + C \|e_u\|_{L^2(\Gamma_h[\y^*])}^2 + {(Ch^{2k-0.2} + \epsilon)}\|\nabla_{\Gamma_h [\y^*]}e_u\|_{L^2(\Gamma_h[\y^*])}^2 + Ch^{2k+2}. \notag
    \end{align}
{By employing the $H^1$ semi-norm equivalence between the surfaces $\Gamma_h[\y^*]$ and $\Gamma_h[\mathbf{x}]$, and by choosing $h$ and $\epsilon$ sufficiently small, the term
\(
(Ch^{2k-0.2} + \epsilon) \|\nabla_{\Gamma_h[\y^*]}e_u\|_{L^2(\Gamma_h[\y^*])}^2
\)
can be absorbed into the left-hand side of the above inequality.} This reduces \eqref{error-u-intermediate-0} to the following result:
    \begin{align}\label{error-u-intermediate}
            &\frac{1}{2} \frac{\d}{\d t}\|e_u^1\|_{L^2(\Gamma_h[\bf x])}^2 + \|\nabla_{\Gamma_h[\bf x]} e_u^1\|_{L^2(\Gamma_h[\bf x])}^2 \notag\\
            & \le C\|e_x\|_{H^1(\Gamma_h[\y^*])}^2 + C \|e_u\|_{L^2(\Gamma_h[\y^*])}^2+ Ch^{2k+2}.
    \end{align}
    Integrating the above inequality from $0$ to $t$, along with the norm equivalence between surfaces $\Gamma_h[\bf x]$ and $\Gamma_h[\y^*(t)]$, we have
    \begin{align}
            &\quad\,\|e_u(t)\|_{L^2(\Gamma_h[\y^*(t)])}^2 + \int_0^t \|\nabla _{\Gamma_h[\y^*(s)]} e_u(s)\|_{L^2(\Gamma_h[\y^*(s)])}^2\,\d s \notag\\
            & \le C \int_0^t \big( \| e_x(s)\|_{H^1(\Gamma_h[\y^*(s)])}^2 + \| e_u(s)\|_{L^2(\Gamma_h[\y^*(s)])}^2 \big)\,\d s +Ch^{2k+2},
    \end{align}
    where $\|e_u(0)\|_{L^2(\Gamma_h[\y^*(0)])} \le Ch^{k+1}$ is used. This proves the result of Lemma \ref{Lemma:e_u}. 
\hfill\end{proof}

\subsection{Estimates of $\|e_v(t)\|_{H^1(\Gamma_h[\y^*(t)])}$}

In this subsection, we establish an estimate of $\|e_v(t)\|_{H^1(\Gamma_h[\y^*(t)])}$ in terms of $\|e_u(t)\|_{H^1(\Gamma_h[\y^*(t)])}$. 

\begin{lemma}\label{Lemma:ev-eu}
{\it 
Under the assumptions in Theorem {\rm\ref{THM1}} and \eqref{error-induction-assumption} there exists $h_0>0$ such that, when mesh size $h\le h_0$, the following estimate holds for $t\in[0,t^*]${\rm:} 
    \begin{equation}\label{ev-eu}
    \begin{aligned}
            \|e_v(t)\|_{H^1(\Gamma_h[\y^*(t)])} \le C \|e_u(t)\|_{H^1(\Gamma_h[\y^*(t)])} .
    \end{aligned}
    \end{equation}
}
\end{lemma}
\begin{proof}
Equation \eqref{error-v-copy} can be written as $\bf e_{\bf v} = - {\bf I}_h[ ({\bf H} - {\bf H}^*)\bullet {\bf n} ] 
- {\bf I}_h[ {\bf H}^* \bullet ({\bf n} - {\bf n}^*) ] $, which implies (using the $H^1$ stability of Lagrange interpolation in Appendix \ref{appendix_A} and the norm equivalence between surfaces $\Gamma_h[\bf x^*]$ and $\Gamma_h[\y^*]$) 
        \begin{align}
            \|e_v\|_{H^1(\Gamma_h[\y^*])} & \le C \|e_v\|_{H^1(\Gamma_h[\x^*])}\notag \\
            & \le C\|e_H\|_{H^1(\Gamma_h[\x^*])} \|n_h^*\|_{W^{1,\infty}(\Gamma_h[\x^*])} + C\|e_n\|_{H^1(\Gamma_h[\x^*])} \|H_h^*\|_{W^{1,\infty}(\Gamma_h[\x^*])} \notag\\
            & \quad \,+ C \|e_H\|_{H^1(\Gamma_h[\x^*])} \|e_n\|_{L^\infty(\Gamma_h[\x^*])} +C \|e_n\|_{H^1(\Gamma_h[\x^*])} \|e_H\|_{L^\infty(\Gamma_h[\x^*])} \notag\\
            & \le C\|e_H\|_{H^1(\Gamma_h[\y^*])} \|n_h^*\|_{W^{1,\infty}(\Gamma_h[\y^*])} + C\|e_n\|_{H^1(\Gamma_h[\y^*])} \|H_h^*\|_{W^{1,\infty}(\Gamma_h[\y^*])} \notag\\
            & \quad \,+ C \|e_H\|_{H^1(\Gamma_h[\y^*])} \|e_n\|_{L^\infty(\Gamma_h[\y^*])} +C \|e_n\|_{H^1(\Gamma_h[\y^*])} \|e_H\|_{L^\infty(\Gamma_h[\y^*])} \notag\\
            & \le C\|e_u\|_{H^1(\Gamma_h[\y^*])}, 
        \end{align}
where \eqref{error-induction-u-infty} and \eqref{u_h-w-infty-bound} are used in the last inequality.
\hfill\end{proof}

We can substitute the estimate in Lemma \ref{Lemma:e_u} into the estimate in Lemma \ref{Lemma:ev-eu}. This yields the following inequality: 
    \begin{align}\label{L2-ev-L2-eu}
            \|e_v\|_{L^2(0,t;H^1(\Gamma_h[\y^*]))} 
            &\le C \|e_u\|_{L^2(0,t;H^1(\Gamma_h[\y^*]))} \\
            &\le C ( \|e_x\|_{L^2(0,t;H^1(\Gamma_h[\y^*]))} + \|e_u\|_{L^2(0,t;L^2(\Gamma_h[\y^*]))} ) + Ch^{k+1} . \notag
    \end{align}

\subsection{Proof of Proposition \ref{Proposition:error-final}}

Since $e_x(\cdot,0)=0$, it follows that 
    \begin{align}\label{e-x-rewrite}
            & \|e_x(t)\|_{H^1(\Gamma_h[\y^*(t)])}^2 \notag\\
            &= \int_0^t \frac{\d}{\d s} \|e_x(s)\|_{H^1(\Gamma_h[\y^*])}^2 \,\d s \notag\\
            & = \int_0^t \Big( 2{\bf e}_\x(s)^\top {\bf K}(\y^*(s)) \dot{\bf e}_\x(s) + {\bf e}_\x(s)^\top\frac{\d}{\d s} {\bf K}(\y^*(s)) {\bf e}_\x(s) \Big) \,\d s \notag\\
            & \le C \int_0^t \Big( \|e_x(s)\|_{H^1(\Gamma_h[\y^*(s)])}\|e_v(s)\|_{H^1(\Gamma_h[\y^*(s)])} + \|e_x(s)\|_{H^1(\Gamma_h[\y^*(s)])}^2\Big) \,\d s \notag\\
            & \le  C\int_0^t \|e_v(s)\|_{H^1(\Gamma_h[\y^*(s)])}^2 \,\d s + C\int_0^t\|e_x(s)\|_{H^1(\Gamma_h[\y^*(s)])}^2 \,\d s , 
    \end{align}
where the second to last inequality follows from the following estimate (which can be derived from the expressions in Lemma \ref{BilinearMA}): 
\begin{align*}
\Big|{\bf e}_\x(s)^\top\frac{\d}{\d s} {\bf K}(\y^*(s)) {\bf e}_\x(s) \Big|
&\le C\|e_x(s)\|_{H^1(\Gamma_h[\y^*(s)])}^2 \|v_h^*(s)\|_{W^{1,\infty}(\Gamma_h[\y^*(s)])} \\
&\le C\|e_x(s)\|_{H^1(\Gamma_h[\y^*(s)])}^2 , 
\end{align*}
in which the boundedness of $\|v_h^*(s)\|_{W^{1,\infty}(\Gamma_h[\y^*(s)])}$ follows from applying the norm equivalence between $\Gamma_h[\y^*]$ and $\Gamma_h[\x^*]$, the triangle inequality, inverse inequality and the $W^{1,p}$ estimate in \eqref{Ritz-W1p} with $p=2$, i.e., 
\begin{align*}
\|v_h^*\|_{W^{1,\infty}(\Gamma_h[\y^*])}
\le C\|\hat v_h^*\|_{W^{1,\infty}(\Gamma_h[\x^*])} 
&\le C\|\hat v_h^* -  \hat I_h^*v\|_{W^{1,\infty}(\Gamma_h[\x^*])} + \|\hat I_h^*v\|_{W^{1,\infty}(\Gamma_h[\x^*])} \\
&\le Ch^{-1} \|\hat v_h^* -  \hat I_h^*v\|_{H^1(\Gamma_h[\x^*])} + \|\hat I_h^*v\|_{W^{1,\infty}(\Gamma_h[\x^*])} \\ 
&\le C  .
\end{align*}

The right-hand side of \eqref{e-x-rewrite} can be estimated with \eqref{L2-ev-L2-eu}. This leads to the following result: 
    \begin{align}\label{ex-L2-H1}
            \|e_x(t)\|_{H^1(\Gamma_h[\y^*(t)])}^2 
            & \le C \int_{0}^{t} \big(\|e_x(s)\|_{H^1(\Gamma_h[\y^*(s)])}^2 + \|e_u(s)\|_{L^2(\Gamma_h[\y^*(s)])}^2\big)\,\d s  + Ch^{2k+2} .
    \end{align}
Then, summing up \eqref{e_u} and \eqref{ex-L2-H1}, we obtain the following result for $t\in(0,t^*] $:  
    \begin{align}\label{ex-eu-L2-H1}
            &\|e_x(t)\|_{H^1(\Gamma_h[\y^*(t)])}^2 + \|e_u(t)\|_{L^2(\Gamma_h[\y^*(t)])}^2 
                + \int_0^{t} \|\nabla_{\Gamma_h[\y^*]} e_u(s)\|_{L^2(\Gamma_h[\y^*(s)])}^2 \, \d s \notag\\
            & \le C \int_{0}^{t} \big(\|e_x(s)\|_{H^1(\Gamma_h[\y^*(s)])}^2 + \|e_u(s)\|_{L^2(\Gamma_h[\y^*(s)])}^2\big)\,\d s  + Ch^{2k+2} .
    \end{align}
The result of Proposition \ref{Proposition:error-final} follows from applying Gronwall's inequality to \eqref{ex-eu-L2-H1}. 
{Moreover, the discussions in Remark \ref{maximal-time-T} show that $t^*=T$. This completes the proof of Theorem \ref{THM2}.}
\hfill$\square$

\section{Numerical tests}
In this section, we present numerical experiments to support the theoretical analysis for the convergence rate of the semidiscrete parametric FEM in \eqref{weak-form}. 

We consider the evolution of the two-dimensional sphere $\Gamma(t)$ under mean curvature flow, which was used for testing the convergence rates of numerical methods for mean curvature flow in \cite{KLL19}. The exact solution of the surface at time $t>0$ is a sphere of radius $R(t) = \sqrt{R(0)^2 - 4t}$ with $R(0)=2$, which reaches zero at time $t = 1$. The mean curvature $H$ and normal vector $n$ of the evolving sphere $\Gamma(t)$ can also be calculated explicitly. 

\begin{figure}[htp!]
  \vspace{-15pt}
  \begin{center}
  \subfigure[Finite elements of degree $k=1$]{\includegraphics[width=0.49\textwidth]{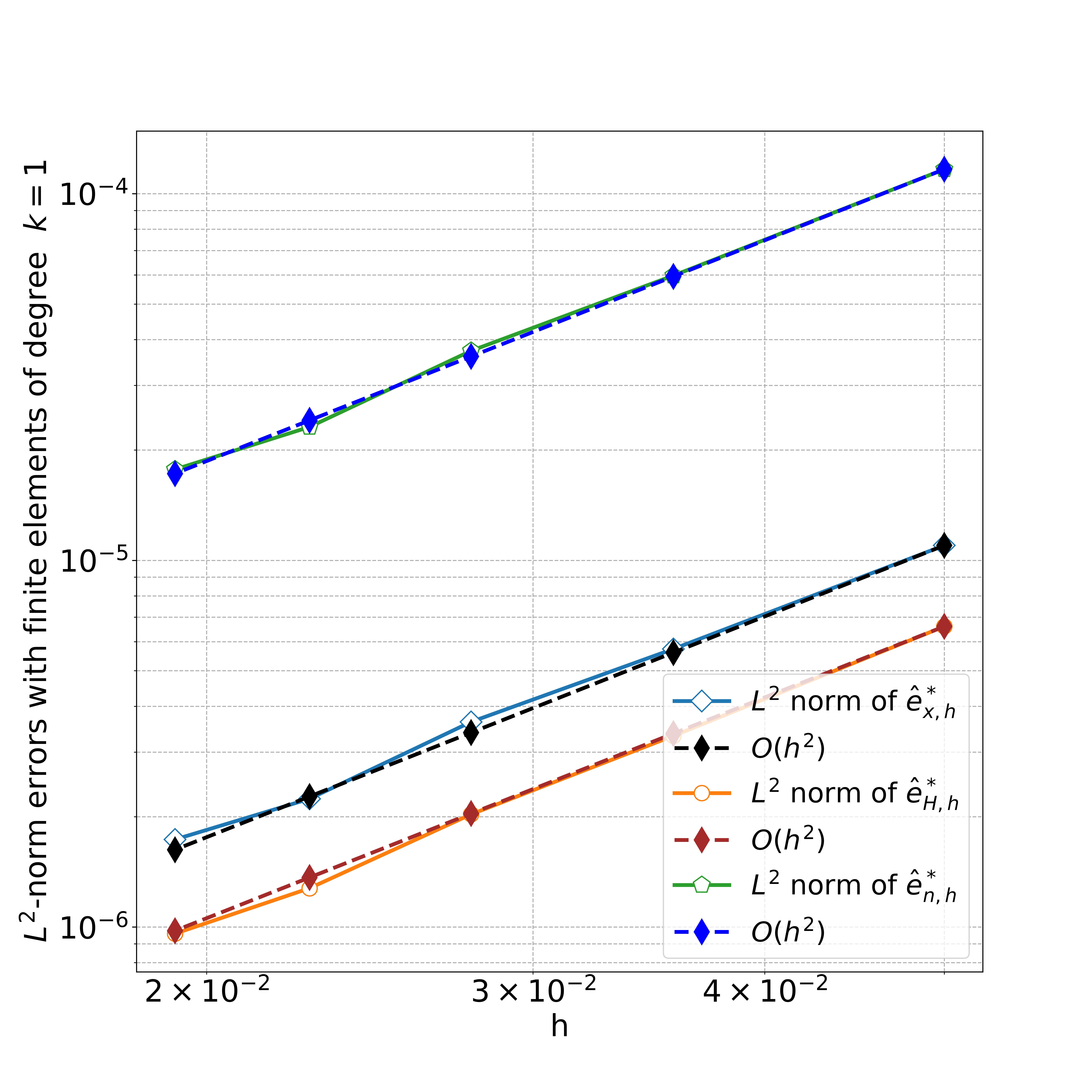}}
  \subfigure[Finite elements of degree $k=2$]{\includegraphics[width=0.49\textwidth]{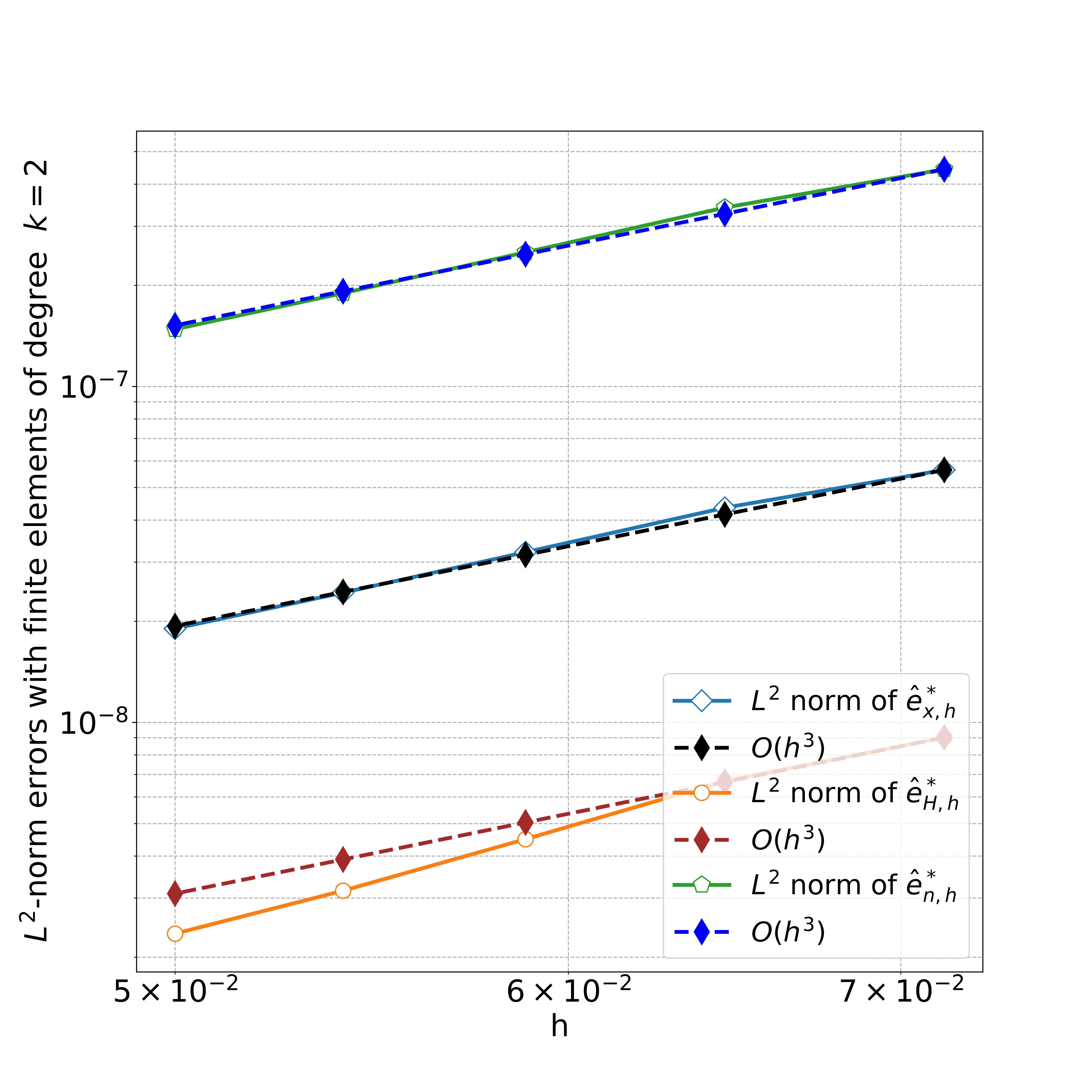}}

  \subfigure[Finite elements of degree $k=3$]{\includegraphics[width=0.49\textwidth]{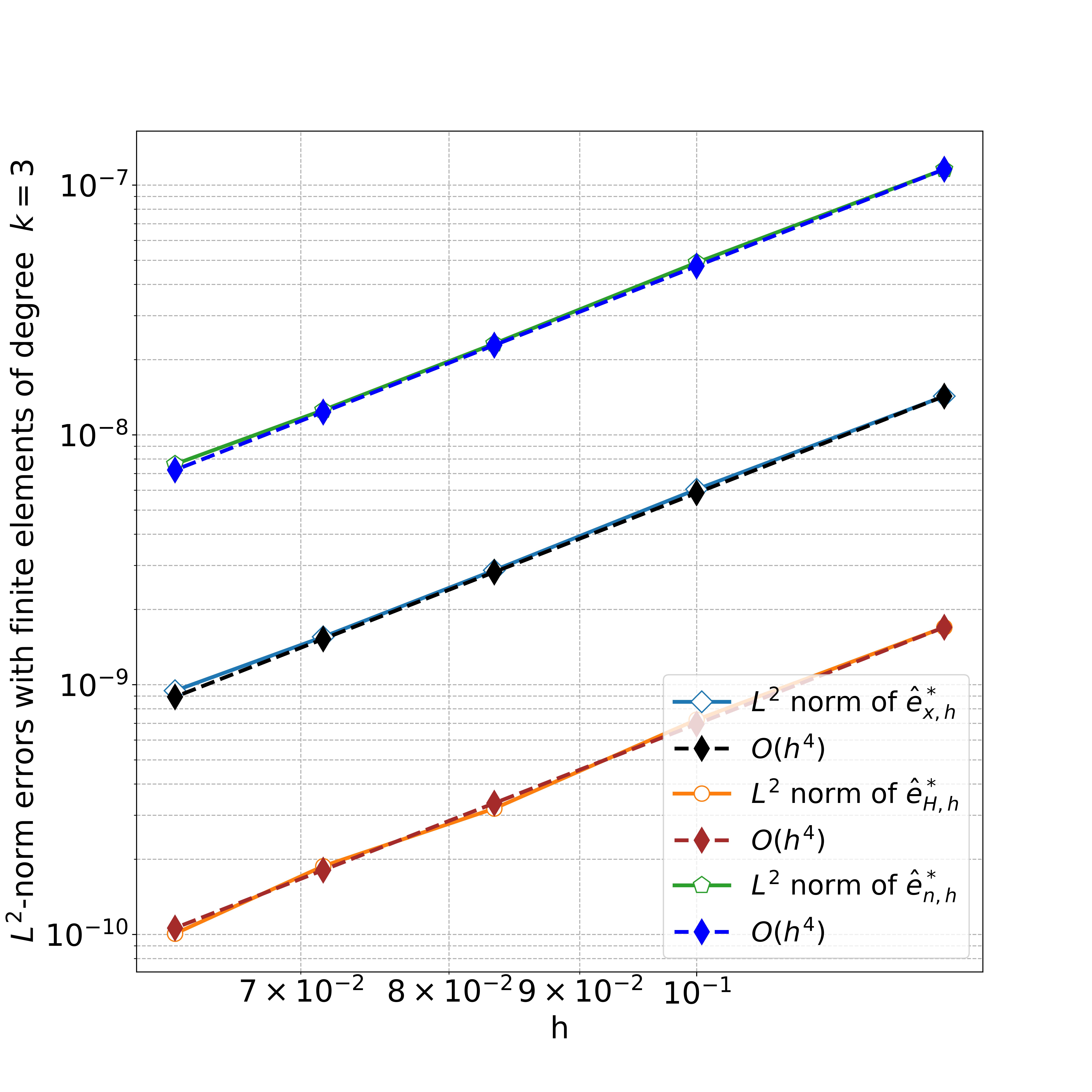}}
  \end{center}
  \vspace{-5pt}
        \caption{\small $L^2$ norm for errors and convergence rates of the numerical solutions.}
        \label{fig:err_rate}
\end{figure}

\begin{figure}[htp!]
  \vspace{-15pt}
  \begin{center}
  \subfigure[Convergence rates for $\hat e_{x,h}$, $\hat e_{H,h}$ and $\hat e_{n,h}$ ]{\includegraphics[width=0.49\textwidth]{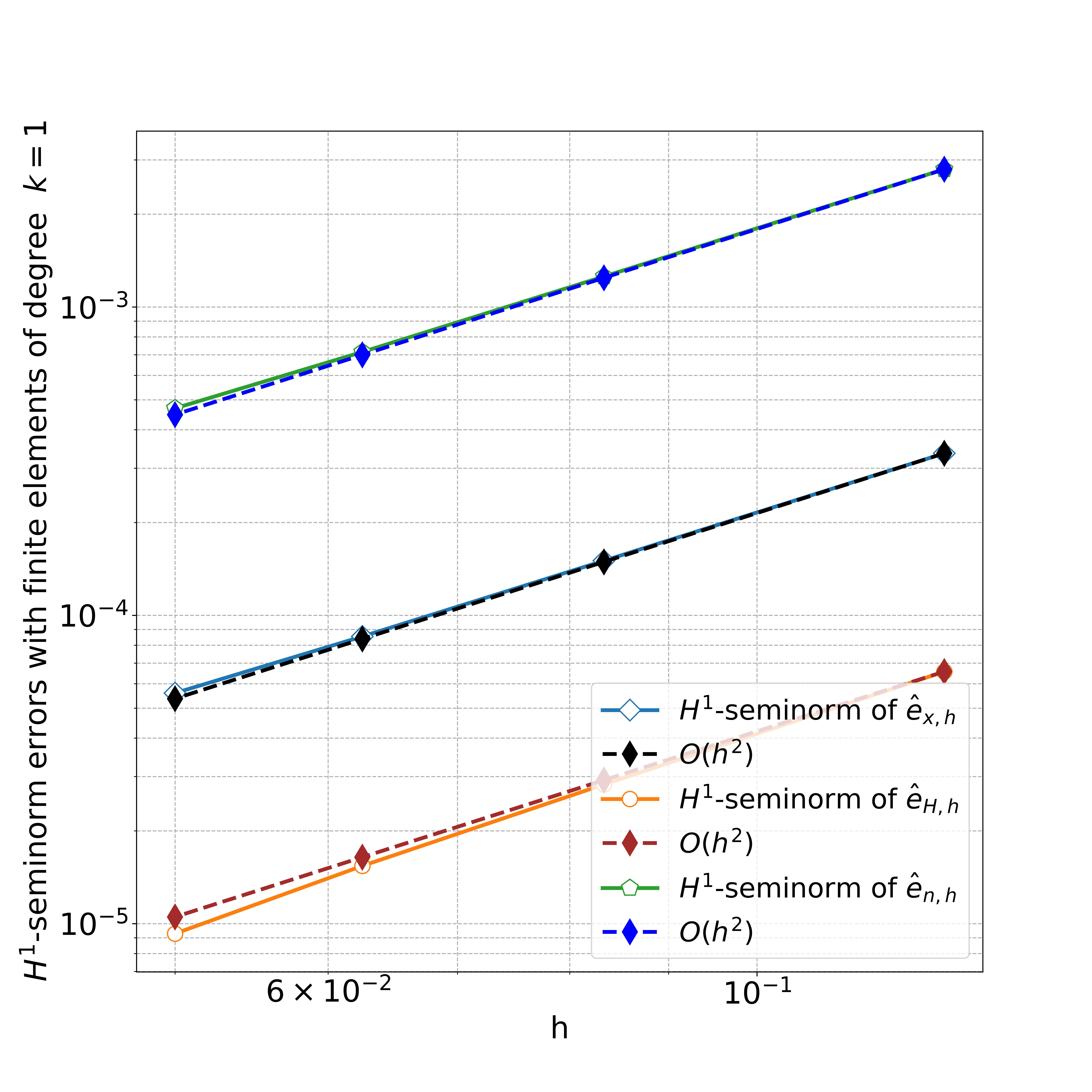}}
  \subfigure[Convergence rates for $\hat e_{x,h}^*$, $\hat e_{H,h}^*$ and $\hat e_{n,h}^*$ ]{\includegraphics[width=0.49\textwidth]{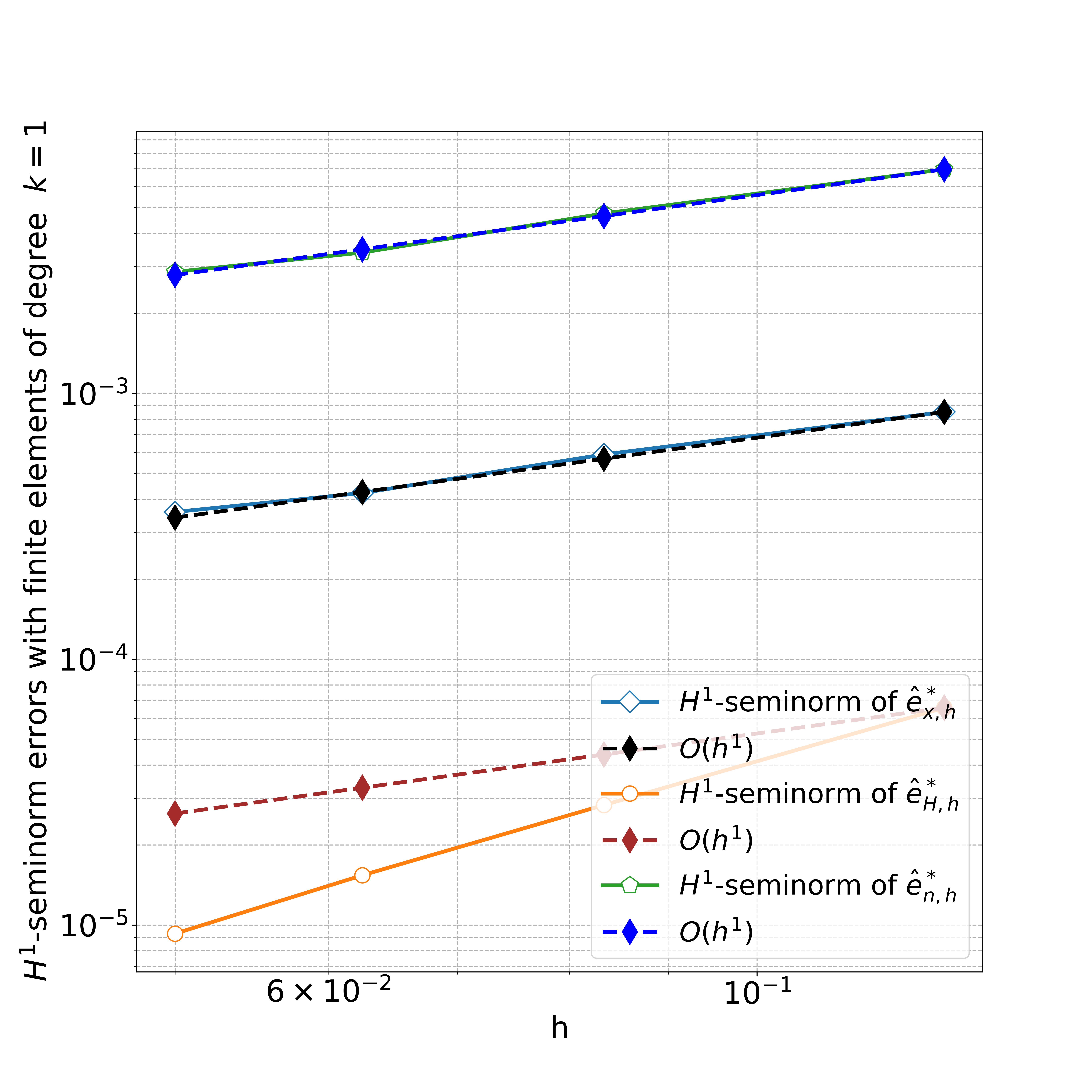}}

  \end{center}
  \vspace{-5pt}
        \caption{\small The $H^1$-seminorm of the errors for finite elements of degree $k=1$}
        \label{fig:err_rate_H1semi}
\end{figure}

We solve the problem by the algorithm in \eqref{weak-form} up to time $T= 0.125$ with finite elements of degrees $1$, $2$ and $3$, respectively, using a four-step backward differentiation formula (BDF), with a sufficiently small time stepsize $\tau = 0.001$ such that the errors from temporal discretizations can be neglected in testing the convergence rates of spatial discretizations. 

The $L^2$ norms of the error functions $\hat e_{x,h}^*=\hat x_h - {\rm id}_{\Gamma_h[\x^*]}$, $\hat e_{H,h}^*=\hat H_h - \hat I_h^*H$ and $\hat e_{n,h}^*=\hat n_h - \hat I_h^*n$ at time $T$ are presented in Figure \ref{fig:err_rate}, which indicates that the errors of the numerical solutions are about $O(h^{k+1})$ for finite elements of degree $k = 1, 2, 3$. This is consistent with the error estimates established in Theorem \ref{THM3}.

{Figure~\ref{fig:err_rate_H1semi}(a) shows the $H^1$-seminorms of the error functions $\hat{e}_{x,h}$, $\hat{e}_{H,h}$, and $\hat{e}_{n,h}$ (between the numerical solution and the dynamic Ritz projection) at time $T$, demonstrating second-order convergence for finite elements of degree $k=1$, which agrees with the error estimate in Theorem~\ref{THM2}. In contrast, the $H^1$-seminorms of $\hat{e}^*_{x,h}$, $\hat{e}^*_{H,h}$, and $\hat{e}^*_{n,h}$ (between the numerical solution and the interpolated solution) is only $O(h)$, as shown in Figure~\ref{fig:err_rate_H1semi}(b). 
}

\section{Conclusion}

We have defined a dynamic Ritz projection of mean curvature flow of closed surfaces in the three-dimensional space, and have proved optimal-order error bounds in the $L^2$ and $W^{1,p}$ norms for the dynamic Ritz projection in approximating the solution of mean curvature flow. By utilizing these approximation results, we have proved optimal-order convergence of parametric FEMs for formulation \eqref{KLL} of mean curvature flow in the $L^\infty(0,T;L^2)$ norm, as well as convergence of parametric FEMs for mean curvature flow with piecewise linear finite elements. The new approach developed in this paper --- analyzing the error of numerical approximation through a dynamic Ritz projection of the mean curvature flow --- can serve as a foundational framework for studying the convergence of numerical approximations for other geometric flows and parametric FEMs.

\section*{Acknowledgement}
The authors would like to thank Prof. Christian Lubich for helpful discussions and valuable comments in an early stage of the research. The research was supported in part by the AMSS-PolyU Joint Laboratory, the Research Grants Council of the Hong Kong Special Administrative Region, China (Project No. PolyU/RFS2324-5S03, PolyU/GRF15303022), and an internal grant of The Hong Kong Polytechnic University (Project ID: P0051154). 

\renewcommand{\refname}{\bf References}

\bibliographystyle{abbrv}
\bibliography{MCF_Vt}

\newpage
\appendix

\begin{center}
{\large\bf Supplementary material} \vspace{10pt}
\end{center}
\renewcommand{\thesection}{\Alph{section}}

This supplementary material complements the paper by providing some fundamental results that will be used in this paper and the detailed proofs of Lemmas \ref{Lemma:L2-dtH}--\ref{Lemma:defects}, which are omitted in the paper as they are based on the same techniques used in the proofs of Lemmas \ref{Lemma:W1p}--\ref{Lemma:L2}. 

\section{${\bf L^p}$ and ${\bf W^{1,p}}$ stability of interpolation operator 
on finite element functions of higher degree} \label{appendix_A}
\renewcommand{\theequation}{A.\arabic{equation}}
\renewcommand{\thesubsection}{A.\arabic{subsection}}

Let $\hat I_h^*:C(\Gamma_h[\x^*])\rightarrow S_h(\Gamma_h[\x^*])$ be the Lagrange interpolation operator onto the finite element space of polynomial degree $k\ge 1$, and let $S_h^r[\x^*]$ be the space of finite element functions of polynomial degree $r> k$. 
In the proof of Lemma \ref{Lemma:W1p} and \ref{Lemma:L2} we have used the following stability results:  
\begin{align}
\begin{aligned}
\|\hat I_h^* \varphi_h \|_{L^p(\Gamma_h[\x^*])} 
&\le
C \| \varphi_h\|_{L^p(\Gamma_h[\x^*])} &&\mbox{for}\,\,\,\varphi_h \in S_h^r[\x^*],
\,\,\, 1\le p\le \infty,  \\
\|\hat I_h^* \varphi_h \|_{W^{1,p}(\Gamma_h[\x^*])} 
&\le
C \| \varphi_h \|_{W^{1,p}(\Gamma_h[\x^*])} 
&&\mbox{for}\,\,\,\varphi_h \in S_h^r[\x^*] ,\,\,\, 1\le p\le \infty, 
\end{aligned}
\end{align}
which can be proved as follows: First, 
\begin{align*}
\|\hat I_h^* \varphi_h - \varphi_h \|_{L^p(\Gamma_h[\x^*])} 
\le
C h^{k+1} \|\nabla_{\Gamma_h[\x^*]}^{k+1}\varphi_h \|_{L^p(\Gamma_h[\x^*])} 
\le 
C \|\varphi_h \|_{L^p(\Gamma_h[\x^*])}  ,
\end{align*}
where the second to last inequality follows from the standard error estimate of Lagrange interpolation, and the last inequality follows from the inverse inequality of finite element functions in $S_h^r[\x^*]$. 
Then, by using the triangle inequality, we have 
\begin{align*}
\|\hat I_h^* \varphi_h \|_{L^p(\Gamma_h[\x^*])} 
\le
\|\hat I_h^* \varphi_h - \varphi_h \|_{L^p(\Gamma_h[\x^*])} 
+ \|\varphi_h \|_{L^p(\Gamma_h[\x^*])} 
\le 
C \|\varphi_h \|_{L^p(\Gamma_h[\x^*])} . 
\end{align*}
The $W^{1,p}$ stability of $\hat I_h^*$ can be proved similarly.

\section{${\bf L^p}$ and ${\bf W^{1,p}}$ stability of ${\bf L^2}$ projection onto $\Gamma_h[\x^*]$} \label{appendix_B}
\renewcommand{\theequation}{B.\arabic{equation}}
\renewcommand{\thesubsection}{B.\arabic{subsection}}

The $L^2$ projection operator $P_h:L^2(\Gamma_h[\x^*])\rightarrow S_h(\Gamma_h[\x^*])$ is defined as 
\begin{align*}
\int_{\Gamma_h[\x^*]} (\varphi - P_h\varphi) \chi_h = 0\quad\forall\, \varphi\in L^2(\Gamma_h[\x^*])\,\,\,\mbox{and}\,\,\, \chi_h \in S_h(\Gamma_h[\x^*]) .
\end{align*}
It is naturally bounded in the $L^2$ norm, i.e., $\|P_h\varphi\|_{L^2(\Gamma_h[\x^*])} \le C\|\varphi\|_{L^2(\Gamma_h[\x^*])} $ for $\varphi\in L^2(\Gamma_h[\x^*])$. Additionally, by the same proof of \cite[Lemma 6.1]{Thomee2006}, it is shown that the $L^2$ projection $P_h$ is bounded in the $L^\infty$ norm, i.e., $\|P_h\varphi\|_{L^\infty(\Gamma_h[\x^*])} \le C\|\varphi\|_{L^\infty(\Gamma_h[\x^*])} $.
Then, by the real interpolation between the $L^2$ and $L^\infty$ stability estimates, we derive the $L^p$ stability of $P_h$ for $2\le p\le \infty$, i.e.,  
$$
\|P_h\varphi\|_{L^p(\Gamma_h[\x^*])} \le C\|\varphi\|_{L^p(\Gamma_h[\x^*])} 
\quad\forall\, \varphi\in L^p(\Gamma_h[\x^*])\,\,\,\mbox{and}\,\,\, 2\le p\le \infty. 
$$
The $W^{1,\infty}$ stability of $P_h$ can be shown by using the inverse inequality of finite element functions, the $L^\infty$ stability of $P_h$, and the approximation property of the Lagrange interpolation, i.e.,
\begin{equation}\label{l2projectionW1infty}
  \begin{aligned}
\|P_h\varphi - I_h\varphi\|_{W^{1,\infty}(\Gamma_h[\x^*])} 
&\le Ch^{-1} \|P_h\varphi - I_h\varphi\|_{L^\infty(\Gamma_h[\x^*])} \\
&=Ch^{-1} \|P_h(\varphi - I_h\varphi)\|_{L^\infty(\Gamma_h[\x^*])} \\
&\le
Ch^{-1} \|\varphi - I_h\varphi\|_{L^\infty(\Gamma_h[\x^*])} \\
&\le 
C\|\varphi\|_{W^{1,\infty}(\Gamma_h[\x^*])} ,
\end{aligned}
\end{equation} 
which implies that 
$$
\|P_h\varphi\|_{W^{1,\infty}(\Gamma_h[\x^*])} 
\le \|P_h\varphi - I_h\varphi\|_{W^{1,\infty}(\Gamma_h[\x^*])} + \|I_h\varphi\|_{W^{1,\infty}(\Gamma_h[\x^*])} \le C\|\varphi\|_{W^{1,\infty}(\Gamma_h[\x^*])},
$$
where the last inequality uses the $W^{1,\infty}$ stability of the Lagrange interpolation. The $H^1$ stability of $P_h$ can be shown similarly by comparing $P_h$ with the standard $H^1$ projection $R_h: H^1(\Gamma_h[\x^*])\rightarrow S_h(\Gamma_h[\x^*]) $, i.e., 
\begin{equation}\label{l2projectionH1}
  \begin{aligned}
  \|P_h\varphi - R_h\varphi\|_{H^1(\Gamma_h[\x^*])} 
  &\le Ch^{-1} \|P_h\varphi - R_h\varphi\|_{L^2(\Gamma_h[\x^*])} \\
  &=Ch^{-1} \|P_h(\varphi - R_h\varphi)\|_{L^2(\Gamma_h[\x^*])} \\
  &\le
  Ch^{-1} \|\varphi - R_h\varphi\|_{L^2(\Gamma_h[\x^*])} \\
  &\le 
  C\|\varphi\|_{H^1(\Gamma_h[\x^*])},
  \end{aligned}
\end{equation}
  which implies that 
  $$
  \|P_h\varphi\|_{H^1(\Gamma_h[\x^*])} 
  \le \|P_h\varphi - R_h\varphi\|_{H^1(\Gamma_h[\x^*])} + \|R_h\varphi\|_{H^1(\Gamma_h[\x^*])} \le C\|\varphi\|_{H^1(\Gamma_h[\x^*])},
  $$
Then, by the complex interpolation between the $H^1$ and $W^{1,\infty}$ stability results, we have 
$$
\|P_h\varphi\|_{W^{1,p}(\Gamma_h[\x^*])} \le C\|\varphi\|_{W^{1,p}(\Gamma_h[\x^*])} \quad\mbox{for}\,\,\, 2\le p\le\infty. 
$$

\section{Proof of Lemma \ref{Lemma:L2-dtH}} \label{appendix_C}
\renewcommand{\theequation}{C.\arabic{equation}}
\renewcommand{\thesubsection}{C.\arabic{subsection}}

The proof of Lemma \ref{Lemma:L2-dtH} is similar as the proofs for Lemma \ref{Lemma:W1p} and Lemma \ref{Lemma:L2}. For the readers' convenience, we present a complete proof and divide it into several subsections. 

\subsection{Basic setting}\label{section:B1}
We define a discrete velocity $\tilde v_h^*$ on the smooth surface $\Gamma$ in the same way as \cite[Section 5.3]{Kovacs2018} and \cite[Section 7.3]{KLLP2017}. In particular, the surface $\Gamma$ is decomposed into curved elements which move with velocity $\tilde v_h^*$, which is defined by 
\begin{equation*}
    \tilde v_h^* \big(X_h^*(\cdot,t)^l,t\big) = \frac{\d}{\d t} (X_h^*(\cdot,t)^l ) .
\end{equation*}
Then the discrete material derivatives on the smooth surface $\Gamma$ and its interpolated surface $\Gamma_h[\x^*]$ are given by
\begin{align*}
\begin{aligned}
\partial_{t,h}^{\bullet} f &=\partial_{t} f + \tilde v_h^* \cdot \nabla f &&\mbox{for $f$ defined on $\Gamma$} , \\
\partial_{t,h}^{\bullet} \varphi_{h} &=\partial_{t} \varphi_{h} + \hat I_h^*v \cdot \nabla \varphi_{h} &&\mbox{for $\varphi_h$ defined on $\Gamma_h[\x^*]$}  , 
\end{aligned}
\end{align*}
where $\hat I_h^*v\in S_h[\x^*]$ is the interpolated velocity (i.e., velocity of the interpolated surface $\Gamma_h[\x^*]$). The two material derivatives above are related by   
\begin{equation}
\label{eq: lift and material derivatives}
\partial_{t,h}^{\bullet} \varphi_{h}^{l}=\big(\partial_{t,h}^{\bullet} \varphi_{h}\big)^{l} 
\,\,\,\mbox{for}\,\,\,\varphi_h\in S_h(\Gamma_h[\x^*]) . 
\end{equation}
It is known that (cf. \cite[Lemma 5.4]{Kovacs2018}) 
\begin{align}\label{tilde_v_h-v}
\|\tilde v_h^*-v\|_{L^\infty(\Gamma)}
+ h\|\nabla_{\Gamma}(\tilde v_h^*-v)\|_{L^\infty(\Gamma)} \le Ch^{k+1} . 
\end{align}
Moreover, the surface $\Gamma$ can be considered as moving with velocity $\tilde v_h^*$, and the following identities hold: 
\begin{align}
\label{transport-hatv1}
\frac{\d}{\d t} \int_{\Gamma} f 
&= \int_{\Gamma} (\partial_{t,h}^\bullet f+ f \nabla_{\Gamma}\cdot \tilde v_h^*) ,
\\[5pt] 
\frac{\d}{\d t} \int_{\Gamma} \nabla_{\Gamma}f\cdot\nabla_{\Gamma}g 
&= \int_{\Gamma} \big( \nabla_{\Gamma} \partial_{t,h}^\bullet f\cdot\nabla_{\Gamma}g +\nabla_{\Gamma} f\cdot\nabla_{\Gamma}\partial_{t,h}^\bullet g 
+\nabla_{\Gamma}f\cdot D_{\Gamma}\tilde v_h^*\nabla_{\Gamma}g\, \big)  ,
\label{transport-hatv2}
\end{align}
where $D_\Gamma \tilde v_h^* = 
(\nabla_{\Gamma}\cdot \tilde v_h^*) I_3
- (\nabla_{\Gamma} \tilde v_h^*+(\nabla_{\Gamma} \tilde v_h^*)^{\top}) $; see Lemma \ref{BilinearMA} and \cite[Lemma 3.1]{Dziuk-Lubich-Mansour-2011}. 
Since the two functions $u_h^*$ (defined on surface $\Gamma_h[\y^*]$) and $\hat u_h^*$  (defined on surface $\Gamma_h[\x^*]$) have the same nodal vector, it follows that 
\begin{align}\label{relation-dthHh}
(\partial_{t,h}^\bullet u_h^*)^\wedge 
=\partial_{t,h}^\bullet \hat u_h^* \,\,\,\mbox{(both with nodal vector $\dot {\bf u}^*$)} . 
\end{align}

Since problem \eqref{def-yHn} is essentially an ODE system, we can assume that $t_*\in(0,T]$ is the maximal time such that the solution of problem \eqref{def-yHn} satisfies the following estimate for $t\in[0,t_*]$:  
\begin{align}\label{ind-H-K}
&
\| \partial_{t,h}^\bullet \hat u_h^* - \hat I_h^*\partial_t^{\bullet}u\|_{W^{1,6}(\Gamma_h[\x^*])}
\le 1 . 
\end{align}
Under this condition, we shall prove that \eqref{L2-dtH*}--\eqref{L2-dtn*} hold for $t\in[0,t_*]$ (for some constants $h_0$ and $C$ which are independent of $t_*$). Then, since problem \eqref{def-yHn} is essentially an ODE system, the local well-posedness of the ODE system and the continuity of $\partial_{t,h}^\bullet \hat u_h^*$ in time imply that \eqref{L2-dtH*}--\eqref{L2-dtn*} hold for $t\in[0,t_*+\delta_h]$ with some $\delta_h>0$ and with $C$ replaced by $2C$. This implies that
$$
\| \partial_{t,h}^\bullet \hat u_h^* - \hat I_h^*\partial_t^{\bullet}u\|_{W^{1,6}(\Gamma_h[\x^*])}
\le Ch^{-\frac{2}{3}} \| \partial_{t,h}^\bullet \hat u_h^* - \hat I_h^*\partial_t^{\bullet}u\|_{H^1(\Gamma_h[\x^*])}
\le Ch^{k-\frac23}\quad\mbox{for}\,\,\, t\in[0,t_*+\delta]. 
$$
The above inequality implies that \eqref{ind-H-K} holds for $t\in[0,t_*+\delta_h]$ (when $k\ge 1$ and $h$ is smaller than some constant which is independent of $t_*$). This would imply that $t_*=T$ (otherwise $t_*$ is not the maximal time for \eqref{ind-H-K} to hold) and therefore completes the proof of Lemma \ref{Lemma:L2-dtH}.

\subsection{Reduction to estimates of $\partial_{t,h}^\bullet w$ and $g$}
For $\varphi_h \in S_h[\x^*]^4$ satisfying $\partial_{t,h}^{\bullet}\varphi_h=0$ on $\Gamma_h[\x^*]$ and accordingly $\partial_{t,h}^{\bullet}\varphi_h^l=0$ on $\Gamma$ as a result of \eqref{eq: lift and material derivatives}, we differentiate \eqref{KH.varphi} in time and using \eqref{transport-hatv1}--\eqref{transport-hatv2}. 
This yields that 
\begin{align}\label{K-x-H-varphi}
&{\bf K}(\x^*)\dot {\bf u}^*\cdot{\bm \varphi} \notag \\
&=
\int_{\Gamma} ((\partial_{t,h}^\bullet  w+\partial_{t,h}^\bullet u) \cdot \varphi_h^l 
+\nabla_{\Gamma} (\partial_{t,h}^\bullet  w+\partial_{t,h}^\bullet u) \cdot \nabla_{\Gamma}  \varphi_h^l) \notag \\
&\quad
+ 
\int_{\Gamma} ((w+ u) \cdot \varphi_h^l \nabla_{\Gamma}\cdot \tilde v_h^* +\nabla_{\Gamma} (w+ u) \cdot D_{\Gamma}\tilde v_h^*   \nabla_{\Gamma}  \varphi_h^l) 
- \bigg(\frac{\d}{\d t}{\bf K}(\x^*) \bigg)  {\bf u}^*\cdot{\bm \varphi} , 
\end{align} 
where
\begin{align*}
\bigg(\frac{\d}{\d t}{\bf K}(\x^*) \bigg)  {\bf u}^*\cdot{\bm \varphi}
=
& 
\int_{\Gamma_h[\x^*]} (\hat u_h^* \cdot \varphi_h \nabla_{\Gamma_h[\x^*]}\cdot \hat I_h^*v 
+\nabla_{\Gamma_h[\x^*]} \hat u_h^* \cdot D_{\Gamma_h[\x^*]}\hat I_h^*v \nabla_{\Gamma_h[\x^*]} \varphi_h) ,
\end{align*} 
which follows from the identities in Lemma \ref{BilinearMA}. 
Let $g\in H^1(\Gamma)^4$ be the solution of the following weak formulation: 
\begin{align}\label{def-g}
\int_{\Gamma} 
(g\cdot \varphi^l + \nabla_{\Gamma} g \cdot \nabla_{\Gamma}  \varphi^l) 
& =
\int_{\Gamma} (w\cdot \varphi^l \nabla_{\Gamma}\cdot \tilde v_h^* +\nabla_{\Gamma} w\cdot D_{\Gamma}\tilde v_h^*   \nabla_{\Gamma}  \varphi^l) 
\notag \\
&\quad
+\int_{\Gamma} u\cdot \varphi^l \nabla_{\Gamma}\cdot \tilde v_h^* 
-\int_{\Gamma_h[\x^*]} \hat u_h^*\cdot \varphi \nabla_{\Gamma_h[\x^*]}\cdot \hat I_h^*v   \notag \\ 
& \quad 
+ \int_{\Gamma} \nabla_{\Gamma} u \cdot D_{\Gamma}\tilde v_h^*   \nabla_{\Gamma}  \varphi^l   
-\int_{\Gamma_h[\x^*]} \nabla_{\Gamma_h[\x^*]} \hat u_h^* \cdot D_{\Gamma_h[\x^*]}\hat I_h^*v  \nabla_{\Gamma_h[\x^*]} \varphi  \notag \\
& =: L_1(\varphi^l) + L_2(\varphi^l) + L_3(\varphi^l) 
\qquad\forall\,\varphi^l\in H^1(\Gamma)^4 , 
\end{align} 
where $\varphi $ is the inverse lift of $\varphi^l$ from $\Gamma$ onto $\Gamma_h[\x^*]$. Then, substituting \eqref{def-g} into \eqref{K-x-H-varphi}, we have  
\begin{align*}
{\bf K}(\x^*)\dot {\bf u}^*\cdot{\bm \varphi} 
&=
\int_{\Gamma} ((\partial_{t,h}^\bullet w+\partial_{t,h}^\bullet u + g)\varphi_h^l 
+\nabla_{\Gamma} (\partial_{t,h} ^\bullet w+\partial_{t,h}^\bullet u + g) \cdot \nabla_{\Gamma}  \varphi_h^l) 
\quad\forall\, \varphi_h\in S_h(\Gamma_h[\x^*])^4, 
\end{align*} 
which implies that 
$\partial_{t,h}^\bullet \hat u_h^*$ (with nodal vector $\dot{\bf u}^*$) is the linear Ritz projection of $\partial_{t,h}^\bullet  w+\partial_{t,h}^\bullet  u + g$ onto $S_h(\Gamma_h[\x^*])^4$. As a result, we have (see \cite[Proof of Corollary 4.2]{Demlow-2009}, {inserting (4.4) into (3.3) and (3.5) therein}) 
\begin{align}
\|(\partial_{t,h} ^\bullet \hat u_h^*)^l - \partial_{t,h}^\bullet w-\partial_{t,h} ^\bullet u - g \|_{H^1(\Gamma)} 
&\le
Ch^{k+1} \|\partial_{t,h}^\bullet  w + \partial_{t,h}^\bullet u + g \|_{H^1(\Gamma)} \\
&\quad
+ C
\inf_{\chi_h\in S_h(\Gamma_h[\x^*])} \|\partial_{t,h}^\bullet  w + \partial_{t,h}^\bullet u + g - \chi_h^l \|_{H^1(\Gamma)} . \notag 
\end{align} 
and
\begin{align}
\|(\partial_{t,h} ^\bullet \hat u_h^*)^l - \partial_{t,h}^\bullet w-\partial_{t,h} ^\bullet u - g \|_{L^2(\Gamma)} 
&\le
Ch^{k+1} \|\partial_{t,h}^\bullet  w + \partial_{t,h}^\bullet u + g \|_{H^1(\Gamma)} \\
&\quad
+ Ch
\inf_{\chi_h\in S_h(\Gamma_h[\x^*])} \|\partial_{t,h}^\bullet  w + \partial_{t,h}^\bullet u + g - \chi_h^l \|_{H^1(\Gamma)} . \notag 
\end{align} 
Hence, by using the triangle inequality, we have
\begin{align}\label{dtHh-dtH-H1-1}
\|(\partial_{t,h} ^\bullet \hat u_h^*)^l - \partial_{t,h} ^\bullet u \|_{H^1(\Gamma)} 
&\le
C( \|\partial_{t,h}^\bullet  w \|_{H^1(\Gamma)} 
+ \|g \|_{H^1(\Gamma)} )
+ Ch^{k+1} \| \partial_{t,h}^\bullet u \|_{H^1(\Gamma)} \notag \\
&\quad 
+ C
\inf_{\chi_h\in S_h(\Gamma_h[\x^*])} \| \partial_{t,h}^\bullet u  - \chi_h^l \|_{H^1(\Gamma)}  . 
\end{align} 
and
\begin{align}\label{dtHh-dtH-L2-1}
\|(\partial_{t,h} ^\bullet \hat u_h^*)^l - \partial_{t,h} ^\bullet u \|_{L^2(\Gamma)} 
&\le
 \|\partial_{t,h}^\bullet  w \|_{L^2(\Gamma)} 
+ \|g \|_{L^2(\Gamma)} 
+ Ch (\|\partial_{t,h}^\bullet  w \|_{H^1(\Gamma)} + \|g \|_{H^1(\Gamma)})
 \notag \\
&\quad 
+ Ch^{k+1} \| \partial_{t,h}^\bullet u\|_{H^1(\Gamma)}  + Ch
\inf_{\chi_h\in S_h(\Gamma_h[\x^*])} \|\partial_{t,h}^\bullet u - \chi_h^l \|_{H^1(\Gamma)}  . 
\end{align} 
It remains to estimate the right-hand sides of \eqref{dtHh-dtH-H1-1} and \eqref{dtHh-dtH-L2-1}. 

\subsection{Estimate for $\|g\|_{H^1(\Gamma)}$}
We estimate $\|g\|_{H^1(\Gamma)}$ from the expression in \eqref{def-g}, where $L_1(\varphi^l) $ can be estimated by using boundedness of $\|\nabla_{\Gamma}\tilde v_h^*\|_{L^\infty(\Gamma)}$ as shown in \eqref{tilde_v_h-v} and the estimate of $\|w\|_{H^1(\Gamma)}$ in \eqref{L2-H1-w}, i.e., 
\begin{align*}
|L_1(\varphi^l) |
\le
C\|w\|_{H^1(\Gamma)} \|\nabla_{\Gamma}\tilde v_h^*\|_{L^\infty(\Gamma)} \|\varphi^l\|_{H^1(\Gamma)} 
\le
Ch^k \|\varphi^l\|_{H^1(\Gamma)} .
\end{align*}
Given the following established estimate (which follows from \cite[Proposition 2.7 and relations (2.1)-(2.12)]{Demlow-2009}): 
\begin{align}\label{tilde-v_p}
   & \|(\nabla _\Gamma \cdot \tilde v _h^*)^{-l} - \nabla_{\Gamma_h[\bf x^*]} \cdot \hat I_h^*v \|_{L^p(\Gamma_h[\bf x^*])} \notag\\
    &\sim \|\nabla _\Gamma \cdot \tilde v _h^* - (\nabla_{\Gamma_h[\bf x^*]} \cdot \hat I_h^*v)^l \|_{L^p(\Gamma)} \notag\\
    &\le \|\nabla _\Gamma \cdot (\tilde v _h^* - v) \|_{L^p(\Gamma)} + \|\nabla _\Gamma \cdot v - (\nabla_{\Gamma_h[\bf x^*]} \cdot \hat I_h^*v)^l \|_{L^p(\Gamma)} \notag\\
    & \le Ch^k ,
\end{align}
we apply Lemma \ref{Lemma:L2} and \eqref{W1p-Hh-H}, along with considering the estimates for surface perturbation via the lift map. Consequently, we have 
\begin{equation*}
  \begin{aligned}
    |L_2(\varphi^l)| 
    &\le \big|\int_{\Gamma} u \cdot \varphi ^l  \nabla _{\Gamma}\cdot \tilde v_h^* - \int_{\Gamma_h[\bf x ^*]}u^{-l} \cdot \varphi (\nabla_{\Gamma} \cdot \tilde v_h^*)^{-l}\big|\\
    & \quad \, + \big| \int_{\Gamma_{h}[\bf x^*]} u^{-l}\cdot \varphi (\nabla_{\Gamma} \cdot \tilde v_h^*)^{-l} - \int_{\Gamma_{h}[\bf x^*]} u^{-l}\cdot \varphi (\nabla_{\Gamma_h[\bf x^*]} \cdot \hat I_h^*v)\big|\\
    & \quad \, + \big| \int_{\Gamma_{h}[\bf x^*]} u^{-l}\cdot \varphi (\nabla_{\Gamma_h[\bf x^*]} \cdot \hat I_h^*v) - \int_{\Gamma_{h}[\bf x^*]} \hat u_h^* \cdot \varphi (\nabla_{\Gamma_h[\bf x^*]} \cdot \hat I_h^*v) \big|\\
    & \le C h^k  \|\varphi^l\|_{L^2(\Gamma)} . 
  \end{aligned}
\end{equation*}
Similarly, by employing the result of Lemma \ref{Lemma:W1p} and following a process analogous to the one described above, we obtain 
\begin{equation*} 
  \begin{aligned}
    |L_3(\varphi^l)| \le Ch^k \|\varphi^l\|_{H^1(\Gamma)}
  \end{aligned}
\end{equation*}
Therefore, by substituting $\varphi^l=g$ into \eqref{def-g} and utilizing the above-established estimates of $L_1(\varphi^l)$, $L_2(\varphi^l)$ and $L_3(\varphi^l)$, we obtain 
\begin{align}\label{H1-estimate-g} 
\|g\|_{H^1(\Gamma)}  \le Ch^k  . 
\end{align}

\subsection{Estimate for $\|g\|_{L^2(\Gamma)}$}

Next, we estimate $\|g\|_{L^2(\Gamma)}$ via a duality argument. Specifically, we define $\varphi^l\in H^2(\Gamma)^4$ to be the solution of 
\begin{align} \label{PDE-varphi-g}
\varphi^l - \Delta_{\Gamma}  \varphi^l = g \,\,\,\mbox{on}\,\,\, \Gamma , 
\end{align} 
which satisfies the following standard $H^2$ regularity estimate: 
\begin{align}\label{H2-varphi-w} 
\|\varphi^l\|_{H^2(\Gamma)} \le C \|g\|_{L^2(\Gamma)} .
\end{align}
By testing \eqref{PDE-varphi-g} with $g$ and estimating the right-hand side of \eqref{def-g}, we shall prove the following result:  
\begin{align} \label{g-remain}
\|g\|_{L^2(\Gamma)}^2 
= \int_{\Gamma} 
(g\varphi^l + \nabla_{\Gamma} g \cdot \nabla_{\Gamma}  \varphi^l) 
\le 
Ch^{k+1} \|\varphi^l\|_{H^2(\Gamma)} .
\end{align} 
This, together with \eqref{H2-varphi-w}, would imply the following estimate: 
\begin{align}\label{L2-estimate-g} 
\|g\|_{L^2(\Gamma)} 
\le
Ch^{k+1} .
\end{align} 
It remains to prove the inequality in \eqref{g-remain}. This is achieved by estimating the right-hand side of \eqref{def-g} as follows. 

Firstly, $| L_1(\varphi^l) |$ can be estimated by using integration by parts (after approximating $\tilde v_h^*$ by $v$ in the integrand), i.e.,  
\begin{align*} 
| L_1(\varphi^l) |
& =
\int_{\Gamma} (w\cdot \varphi^l \nabla_{\Gamma}\cdot \tilde v_h^* +\nabla_{\Gamma} w\cdot D_{\Gamma}\tilde v_h^*   \nabla_{\Gamma}  \varphi^l) \\
& =
\int_{\Gamma} (w\cdot \varphi^l \nabla_{\Gamma}\cdot \tilde v_h^* +\nabla_{\Gamma} w\cdot D_{\Gamma}v   \nabla_{\Gamma}  \varphi^l) 
+ 
\int_{\Gamma}\nabla_{\Gamma} w\cdot D_{\Gamma}(\tilde v_h^* - v)  \nabla_{\Gamma}  \varphi^l  \\
&=
\int_{\Gamma} [w\cdot \varphi^l \nabla_{\Gamma}\cdot \tilde v_h^* - w \cdot(\nabla_{\Gamma} \cdot (D_{\Gamma}v   \nabla_{\Gamma}  \varphi^l)) ] 
+ 
\int_{\Gamma}\nabla_{\Gamma} w\cdot D_{\Gamma}(\tilde v_h^* - v)  \nabla_{\Gamma}  \varphi^l  \\
&\le
C\|w\|_{L^2(\Gamma)} \| \nabla_{\Gamma} \cdot \tilde v_h^*\|_{L^\infty(\Gamma)}\|\varphi^l\|_{H^{2}(\Gamma)} 
+C\|w\|_{L^2(\Gamma)}\|\varphi^l\|_{H^{2}(\Gamma)} \\
&\quad 
+ C\|w\|_{H^1(\Gamma)} \|D_{\Gamma}(\tilde v_h^* - v) \|_{L^\infty} 
\|\varphi^l\|_{H^1(\Gamma)} \\
&\le C\|w\|_{L^2(\Gamma)} \|\varphi^l\|_{H^{2}(\Gamma)} 
+  Ch\|w\|_{H^1(\Gamma)} \|\varphi^l\|_{H^{2}(\Gamma)} 
\quad
\mbox{(here \eqref{tilde_v_h-v} is used)} \\
&\le
Ch^{k+1} \|\varphi^l\|_{H^{2}(\Gamma)} , 
\end{align*} 
where we have used the estimate of $\|w\|_{L^2(\Gamma)}$ and $\|w\|_{H^1(\Gamma)} $ in \eqref{L2-H1-w}. 

Secondly, $| L_2(\varphi^l) |$ can be estimated by defining and utilizing intermediate surfaces between $\Gamma_h[\x^*]$ to $\Gamma$. Namely, we define 
$$
x^{\theta} := (1-\theta)x^{l}+\theta x\,\,\,\mbox{for}\,\,\, x\in\Gamma_h[\x^*]
\quad\mbox{and}\quad 
\Gamma^{\theta}:=(1-\theta)\Gamma+\theta\Gamma_h[\x^*]
=\{x^\theta:x\in\Gamma_h[\x^*]\} .
$$ 
Let $u^\theta$, $\varphi^\theta$, $\tilde v_h^\theta$ and $e^\theta$ be functions on $\Gamma^\theta$ defined by 
\begin{align*}
\begin{aligned}
&u^\theta(x^\theta) = (1-\theta)u(x^l) +\theta \hat u_h^*(x) , \\
&\varphi^\theta(x^\theta) = \varphi(x) = \varphi^l(x^l) , \\ 
&\tilde v_h^\theta(x^\theta) = (1-\theta)\tilde v_h^*(x^l) +\theta \hat I_h^*v(x) , \\
&e^\theta(x^\theta) = x-x^l 
&&\mbox{for}\,\,\, x \in\Gamma_h[\x^*] . 
\end{aligned}
\end{align*}
Then the surface $\Gamma^\theta$ moves with velocity $e^\theta$ (as $\theta$ changes) and 
$$
\partial_{\theta}^{\bullet} u^{\theta} (x^\theta) = 
\hat u_h^*(x) - u(x^l) , 
\quad
\partial_{\theta}^{\bullet} \varphi^{\theta} =0 
\quad\mbox{and}\quad
\partial_{\theta}^{\bullet} \tilde v_h^\theta(x^\theta)  = 
\hat I_h^*v(x) - \tilde v_h^*(x^l) .
$$
From \eqref{tilde_v_h-v} we conclude that 
\begin{equation}\label{v^*_t}
  \begin{aligned}
    |\partial_{\theta}^{\bullet} \tilde v_h^\theta(x^\theta) |
    = | \hat I_h^*v(x) - \tilde v_h^*(x^l) | 
    \le | \hat I_h^*v(x) - v(x^l) |
    + | v(x^l) - \tilde v_h^*(x^l) | 
    \le
    Ch^{k+1} .
  \end{aligned}
\end{equation}
Since
\begin{equation*}
  \begin{aligned}
    \|\nabla_{\Gamma_h[\mathbf{x}^*]} x^\theta - \nabla_{\Gamma_h[\mathbf{x}^*]} \mathrm{id}\|_{L^p(\Gamma_h[\mathbf{x}^*])} &= (1 - \theta) \|\nabla_{\Gamma_h[\mathbf{x}^*]}e^1\|_{L^p(\Gamma_h[\mathbf{x}^*])} 
    \le Ch^k \,\,\,\mbox{for}\,\,\, 2\le p\le \infty,
  \end{aligned}
\end{equation*}
it follows that the \(L^p(\Gamma^\theta)\) norm of \(f\) is equivalent to the \(L^p(\Gamma_h[\mathbf{x}^*])\) norm of \(f \circ X^\theta\). The result of Lemma \ref{Lemma:L2} and the norm equivalence between $\Gamma^\theta$ and $\Gamma_h[\x^*]$ imply that 
\begin{equation}\label{H^*Lp_t}
\|\partial_{\theta}^{\bullet} u^\theta(x^\theta)\|_{L^2(\Gamma^\theta)} \le C \|\hat u_h^* - u^{-l}\|_{L^2(\Gamma_h[\mathbf{x}^*])} \le Ch^{k+1} 
   .
\end{equation}
Then $|L_{2} (\varphi^l)|$ can be rewritten as follows:
\begin{align*} 
|L_2(\varphi^l) |
& =
\bigg| \int_{\Gamma} u\cdot\varphi^l \nabla_{\Gamma}\cdot \tilde v_h^* 
-\int_{\Gamma_h[\x^*]} \hat u_h^*\cdot \varphi \nabla_{\Gamma_h[\x^*]}\cdot \hat I_h^*v \bigg|  \\
&=
\bigg|\int_0^1\frac{\d}{\d \theta} \int_{\Gamma^\theta} u^\theta \cdot \varphi^\theta \nabla_{\Gamma^\theta }\cdot \tilde v_h^\theta  \,\d \theta \bigg| \\
&=
\bigg|\int_0^1\int_{\Gamma^\theta} 
[u^\theta \cdot \varphi^\theta (\nabla_{\Gamma^\theta }\cdot \tilde v_h^\theta)
(\nabla_{\Gamma^\theta} \cdot e^{\theta}) 
+\partial_{\theta}^\bullet u^\theta \cdot\varphi^\theta  (\nabla_{\Gamma^\theta }\cdot \tilde v_h^\theta) ]\d\theta \\
&\quad 
+\int_0^1\int_{\Gamma^\theta} u^\theta \cdot \varphi^\theta  \big( \nabla_{\Gamma^\theta }\cdot \partial_{\theta}^\bullet\tilde v_h^\theta - {\rm tr}(\nabla_{\Gamma^{\theta}} e^{\theta} \nabla_{\Gamma^{\theta}} \tilde v_h^\theta) +{\rm tr}[n^{\theta}(n^{\theta})^{\top}(\nabla_{\Gamma^{\theta}} e^{\theta})^{\top} \nabla_{\Gamma^{\theta}} \tilde v_h^\theta] \big) \,\d \theta \bigg| \\
& =: |L_{21} (\varphi^l) + L_{22}(\varphi^l) + L_{23}(\varphi^l) + L_{24}(\varphi^l) + L_{25}(\varphi^l)|. 
\end{align*} 
The term $L_{22} (\varphi^l)$ can be estimated by utilizing \eqref{H^*Lp_t} and $\|\nabla_{\Gamma^\theta }\cdot \tilde v_h^\theta\|_{L^\infty(\Gamma^\theta)}\le C$. The latter result follows from \eqref{tilde_v_h-v} and $\|\nabla_{\Gamma_h[\x^*]}\cdot \hat I_h v\|_{L^\infty(\Gamma_h[\x^*])}\le C$, and the equivalence of norms on $\Gamma^\theta$, $\theta\in[0,1]$. These results, and the equivalence between $\|\varphi^\theta\|_{L^2(\Gamma^\theta)}$ and $\|\varphi^l\|_{L^2(\Gamma)}$, lead to the following estimate of $L_{22} (\varphi^l)$:  
\begin{equation*}
  \begin{aligned}
    |L_{22} (\varphi^l)|&\le Ch^{k+1} \|\varphi^l\|_{L^2(\Gamma)}.
  \end{aligned}
\end{equation*}
The term $L_{21} (\varphi^l)$ can be rewritten as follows:
\begin{equation*}
  \begin{aligned}
    |L_{21} (\varphi^l)| & = \bigg| \int_0^1 \Big[\int_{\Gamma^\theta} u^\theta \cdot \varphi^\theta  (\nabla_{\Gamma^\theta} \cdot \tilde v_h^\theta)(\nabla_{\Gamma^\theta}\cdot e^\theta) -   \int_{\Gamma} u \cdot\varphi^l (\nabla_{\Gamma}\cdot \tilde v_h^*)(\nabla_{\Gamma}\cdot e^1)\Big] \,\d \theta  \\
    & \quad \, + \int_\Gamma u \cdot \varphi^l (\nabla_{\Gamma}\cdot \tilde v_h^* - \nabla_{\Gamma}\cdot v)(\nabla_\Gamma e^1)+ \int_\Gamma u \cdot \varphi^l (\nabla_{\Gamma}\cdot v)(\nabla_\Gamma e^1) \bigg|\\
    & =: |L_{211} (\varphi^l) + L_{212}(\varphi^l) + L_{213}(\varphi^l)|.
  \end{aligned}
\end{equation*}
Since 
\begin{equation*}
  \begin{aligned}
    L_{211} (\varphi^l) 
    &= \int_{0}^\theta \frac{\d}{\d\alpha} \int_{\Gamma^\alpha} u^\alpha \varphi^\alpha (\nabla_{\Gamma^\alpha}\cdot \tilde v_h^\alpha)(\nabla_{\Gamma^\alpha}\cdot e^\alpha) \,\d \alpha \\
    & = \int_{0}^\theta \int_{\Gamma^\alpha}\big[u^\alpha \varphi^\alpha(\nabla_{\Gamma^\alpha}\cdot \tilde v_h^\alpha)(\nabla_{\Gamma^\alpha}\cdot e^\alpha)(\nabla_{\Gamma^\alpha}\cdot e^\alpha)  + \partial_\alpha^\bullet u^\alpha \varphi^\alpha(\nabla_{\Gamma^\alpha}\cdot \tilde v_h^\alpha)(\nabla_{\Gamma^\alpha}\cdot e^\alpha)\\
    & \qquad\qquad\quad + u^\alpha \varphi^\alpha\partial_\alpha^\bullet (\nabla_{\Gamma^\alpha}\cdot \tilde v_h^\alpha)(\nabla_{\Gamma^\alpha}\cdot e^\alpha) + u^\alpha \varphi^\alpha(\nabla_{\Gamma^\alpha}\cdot \tilde v_h^\alpha)\partial_\alpha^\bullet (\nabla_{\Gamma^\alpha}\cdot e^\alpha)\big]\,\d \alpha ,
  \end{aligned}
\end{equation*}
where $\partial_\alpha^\bullet (\nabla_{\Gamma^\alpha}\cdot \tilde v_h^\alpha)$ and $\partial_\alpha^\bullet (\nabla_{\Gamma^\alpha}\cdot e^\alpha)$ can be further expressed as follows using relation \eqref{dt-grad-f}: 
\begin{align*} 
\partial_\alpha^\bullet (\nabla_{\Gamma^\alpha}\cdot \tilde v_h^\alpha)
&=(\nabla_{\Gamma^\alpha}\cdot \partial_\alpha^\bullet\tilde v_h^\alpha)-(\nabla_{\Gamma^\alpha} e^\alpha - n_{\Gamma^\alpha} n_{\Gamma^\alpha}^\top(\nabla_{\Gamma^\alpha} e^\alpha)^\top)\nabla_{\Gamma^\alpha} \tilde v_h^\alpha \\
\partial_\alpha^\bullet (\nabla_{\Gamma^\alpha}\cdot e^\alpha)
&=(\nabla_{\Gamma^\alpha}\cdot \partial_\alpha^\bullet e^\alpha)-(\nabla_{\Gamma^\alpha} e^\alpha - n_{\Gamma^\alpha} n_{\Gamma^\alpha}^\top(\nabla_{\Gamma^\alpha} e)^\top)\nabla_{\Gamma^\alpha}e^\alpha . 
\end{align*} 
By using the norm equivalence between $\Gamma$ and $\Gamma^\alpha$, and the interpolation error estimates 
\begin{align}\label{interpl-e}
\begin{aligned}
\|e^\alpha\|_{L^\infty(\Gamma^\alpha)} &\sim \|{{{\rm id}_{\Gamma_h[\x^*]}}}^l - {{{\rm id}_{\Gamma_h[\x^*]}}}\|_{L^\infty(\Gamma_h[\x^*])}\le Ch^{k+1}  \\
\|\nabla_{\Gamma^\alpha} e^\alpha\|_{L^\infty(\Gamma^\alpha)} &\sim \|\nabla_{\Gamma_h[\x^*]} ({{{\rm id}_{\Gamma_h[\x^*]}}}^l - {{{\rm id}_{\Gamma_h[\x^*]}}})\|_{L^\infty(\Gamma_h[\x^*])}\le Ch^k , 
\end{aligned}
\end{align}  
as well as the estimates in \eqref{v^*_t}--\eqref{H^*Lp_t}, we can derive the following estimate of $L_{211} (\varphi^l)$:  
\begin{equation*}
  \begin{aligned}
     |L_{211} (\varphi^l)| &\le C h^{2k} \|\varphi^l\|_{L^2(\Gamma)} \le C h^{k+1} \|\varphi^l\|_{L^2(\Gamma)}.
  \end{aligned}
\end{equation*}
The term $L_{212} (\varphi^l)$ can be estimated by employing inequalities \eqref{tilde_v_h-v} and \eqref{interpl-e}, i.e., 
\begin{equation*}
  \begin{aligned}
    |L_{212} (\varphi^l)| &\le C h^{2k} \|\varphi^l\|_{L^2(\Gamma)} \le C h^{k+1} \|\varphi^l\|_{L^2(\Gamma)}.
  \end{aligned}
\end{equation*}
The term $L_{213} (\varphi^l)$ can be estimated with integration by parts, which allows us to transfer the gradient from $e^1$ to $H \varphi^l (\nabla_{\Gamma} \cdot v)$. This, along with \eqref{interpl-e}, leads to the following result:
\begin{equation*}
  \begin{aligned}
    |L_{213} (\varphi^l)| &\le C h^{k+1} \|\varphi^l\|_{H^1(\Gamma)}.
  \end{aligned}
\end{equation*}

By employing a similar method as that for estimating $L_{21}(\varphi^l)$, we can derive the following result: 
\begin{align*}
|L_{23}(\varphi^l) + L_{24}(\varphi^l) + L_{25}(\varphi^l)| &\le Ch^{k+1} \|\varphi^l\|_{H^1(\Gamma)}.  
\end{align*}
Combining the estimates for \( |L_{2i}(\varphi^l)| \) for \( i=1,\ldots,5 \), we have the following result: 
\begin{equation*}
  \begin{aligned}
    |L_{2}(\varphi^l)| &\le C h^{k+1} \|\varphi^l\|_{H^1(\Gamma)}.
  \end{aligned}
\end{equation*}

Thirdly, we can estimate $ |L_{3}(\varphi^l)|$ similarly as $ |L_{2}(\varphi^l)|$, by utilizing the intermediate surface $\Gamma^\theta$ between $\Gamma_h[\x^*]$ and $\Gamma$. This leads to the following result: 
\begin{equation*}
  \begin{aligned}
    |L_{3}(\varphi^l)| &\le C h^{k+1} \|\varphi^l\|_{H^2(\Gamma)}.
  \end{aligned}
\end{equation*}
Substituting these estimates for $L_i(\varphi^l)$ for $i =1, 2,3$ into \eqref{def-g}, we obtain 
\begin{align} \label{dual-estimate-g}
\bigg| \int_{\Gamma} 
(g\varphi^l + \nabla_{\Gamma} g \cdot \nabla_{\Gamma}  \varphi^l) \bigg| 
\le 
Ch^{k+1} \|\varphi^l\|_{H^2(\Gamma)} . 
\end{align} 
This proves \eqref{g-remain} and therefore completes the proof of \eqref{L2-estimate-g}.

\subsection{Estimate for $\|\partial_{t,h}^\bullet w\|_{H^1(\Gamma)}$}
\label{section:def-varphi_h}

To estimate $\|\partial_{t,h}^\bullet w\|_{H^1(\Gamma)}$, we need to consider the time derivative of \eqref{def-w} by extending $\varphi^l\in H^1(\Gamma)^4$ in time in such a way that 
\begin{align}\label{mat-varphi=0}
\partial_{t,h}^\bullet \varphi^l = (\partial_{t,h}^\bullet \varphi)^l =0 ,
\end{align}
where $\partial_{t,h}^\bullet \varphi^l$ denotes the material derivative with respect to velocity $\tilde v_h^*$. For $\y^\theta:=(1-\theta)\x^*+\theta \y^*$, let $\hat u_h^{\theta}$, $\hat \varphi_h^{\theta}$ and $\hat e_y^\theta$ be the finite element functions on the intermediate surface $\Gamma_h[\y^\theta]$ with nodal vectors ${\bf u}^*$, ${\bf P}_h\varphi$ and ${\bf e_y}:=\y^*-\x^*$, respectively. 
Then the surface $\Gamma_h[\y^\theta]$ moves with velocity $\hat e_y^{\theta} $ as $\theta\in[0,1]$ changes, and 
$$
\partial_\theta^{\bullet} \hat u_h^{\theta} = \partial_\theta^{\bullet} \hat \varphi_h^{\theta} = 0 .
$$
When $\theta$ is fixed and $t$ increases, the surface $\Gamma_h[\y^\theta]$ moves with velocity $v_h^\theta$, which is a finite element function on $\Gamma_h[\y^\theta]$ with nodal vector 
$$ 
{\bf v}^\theta = (1-\theta){\bf I}_h v + \theta {\bf v}^* . 
$$ 
Let $e_v^\theta$ be the finite element function on $\Gamma_h[\y^\theta]$ with nodal vector ${\bf v}^* - {\bf I}_h v$, which is the same as the nodal vector of $- \hat I_h^*(\hat H_h^* \hat n_h^* - \hat I_h^* H \hat I_h^* n)$. Hence, by using the $L^p$ and $W^{1,p}$ stability of $\hat I_h^*$ (see Appendix A), and the $W^{1,p}$ and $L^2$ estimates of the dynamic Ritz projection in Lemma \ref{Lemma:W1p} and Lemma \ref{Lemma:L2}, we can derive the following estimates:  
\begin{align}\label{ev-theta}
\|e_v^\theta\|_{L^2(\Gamma_h[\y^\theta])} \le Ch^{k+1}
\quad\mbox{and}\quad 
\|e_v^\theta\|_{W^{1,p}(\Gamma_h[\y^\theta])} 
\le Ch^{k} \,\,\,\mbox{for}\,\,\, 2\le p<\infty.  
\end{align}
For a function $f$ defined on $\Gamma_h[\y^\theta]$ (such as $f=\hat u_h^\theta$ and $f=\hat\varphi_h^\theta$), we denote by $\partial_{t,h}^\bullet f$ the material derivative with respect to the velocity $v_h^\theta$. 
By differentiating \eqref{def-w} in time we obtain  
\begin{align} \label{dth-w-elliptic}
&\int_{\Gamma} (\partial_{t,h}^\bullet w\cdot  \varphi^l+\nabla_{\Gamma} \partial_{t,h}^\bullet w \cdot \nabla_{\Gamma}  \varphi^l) \notag \\
&
= 
\frac{\d}{\d t} [ ({\bf K}(\x^*)-{\bf K}(\y^*)){\bf u}^*\cdot ({\bf P}_h\varphi) ]
-\int_{\Gamma} (w\cdot \varphi^l \nabla_{\Gamma}\cdot \tilde v_h^* 
+\nabla_{\Gamma} w \cdot D_{\Gamma}\tilde v_h^* \nabla_{\Gamma}  \varphi^l) \notag \\ 
&
= -\frac{\d}{\d t} \int_0^1 \frac{\d}{\d \theta} \int_{\Gamma_h[\y^\theta]} 
( \hat u_h^\theta \cdot \hat \varphi_h^\theta 
+ \nabla_{\Gamma_h[\y^\theta]} \hat u_h^\theta \cdot \nabla_{\Gamma_h[\y^\theta]} \hat \varphi_h^\theta ) \d\theta 
-\int_{\Gamma} (w\cdot\varphi^l \nabla_{\Gamma}\cdot\tilde v_h^* 
+\nabla_{\Gamma} w \cdot D_{\Gamma}\tilde v_h^* \nabla_{\Gamma}  \varphi^l) \notag \\
&
= -\frac{\d}{\d t} \int_0^1 \int_{\Gamma_h[\y^\theta]} 
( \hat u_h^\theta \cdot \hat \varphi_h^\theta 
\nabla_{\Gamma_h[\y^\theta]} \cdot \hat e_y^\theta
+ \nabla_{\Gamma_h[\y^\theta]} \hat u_h^\theta \cdot D_{\Gamma_h[\y^\theta]}\hat e_y^\theta\nabla_{\Gamma_h[\y^\theta]} \hat \varphi_h^\theta ) \d\theta \notag \\
&\quad\,
-\int_{\Gamma} (w\cdot \varphi^l \nabla_{\Gamma}\cdot\tilde v_h^* 
+\nabla_{\Gamma} w \cdot D_{\Gamma}\tilde v_h^* \nabla_{\Gamma}  \varphi^l) \notag \\
&
= 
-\int_0^1 \int_{\Gamma_h[\y^\theta]} \hat u_h^\theta \cdot \hat \varphi_h^\theta \nabla_{\Gamma_h[\y^\theta]} \cdot \hat e_y^\theta \, \nabla_{\Gamma_h[\y^\theta]} \cdot v_h^\theta  \d\theta \notag \\
&\quad\,
- \int_0^1 \int_{\Gamma_h[\y^\theta]} 
\partial_{t,h}^\bullet \hat u_h^\theta \cdot \hat \varphi_h^\theta \nabla_{\Gamma_h[\y^\theta]} \cdot \hat e_y^\theta \d\theta \notag\\
&\quad\,
- \int_0^1 \int_{\Gamma_h[\y^\theta]} \hat u_h^\theta \cdot \partial_{t,h}^\bullet\hat \varphi_h^\theta \nabla_{\Gamma_h[\y^\theta]} \cdot \hat e_y^\theta \d\theta \notag \\
&\quad\,
- \int_0^1 \int_{\Gamma_h[\y^\theta]} \hat u_h^\theta \cdot \hat \varphi_h^\theta 
\Big( \nabla_{\Gamma_h[\y^\theta]} \cdot \partial_{t,h}^\bullet  \hat e_y^\theta 
- {\rm tr}(\nabla_{\Gamma_h[\y^\theta]} v_h^{\theta} \nabla_{\Gamma_h[\y^\theta]} \hat e_y^\theta) \notag \\
&\qquad\qquad\qquad\qquad\qquad
+ {\rm tr}[n_{\Gamma_h[\y^\theta]} n_{\Gamma_h[\y^\theta]} ^{\top}(\nabla_{\Gamma_h[\y^\theta]} v_h^{\theta})^{\top} \nabla_{\Gamma_h[\y^\theta]} \hat e_y^\theta]  
\Big)  \d\theta \notag \\
&\quad\,
- \int_0^1 \int_{\Gamma_h[\y^\theta]} \nabla_{\Gamma_h[\y^\theta]} \hat u_h^\theta \cdot D_{\Gamma_h[\y^\theta]}\hat e_y^\theta\nabla_{\Gamma_h[\y^\theta]} \hat \varphi_h^\theta 
\nabla_{\Gamma_h[\y^\theta]}\cdot v_h^\theta \d\theta \notag \\
&\quad\, 
-\int_0^1 \int_{\Gamma_h[\y^\theta]} 
\partial_{t,h}^\bullet \nabla_{\Gamma_h[\y^\theta]} \hat u_h^\theta \cdot D_{\Gamma_h[\y^\theta]}\hat e_y^\theta\nabla_{\Gamma_h[\y^\theta]} \hat \varphi_h^\theta \d\theta
 \notag \\ 
&\quad\, 
-\int_0^1 \int_{\Gamma_h[\y^\theta]} 
\nabla_{\Gamma_h[\y^\theta]} \hat u_h^\theta \cdot  \partial_{t,h}^\bullet D_{\Gamma_h[\y^\theta]}\hat e_y^\theta\nabla_{\Gamma_h[\y^\theta]} \hat \varphi_h^\theta \d\theta  \notag \\
&\quad\, 
-\int_0^1 \int_{\Gamma_h[\y^\theta]} 
\nabla_{\Gamma_h[\y^\theta]} \hat u_h^\theta \cdot  D_{\Gamma_h[\y^\theta]}\hat e_y^\theta \partial_{t,h}^\bullet \nabla_{\Gamma_h[\y^\theta]} \hat \varphi_h^\theta \d\theta \notag \\
&\quad\,
-\int_{\Gamma} (w\cdot \varphi^l \nabla_{\Gamma}\cdot\tilde v_h^* 
+\nabla_{\Gamma} w \cdot D_{\Gamma}\tilde v_h^* \nabla_{\Gamma}  \varphi^l) \notag \\
&=:
\sum_{j=1}^9 W_j .
\end{align} 
Since $\partial_{t,h}^\bullet \hat u_h^\theta$ has nodal vector $\dot\u^*$, it follows from \eqref{ind-H-K} that $\|\partial_{t,h}^\bullet \hat u_h^\theta\|_{W^{1,6}(\Gamma_h[\y^\theta])}\le C$.  
By using the estimate of $\|\nabla_{\Gamma_h[\y^\theta]} \hat e_y^\theta\|_{L^2(\Gamma_h[\y^\theta])} \sim \|\nabla_{\Gamma_h[\x^*]} ({{{\rm id}_{\Gamma_h[\x^*]}}}-\hat y_h^*)\|_{L^2(\Gamma_h[\x^*])}$ in Lemma \ref{Lemma:W1p}, as well as inequalities \eqref{L2-H1-w}, \eqref{tilde_v_h-v} and \eqref{ind-H-K}, we can derive the following estimate:  
$$
 |W_1|  +  |W_2  |   +  |W_5| +  |W_9| \le Ch^k \| \varphi^l \|_{H^{1}(\Gamma)} , 
$$

The term $|W_6|$ can be estimated by using relation \eqref{dt-grad-f}, which allows us to write $\partial_{t,h}^\bullet \nabla_{\Gamma_h[\y^\theta]} \hat u_h^\theta$ as  
$$ 
\partial_{t,h}^\bullet \nabla_{\Gamma_h[\y^\theta]} \hat u_h^\theta
=\nabla_{\Gamma_h[\y^\theta]} \partial_{t,h}^\bullet\hat u_h^\theta
- (\nabla_{\Gamma_h[\y^\theta]} v_h^\theta - n_{\Gamma_h[\y^\theta]}n_{\Gamma_h[\y^\theta]}^\top(\nabla_{\Gamma_h[\y^\theta]} v_h^\theta)^\top)\nabla_{\Gamma_h[\y^\theta]} \hat u_h^\theta . 
$$ 
This, together with Lemma \ref{Lemma:W1p} and \eqref{ind-H-K}, implies $\|\partial_{t,h}^\bullet \nabla_{\Gamma_h[\y^\theta]} \hat u_h^\theta\|_{L^6(\Gamma_h[\y^\theta])} \le C$ and therefore 
$$ 
|W_6| \le C\|\partial_{t,h}^\bullet \nabla_{\Gamma_h[\y^\theta]} \hat u_h^\theta\|_{L^6(\Gamma_h[\y^\theta])} 
\| D_{\Gamma_h[\y^\theta]}\hat e_y^\theta \|_{L^3(\Gamma_h[\y^\theta])}
\| \nabla_{\Gamma_h[\y^\theta]} \hat \varphi_h^\theta \|_{L^2(\Gamma_h[\y^\theta])} 
\le Ch^k \| \varphi^l \|_{H^1(\Gamma)} .
$$ 

The estimate of $|W_3|$ depends on the estimate of $\partial_{t,h}^\bullet\hat \varphi_h^\theta $. 
Since $(\partial_{t,h}^\bullet \varphi)^l = (\partial_{t,h}^\bullet \varphi^l) = 0$, as shown in \eqref{mat-varphi=0}, where $\varphi = (\varphi^l)^{-l}$ and $\partial_{t,h}^\bullet \varphi^l$ denotes the material derivative with respect to velocity $\tilde v_h^*$, differentiating equation $0=\int_{\Gamma_h[\mathbf{x}^*]} (P_h \varphi - \varphi) \chi_h $ with respect to time, we obtain the following relation for all functions $\chi_h\in S_h(\Gamma_h[\x^*])$ such that $\partial_{t,h}^\bullet\chi_h=0$: 
\begin{equation}\label{l2}
  \begin{aligned}
     0 &= \frac{\d}{\d t} \int_{\Gamma_h[\mathbf{x}^*]} (P_h \varphi - \varphi) \chi_h \\
     & = \int_{\Gamma_h[\mathbf{x}^*]} \partial_{t,h}^\bullet (P_h \varphi - \varphi) \chi_h + \int_{\Gamma_h[\mathbf{x}^*]} (P_h \varphi - \varphi) \chi_h \nabla_{\Gamma_h[\mathbf{x}^*]} \cdot v_h^* \\
     & = \int_{\Gamma_h[\mathbf{x}^*]} \partial_{t,h}^\bullet P_h \varphi \chi_h + \int_{\Gamma_h[\mathbf{x}^*]} (P_h \varphi - \varphi) \chi_h \nabla_{\Gamma_h[\mathbf{x}^*]} \cdot v_h^* . 
  \end{aligned}
\end{equation}
This implies the following estimate by the duality argument:
\begin{equation}\label{varphi_t}
    \|\partial_{t,h}^\bullet P_h \varphi\|_{L^2(\Gamma_h[\x^*])} 
    \le C \|P_h \varphi - \varphi\|_{L^2(\Gamma_h[\x^*])} \le C h\|\varphi^l\|_{H^1(\Gamma)} .
\end{equation}
From the definition of $\partial_{t,h}^\bullet \hat \varphi_h^\theta$ and inequality \eqref{varphi_t}, along with the norm equivalence on $\Gamma_h[\y^\theta]$ and $\Gamma_h[\bf x^*]$, we obtain
\begin{equation}\label{varphitheta_t}
  \begin{aligned}
    \|\partial_{t,h}^\bullet \hat \varphi_h^\theta\|_{L^2(\Gamma_h[\y^\theta])} = \|(\partial_{t,h}^\bullet P_h \varphi )\circ X_{h}^* \circ (Y_h^\theta)^{-1}\|_{L^2(\Gamma_h[\y^\theta])} \le C h\|\varphi^l\|_{H^1(\Gamma)},
  \end{aligned}
\end{equation}
where $Y^\theta_h: \Gamma_h^0 \rightarrow \Gamma_h[\y^\theta] $ is the finite element flow map with nodal vector $\y^\theta$. 
By utilizing \eqref{varphitheta_t} and the estimate of $\|\nabla_{\Gamma_h[\y^\theta]} \hat e_y^\theta\|_{L^2(\Gamma_h[\y^\theta])} \sim \|\nabla_{\Gamma_h[\x^*]} ({{{\rm id}_{\Gamma_h[\x^*]}}}-\hat y_h^*)\|_{L^2(\Gamma_h[\x^*])}$ in Lemma \ref{Lemma:W1p}, we obtain  
\begin{equation*}
  \begin{aligned}
    |W_3| \le C h^{k+1} \|\varphi^l\|_{H^1(\Gamma)}. 
  \end{aligned}
\end{equation*} 

In order to estimate $|W_8|$, we use relation \eqref{dt-grad-f} again to write $\partial_{t,h}^\bullet \nabla_{\Gamma_h[\y^\theta]} \hat \varphi_h^\theta$ as 
$$ 
\partial_{t,h}^\bullet \nabla_{\Gamma_h[\y^\theta]} \hat \varphi_h^\theta
= \nabla_{\Gamma_h[\y^\theta]} \partial_{t,h}^\bullet\hat \varphi_h^\theta
- (\nabla_{\Gamma_h[\y^\theta]} v_h^\theta - n_{\Gamma_h[\y^\theta]}n_{\Gamma_h[\y^\theta]}^\top(\nabla_{\Gamma_h[\y^\theta]} v_h^\theta)^\top)\nabla_{\Gamma_h[\y^\theta]} \hat \varphi_h^\theta . 
$$ 
where $\nabla_{\Gamma_h[\y^\theta]} \partial_{t,h}^\bullet\hat \varphi_h^\theta$ can be estimated by combining the inverse inequality with \eqref{varphitheta_t}, i.e., 
\begin{equation}\label{nablavarphitheta_t}
  \begin{aligned}
    \|\nabla_{\Gamma_h[\y^\theta]}\partial_{t,h}^\bullet \hat \varphi_h^\theta\|_{L^2(\Gamma_h[\y^\theta])} \le Ch^{-1}\|\partial_{t,h}^\bullet \hat \varphi_h^\theta\|_{L^2(\Gamma_h[\y^\theta])} \le C \|\varphi^l\|_{H^1(\Gamma)} . 
  \end{aligned}
\end{equation}
This, together with the estimate of $\|\nabla_{\Gamma_h[\y^\theta]} \hat e_y^\theta\|_{L^2(\Gamma_h[\y^\theta])} \sim \|\nabla_{\Gamma_h[\x^*]} ({{{\rm id}_{\Gamma_h[\x^*]}}}-\hat y_h^*)\|_{L^2(\Gamma_h[\x^*])}$ in Lemma \ref{Lemma:W1p}, leads to the following result:
\begin{equation*}
  \begin{aligned}
    |W_8| \le C h^k \|\varphi^l\|_{H^1(\Gamma)}.
  \end{aligned}
\end{equation*}

Since $\partial_{t,h}^\bullet {{{\rm id}_{\Gamma_h[\x^*]}}} = \hat I_h^* v = - \hat I_h^*(H n) = - \hat I_h^*(\hat I_h^*H \hat I_h^*n)$ and $\partial_{t,h}^\bullet \hat y_h^* = - \hat I_h^*(\hat H_h^* \hat n_h^*) $ on $\Gamma_h[\x^*]$ (the latter follows from the relation $\dot\y^*=-{\bf I}_h( {\bf H}^*\bullet {\bf n}^*) $ in \eqref{def-yHn}), the following identity holds: 
\begin{equation*}
  \begin{aligned}
    \partial_{t,h} ^\bullet \hat e_y^\theta  & = \Big[\frac{\d}{\d t} ((\hat y_h^* - {{{\rm id}_{\Gamma_h[\x^*]}}} )\circ X_h^*)\Big] \circ (Y_h^\theta)^{-1} = \partial_{t,h}^\bullet (\hat y_h^* - {{{\rm id}_{\Gamma_h[\x^*]}}} )  \circ X_h^* \circ (Y_h^\theta)^{-1}\\
    & = - (\hat I_h^*(\hat H_h^* \hat n_h^*) - \hat I_h^*(\hat I_h^*H\hat I_h^*n))\circ X_h^* \circ (Y_h^\theta)^{-1} . 
  \end{aligned}
\end{equation*}
By using the stability of $\hat I_h^*$ in the $H^1(\Gamma_h[\mathbf{x}^*])$ norm for products of finite element functions (see \cite[Lemma 5.3]{KLL-2021}) and Lemma \ref{Lemma:W1p}, we obtain:
\begin{equation}\label{eytheta_t}
  \begin{aligned}
    \|\partial_{t,h}^\bullet \hat e_y^\theta\|_{H^1(\Gamma_h[\y^\theta])} 
    &\le C \|\hat I_h^*(\hat H_h^* \hat n_h^*) - \hat I_h^*(\hat I_h^*H \hat I_h^* n)\|_{H^1(\Gamma_h[\mathbf{x}^*])} \\
    &\le C \|\hat H_h^* - \hat I_h^*H \|_{H^1(\Gamma_h[\mathbf{x}^*])}
    + C \|\hat n_h^* - \hat I_h^*n \|_{H^1(\Gamma_h[\mathbf{x}^*])} 
    \le Ch^k . 
  \end{aligned}
\end{equation} 
Using the inequality \eqref{eytheta_t} and the estimate of $\|\nabla_{\Gamma_h[\y^\theta]} \hat e_y^\theta\|_{L^2(\Gamma_h[\y^\theta])} \sim \|\nabla_{\Gamma_h[\x^*]} ({{{\rm id}_{\Gamma_h[\x^*]}}}-\hat y_h^*)\|_{L^2(\Gamma_h[\x^*])}$ in Lemma \ref{Lemma:W1p}, we derive that 
\begin{equation*}
  \begin{aligned}
    |W_4| + |W_7| \le Ch^k \|\varphi^l\|_{H^1(\Gamma)}. 
  \end{aligned}
\end{equation*}
As a result, substituting $\varphi^l= \partial_{t,h}^\bullet w$ into \eqref{dth-w-elliptic} yields 
\begin{align} \label{dth-w-H1-h}
 \| \partial_{t,h}^\bullet w \|_{H^1(\Gamma)} 
 \le 
Ch^k .  
\end{align}

\subsection{Proof of the ${\bf H^1}$ estimates in Lemma \ref{Lemma:L2-dtH}}

Now, substituting \eqref{H1-estimate-g}  and \eqref{dth-w-H1-h} into \eqref{dtHh-dtH-H1-1}, along with assumption \eqref{ind-H-K}, we obtain 
\begin{align*}
\|(\partial_{t,h} ^\bullet \hat u_h^*)^l - \partial_{t,h} ^\bullet u \|_{H^1(\Gamma)} 
&\le
Ch^k + C
\inf_{\chi_h\in S_h(\Gamma_h[\x^*])} \| \partial_{t,h}^\bullet u  - \chi_h^l \|_{H^1(\Gamma)}  \notag \\
&=
Ch^k  + C
\inf_{\chi_h\in S_h(\Gamma_h[\x^*])} \| \partial_{t}^\bullet u + (\tilde v_h^*-v)\cdot\nabla_\Gamma u - \chi_h^l \|_{H^1(\Gamma)}  \notag \\
&\le 
Ch^k  + C\| \partial_{t}^\bullet u - (I_h^*\partial_{t}^\bullet u)^l  \|_{H^1(\Gamma)} , \notag
\end{align*} 
where the last inequality follows from choosing $\chi_h=I_h^*\partial_{t}^\bullet u$ and utilizing $\|\tilde v_h^*-v\|_{H^1(\Gamma)}\le Ch^k$. The latter has been shown in \eqref{tilde_v_h-v}. Since $\partial_{t}^\bullet u$ is a smooth function on $\Gamma$, it follows that $\| \partial_{t}^\bullet  u - (I_h^*\partial_{t}^\bullet u)^l  \|_{H^1(\Gamma)}\le Ch^k$ and therefore 
\begin{align}\label{dtHh-dtH-H1-f}
\|(\partial_{t,h} ^\bullet \hat u_h^*)^l - \partial_{t}^\bullet u \|_{H^1(\Gamma)} 
&\le Ch^k \quad\mbox{for}\,\,\, t\in[0,t_*+\delta] . 
\end{align}

\subsection{Estimate for $\|\partial_{t,h}^\bullet w\|_{L^2(\Gamma)}$}

We estimate $\|\partial_{t,h}^\bullet w\|_{L^2(\Gamma)}$ via a duality argument. Specifically, we define $\varphi^l\in H^2(\Gamma)^4$ to be the solution of 
\begin{align} \label{PDE-varphi-dtw}
\varphi^l - \Delta_{\Gamma}  \varphi^l = \partial_{t,h}^\bullet w \,\,\,\mbox{on}\,\,\, \Gamma , 
\end{align} 
which satisfies the following standard $H^2$ regularity estimate: 
\begin{align}\label{H2-varphi-dtw} 
\|\varphi^l\|_{H^2(\Gamma)} \le C \|\partial_{t,h}^\bullet w\|_{L^2(\Gamma)} .
\end{align}
By testing \eqref{PDE-varphi-dtw} with $\partial_{t,h}^\bullet w$ and estimating the right-hand side of \eqref{dth-w-elliptic}, we shall prove the following result:  
\begin{align} \label{dtw-remain}
\|\partial_{t,h}^\bullet w\|_{L^2(\Gamma)}^2 
= \int_{\Gamma} 
(\partial_{t,h}^\bullet w\cdot \varphi^l + \nabla_{\Gamma} \partial_{t,h}^\bullet w \cdot \nabla_{\Gamma}  \varphi^l) 
\le 
Ch^{k+1} \|\varphi^l\|_{H^2(\Gamma)} .
\end{align} 
This, together with \eqref{H2-varphi-dtw}, would imply the following estimate: 
\begin{align}\label{L2-estimate-dtw} 
\|\partial_{t,h}^\bullet w\|_{L^2(\Gamma)} 
\le
Ch^{k+1} .
\end{align} 
It remains to prove the inequality in \eqref{dtw-remain}. This is achieved by estimating $W_j$, $j=1,\dots,9$, on the right-hand side of \eqref{dth-w-elliptic} as follows. 

The same as before (see Section \ref{section:def-varphi_h}), we denote by $\hat u_h^{\theta}$, $\hat \varphi_h^{\theta}$ and $\hat e_y^\theta$ the finite element functions on surface $\Gamma_h[\y^\theta]$ with nodal vectors ${\bf u}^*$, ${\bf P}_h\varphi$ and ${\bf e_y}:=\y^*-\x^*$, respectively, and denote $\hat\varphi_h^*=\hat\varphi_h^0=P_h\varphi \in S_h[\x^*]= S_h[\y^0]$. Then the surface $\Gamma_h[\y^\theta]$ moves with velocity $\hat e_y^{\theta} $ as $\theta\in[0,1]$ changes, and 
$$
\partial_\theta^{\bullet} \hat u_h^{\theta} = \partial_\theta^{\bullet} \hat \varphi_h^{\theta} = 0 .
$$
When $\theta$ is fixed and $t$ increases, the surface $\Gamma_h[\y^\theta]$ moves with velocity $v_h^\theta$, which is a finite element function on $\Gamma_h[\y^\theta]$ with nodal vector 
$$ 
{\bf v}^\theta = (1-\theta){\bf I}_h v + \theta {\bf v}^* . 
$$  
By comparing $\int_{\Gamma_h[\y^\theta]} \hat u_h^\theta \cdot  \hat \varphi_h^\theta \nabla_{\Gamma_h[\y^\theta]} \cdot \hat e_y^\theta \, \nabla_{\Gamma_h[\y^\theta]} \cdot v_h^\theta  \d\theta$ with its value at $\theta=0$, and using the Newton--Leibniz rule, we have 
\begin{align*} 
W_1
&
=\int_0^1 \int_{\Gamma_h[\y^\theta]} \hat u_h^\theta \cdot \hat \varphi_h^\theta \nabla_{\Gamma_h[\y^\theta]} \cdot \hat e_y^\theta \, \nabla_{\Gamma_h[\y^\theta]} \cdot v_h^\theta  \d\theta \\
& 
= \int_{0}^1 \int_0^\theta \frac{\d}{\d\alpha} \int_{\Gamma_h[\y^\alpha]}  \hat u^\alpha_h \cdot\hat \varphi_h^\alpha\nabla_{\Gamma_h[\y^\alpha]}\cdot \hat e_y^\alpha  \nabla _{\Gamma_h[\y^\alpha]}\cdot v_h^\alpha \,\d \alpha \,\d \theta 
&&\mbox{(Newton--Leibniz rule)} \\
& \quad\,+\int_{\Gamma_h[{\bf x ^*}]} \hat u_h^* \cdot \hat \varphi_h^*\nabla _{\Gamma_h[\bf x ^*]} \cdot \hat e_y \nabla _{\Gamma_h[\bf x ^*]} \cdot \hat I_h^* v &&\mbox{(value at $\theta=0$)} \\
&=:  W_{11} + W_{12} . 
\end{align*} 
Since $\partial_\alpha^\bullet\hat u_h^\alpha=\partial_\alpha^\bullet\hat e_y^\alpha=0$, $\partial_\alpha^\bullet\hat \varphi_h^\alpha=0$ and $\partial_\alpha^\bullet v_h^\alpha=e_v^\alpha$, it follows that, by using formula \eqref{dt-grad-f}, 
\begin{equation*}
  \begin{aligned}
    |W_{11}|  
    & \le C(\|\hat e_y^\alpha\|_{W^{1,6}(\Gamma_h[\y^\alpha])}^2 + \|\hat e_y^\alpha\|_{W^{1,6}(\Gamma_h[\y^\alpha])}\|\hat e_v^\alpha\|_{W^{1,6}(\Gamma_h[\y^\alpha])})
     \|\varphi^l\|_{H^1(\Gamma)} \\
    & \le Ch^{2k} \|\varphi^l\|_{H^1(\Gamma)} , 
  \end{aligned} 
\end{equation*} 
where we have used Lemma \ref{Lemma:W1p} and \eqref{ev-theta} in estimating $\|\hat e_y^\alpha\|_{W^{1,6}(\Gamma_h[\y^\alpha])}$ and  $\|\hat e_v^\alpha\|_{W^{1,6}(\Gamma_h[\y^\alpha])}$ in the last inequality, respectively. By estimating the surface perturbation error between $\Gamma_h[\bf x ^*]$ and $\Gamma$, and using integration by parts for the integral on $\Gamma$ to remove the gradient operator $\nabla_{\Gamma}$ from $(\hat e_y)^l$, along with error estimates of Lagrange interpolation, we have 
\begin{align*}
    |W_{12}| &\le \Big|\int_{\Gamma_h[\bf x ^*]} \hat u_h^* \cdot\hat \varphi_h^*  \nabla_{\Gamma_h[\bf x ^*]}\cdot \hat e_y \nabla_{\Gamma_h[\bf x ^*]}\cdot \hat I_h^* v - \int_{\Gamma} (\hat u_h^*)^l \cdot (\hat \varphi_h^*)^l \nabla_{\Gamma}\cdot (\hat e_y)^l \nabla_{\Gamma}\cdot (\hat I_h^* v)^l \Big| \\
    &  \quad \, + \Big|  \int_{\Gamma} (\hat u_h^*)^l \cdot (\hat \varphi_h^*)^l \nabla_{\Gamma}\cdot (\hat e_y)^l (\nabla_{\Gamma}\cdot (\hat I_h^* v)^l - \nabla_\Gamma \cdot v) \Big| + \Big| \int_{\Gamma} (\hat u_h^*)^l \cdot(\hat \varphi_h^*)^l \nabla_{\Gamma}\cdot (\hat e_y)^l \nabla_\Gamma \cdot v \Big|\\
    & \le Ch^{2k} \|\varphi^l\|_{H^1(\Gamma)} + Ch^{2k} \|\varphi^l\|_{H^1(\Gamma)}  + Ch^{k+1} \|\varphi^l\|_{H^1(\Gamma)} 
    \le Ch^{k+1} \|\varphi^l\|_{H^1(\Gamma)} . 
\end{align*}
This proves that
\begin{align}\label{estimate-W1}
|W_1|
& \le
Ch^{k+1}  \| \varphi^l \|_{H^{1}(\Gamma)} .
\end{align}

Similarly, by converting the integral from $\Gamma_h[\y^\theta]$ to $\Gamma$ and then applying integration by parts along with the induction assumption \eqref{ind-H-K}, we can derive the following estimate:  
\begin{align}
|W_2|
:= \Big|\int_0^1 \int_{\Gamma_h[\y^\theta]} 
\partial_{t,h}^\bullet \hat u_h^\theta \cdot \hat \varphi_h^\theta \nabla_{\Gamma_h[\y^\theta]} \cdot \hat e_y^\theta \d\theta \Big|
& \le
Ch^{k+1}  \| \varphi^l \|_{H^{1}(\Gamma)} .
\end{align}
By using \eqref{varphitheta_t}, we have 
\begin{align}
|W_3|
&\le
Ch^{k+1}  \| \varphi^l \|_{H^{1}(\Gamma)}  .
\end{align}
By using \eqref{l2projectionH1}, the first inequality in \eqref{varphi_t}, and the first equality in \eqref{varphitheta_t}, we have
\begin{equation*}
  \begin{aligned}
    \|\partial_{t,h}^\bullet \hat \varphi_h^\theta\|_{H^1(\Gamma_h[\x^*])} & \le Ch^{-1}\|\partial_{t,h}^\bullet \hat \varphi_h^\theta\|_{L^2(\Gamma_h[\x^*])}\\
    &\le C h^{-1} \|P_h \varphi - \varphi\|_{L^2(\Gamma_h[\x^*])} \le C h\|\varphi^l\|_{H^2(\Gamma)},
  \end{aligned}
\end{equation*}
which implies that
\begin{align}
  |W_8|&\le
  Ch^{k+1}  \| \varphi^l \|_{H^{2}(\Gamma)}  .
\end{align}
Since $\partial_{t,h}^\bullet \hat e^\theta_{y}$ has the same nodal vector as $\hat v_h^* - \hat I_h^*v$, by the similar method along with Lemma \ref{Lemma:L2} and inequality \eqref{w_L2_bound}, we can obtain the following estimates for $W_j$, $j=4,5,7,9$, i.e., 
\begin{align} 
|W_4| + |W_5| + |W_7| + |W_9| 
&\le 
Ch^{k+1} \|\varphi^l \|_{H^2(\Gamma)} .  
\end{align}
It remains to prove the following result:
\begin{align}\label{estimate-W6}
|W_{6}|
\le
Ch^{k+1}\|\varphi^l \|_{H^2(\Gamma)} . 
\end{align} 
This can be done similarly by converting the integral in the expression of $W_6$ from $\Gamma_h[\y^\theta]$ to $\Gamma_h[\x^*]$, and using the Newton--Leibniz formula, we have 
\begin{align}\label{def-W6}
    W_ 6 &= \int_0^1 \int_{\Gamma_h[\y^\theta]} \partial_{t,h}^{\bullet} \nabla_{\Gamma_h[\y^\theta]} \hat u^\theta_h \cdot D_{\Gamma_h[\y^\theta]} \hat e^\theta _y \nabla_{\Gamma_h[\y^\theta]} \hat \varphi_h^\theta \, \d \theta \notag\\
    & = \int_0^1 \int_0^\theta \frac{\d}{\d\alpha}\int_{\Gamma_h[\y^\alpha]} \partial_{t,h}^{\bullet} \nabla_{\Gamma_h[\y^\alpha]} \hat u_h^\alpha \cdot D_{\Gamma_h[\y^\alpha]} \hat e_y^\alpha \nabla_{\Gamma_h[\y^\alpha]} \hat \varphi_h^\alpha \,\d \alpha \,\d\theta \notag\\
    & \quad \, + \int_{\Gamma_h[\bf x^*]} \partial_{t,h}^{\bullet} \nabla_{\Gamma_h[\bf x^*]} \hat u_h^* \cdot D_{\Gamma_h[\bf x^*]} \hat e_y \nabla_{\Gamma_h[\bf x^*]} \hat \varphi_h^* \, \d\theta \notag\\
    & = :W_{61} + W_{62},
\end{align}
where $W_{61}$ can be rewritten as follows:
\begin{equation*}
  \begin{aligned}
    W_{61} & = \int_0^1 \int_0^\theta \int_{\Gamma_h ^\alpha} \partial_{t,h}^\bullet \nabla_{\Gamma_h[\y^\alpha]} \hat u_h^\alpha \cdot D_{\Gamma_h[\y^\alpha]} \hat  e_y^\alpha  \nabla _{\Gamma_h[\y^\alpha]} \hat \varphi_h^\alpha  \nabla_{\Gamma_h[\y^\alpha]} \cdot \hat e_y^\alpha \,\d\alpha \,\d\theta \\
    & \quad \, + \int_0^1 \int_0^\theta \int_{\Gamma_h[\y^\alpha] } \partial_{\alpha}^\bullet  (\partial_{t,h}^\bullet \nabla_{\Gamma_h[\y^\alpha]} \hat u_h^\alpha) \cdot D_{\Gamma_h[\y^\alpha]} \hat  e_y^\alpha  \nabla _{\Gamma_h[\y^\alpha]} \hat \varphi_h^\alpha  \,\d\alpha \,\d\theta\\
    & \quad\,+\int_0^1 \int_0^\theta \int_{\Gamma_h[\y^\alpha]} \partial_{t,h}^\bullet \nabla_{\Gamma_h[\y^\alpha]} \hat u_h^\alpha \cdot \partial_{\alpha}^\bullet (D_{\Gamma_h[\y^\alpha]} \hat  e_y^\alpha)  \nabla _{\Gamma_h[\y^\alpha]} \hat \varphi_h^\alpha  \,\d\alpha \,\d\theta\\
    & \quad \, + \int_0^1 \int_0^\theta \int_{\Gamma_h[\y^\alpha]} \partial_{t,h}^\bullet \nabla_{\Gamma_h[\y^\alpha]} \hat u_h^\alpha \cdot D_{\Gamma_h[\y^\alpha]} \hat  e_y^\alpha \partial_{\alpha}^\bullet ( \nabla _{\Gamma_h[\y^\alpha]} \hat \varphi_h^\alpha)  \,\d\alpha \,\d\theta\\
    & =: W_{611} + W_{612} + W_{613}+ W_{614}.
  \end{aligned}
\end{equation*}
Since the nodal vector of $\partial_{t,h}^\bullet \hat u^\alpha_h $ is $\dot\u^*$, which is independent of $\alpha$, it follows that $\partial_{\alpha}^\bullet\partial_{t,h}^\bullet \hat u^\alpha_h = 0$. Hence, by applying $\partial_{t,h}^\bullet$ and $\partial_{\alpha}^\bullet$ to $\nabla_{\Gamma_h[\y^\alpha]} \hat u^\alpha_h$ sequentially and using relation \eqref{dt-grad-f}, we obtain 
\begin{align*}
   \partial_{t,h}^\bullet \nabla_{\Gamma_h[\y^\alpha]} \hat u^\alpha_h &= \nabla_{\Gamma_h[\y^\alpha]}\partial_{t,h}^\bullet \hat u^\alpha_h - (\nabla_{\Gamma_h[\y^\alpha]} v_h^\alpha - n_{\Gamma_h[\y^\alpha]} n_{\Gamma_h[\y^\alpha] }^\top (\nabla_{\Gamma_h[\y^\alpha] }v_h^\alpha)^\top) \nabla_{\Gamma_h[\y^\alpha] }\hat u_h^\alpha \\
    \partial_{\alpha}^\bullet \partial_{t,h}^\bullet \nabla_{\Gamma_h[\y^\alpha]} \hat u^\alpha_h 
    & =  
    -(\nabla_{\Gamma_h[\y^\alpha]} \hat e_y^\alpha - n_{\Gamma_h[\y^\alpha]} n_{\Gamma_h[\y^\alpha]}^\top (\nabla_{\Gamma_h[\y^\alpha]} \hat e_y^\alpha)^\top ) \nabla_{\Gamma_h[\y^\alpha]}\partial_{t,h}^\bullet \hat u_h^\alpha\\
    & \quad \, -\partial_{\alpha}^\bullet (\nabla_{\Gamma_h[\y^\alpha]} v_h^\alpha - n_{\Gamma_h[\y^\alpha]}n_{\Gamma_h[\y^\alpha]}^\top (\nabla _{\Gamma_h[\y^\alpha] }v_h^\alpha)^\top) \nabla_{\Gamma_h[\y^\alpha] }\hat u_h^\alpha \\
    &\quad\, + (\nabla_{\Gamma_h[\y^\alpha]} v_h^\alpha - n_{\Gamma_h[\y^\alpha]}n_{\Gamma_h[\y^\alpha]}^\top (\nabla _{\Gamma_h[\y^\alpha] }v_h^\alpha)^\top) \\
    &\qquad\cdot (\nabla_{\Gamma_h[\y^\alpha]} \hat e_y^\alpha - n_{\Gamma_h[\y^\alpha]} n_{\Gamma_h[\y^\alpha]}^\top (\nabla_{\Gamma_h[\y^\alpha]} \hat e_y^\alpha)^\top ) \nabla_{\Gamma_h[\y^\alpha]}\hat u_h^\alpha 
\end{align*}
where
\begin{align*}
\partial_{\alpha}^\bullet (\nabla_{\Gamma_h[\y^\alpha]} v_h^\alpha)
&= \nabla_{\Gamma_h[\y^\alpha]} \partial_{\alpha}^\bullet  v_h^\alpha 
-(\nabla_{\Gamma_h[\y^\alpha]} \hat e_y^\alpha - n_{\Gamma_h[\y^\alpha]} n_{\Gamma_h[\y^\alpha]}^\top (\nabla_{\Gamma_h[\y^\alpha]} \hat e_y^\alpha)^\top )  \nabla_{\Gamma_h[\y^\alpha]} v_h^\alpha \\
\partial_{\alpha}^\bullet n_{\Gamma_h[\y^\alpha]} 
& =  - (\nabla_{\Gamma_h[\y^\alpha]} \hat e_y^\alpha) n_{\Gamma_h[\y^\alpha]} .
\end{align*}
Since $\partial_{\alpha}^\bullet  v_h^\alpha = e_v^\alpha $, which has been estimated in \eqref{ev-theta}, it follows that 
\begin{align*}
&\| \partial_{\alpha}^\bullet (\nabla_{\Gamma_h[\y^\alpha]} v_h^\alpha) \|_{L^6(\Gamma_h[\y^\alpha])} 
+ \| \partial_{\alpha}^\bullet n_{\Gamma_h[\y^\alpha]} \|_{L^6(\Gamma_h[\y^\alpha])} \\
& \le C
\|e_v^\alpha \|_{W^{1,6}(\Gamma_h[\y^\alpha])} + C\|\hat e_y^\alpha \|_{W^{1,6}(\Gamma_h[\y^\alpha])}
\le Ch^k . 
\end{align*}
By using this result and the estimate of $\|\nabla_{\Gamma_h[\y^\alpha]} \hat e_y^\alpha \|_{L^p(\Gamma_h[\y^\alpha])} \sim \|\nabla_{\Gamma_h[\x^*]} ({{{\rm id}_{\Gamma_h[\x^*]}}}-\hat y_h^*)\|_{L^p(\Gamma_h[\x^*])}$ in Lemma \ref{Lemma:W1p}, as well as the estimate of $\|\nabla_{\Gamma_h[\y^\alpha]}\partial_{t,h}^\bullet \hat u_h^\alpha\|_{L^6(\Gamma_h[\y^\alpha])}$ in \eqref{ind-H-K}, we have 
$$
\|\partial_{t,h}^\bullet \nabla_{\Gamma_h[\y^\alpha]} \hat u^\alpha_h\|_{L^6(\Gamma_h[\y^\alpha])}
\le C
$$
and
\begin{align*}
\|\partial_{\alpha}^\bullet \partial_{t,h}^\bullet \nabla_{\Gamma_h[\y^\alpha]} \hat u^\alpha_h\|_{L^3(\Gamma_h[\y^\alpha])} 
&\le C\|\hat e_y^\alpha \|_{W^{1,6}(\Gamma_h[\y^\alpha])} \|\partial_{t,h}^\bullet \hat u_h^\alpha\|_{W^{1,6}(\Gamma_h[\y^\alpha])} \\
&\quad\,+ C\| \partial_{\alpha}^\bullet (\nabla_{\Gamma_h[\y^\alpha]} v_h^\alpha) \|_{L^6(\Gamma_h[\y^\alpha])} 
+ C\| \partial_{\alpha}^\bullet n_{\Gamma_h[\y^\alpha]} \|_{L^6(\Gamma_h[\y^\alpha])} \\
&\quad\, + C\| \nabla_{\Gamma_h[\y^\alpha]} \hat e_y^\alpha \|_{L^6(\Gamma_h[\y^\alpha])} \\
&\le Ch^k.
\end{align*}
This implies that 
\begin{align*}
|W_{611}|
&\le
C\|\partial_{t,h}^\bullet \nabla_{\Gamma_h[\y^\alpha]} \hat u^\alpha_h\|_{L^6(\Gamma_h[\y^\alpha])}
\|\hat e_y^\alpha \|_{W^{1,6}(\Gamma_h[\y^\alpha])}^2 
\| \hat\varphi_h^\alpha\|_{H^1(\Gamma_h[\y^\alpha])} 
\le Ch^{2k} \|\varphi^l\|_{H^1(\Gamma)} \\
|W_{612}|
&\le
C\|\partial_{\alpha}^\bullet \partial_{t,h}^\bullet \nabla_{\Gamma_h[\y^\alpha]} \hat u^\alpha_h\|_{L^3(\Gamma_h[\y^\alpha])} 
\|\hat e_y^\alpha \|_{W^{1,6}(\Gamma_h[\y^\alpha])} 
\| \hat\varphi_h^\alpha\|_{H^1(\Gamma_h[\y^\alpha])} 
\le Ch^{2k} \|\varphi^l\|_{H^1(\Gamma)} . 
\end{align*}
The estimates of $|W_{613}| $ and $ |W_{614}|$ can be done similarly by using relation \eqref{dt-grad-f}, i.e.,  
\begin{align*}
&|W_{613}|+ |W_{614}| \\
&\le
C\|\partial_{t,h}^\bullet \nabla_{\Gamma_h[\y^\alpha]} \hat u^\alpha_h\|_{L^6(\Gamma_h[\y^\alpha])}
\|\hat e_y^\alpha \|_{W^{1,6}(\Gamma_h[\y^\alpha])}^2 
\| \hat\varphi_h^\alpha\|_{H^1(\Gamma_h[\y^\alpha])} 
\le Ch^{2k} \|\varphi^l\|_{H^1(\Gamma)} .
\end{align*}
Therefore, we have 
\begin{align}\label{estimate-W61}
    |W_{61}| \le Ch^{k+1} \|\varphi^l\|_{H^1(\Gamma)} . 
\end{align}

The estimate of $|W_{62}| $ can be obtained by changing the underlying surface from $\Gamma_h[\x^*]$ to $\Gamma$ and then performing integration by parts on $\Gamma$. This can be done by defining a family of intermediate surfaces $\Gamma^{\theta}$, $\theta\in[0,1]$, between smooth surface $\Gamma$ and interpolated surface $\Gamma_h[\x^*]$, i.e., 
$$
\Gamma^{\theta}:=(1-\theta)\Gamma+\theta\Gamma_h[\x^*]
=\{x^\theta:x\in\Gamma_h[\x^*]\} ,
\,\,\,\mbox{with}\,\,\,
x^{\theta} := (1-\theta)x^{l}+\theta x\,\,\,\mbox{for}\,\,\, x\in\Gamma_h[\x^*] . 
$$ 
Let $u^\theta$, $\varphi^\theta$, $\tilde v_h^\theta$, $\hat e_y^\theta$ and $e^\theta$ be functions on $\Gamma^\theta$ defined by 
\begin{align*}
&\hat u^\theta(x^\theta) = (1-\theta)u(x^l) +\theta \hat u_h^*(x) , \\
&\hat \varphi^\theta(x^\theta) = (1-\theta) \varphi^l(x^l) +\theta \hat \varphi_h^*(x) , \\ 
&\hat e_y ^\theta(x^\theta) = \hat e_y(x) = \hat y_h^* - x \quad\mbox{(see the definition in Section \ref{section:def-hat-ey})} , \\
&e^\theta(x^\theta) = x-x^l 
&&\mbox{for}\,\,\, x \in\Gamma_h[\x^*] . 
\end{align*}
By this definition, $\varphi^0 = \varphi^l$ and $\hat e_y^0 = \hat e_y^l$, and surface $\Gamma^\theta$ moves with velocity $e^\theta$ as $\theta\in[0,1]$ changes. Moreover, the following relations hold: 
$$
\partial_{\theta}^{\bullet} \hat u^{\theta} (x^\theta) = 
\hat u_h^*(x) - u(x^l) , 
\quad
\partial_{\theta}^{\bullet} \varphi^{\theta}(x^\theta) = \hat \varphi_h^* (x) - \varphi^l(x^l), 
\quad\mbox{and}\quad
\partial_{\theta}^{\bullet} \hat e_y^\theta(x^\theta)  = 
0 .
$$
Then we can rewrite $W_{62}$ as follows by using formula \eqref{dt-grad-f}: 
\begin{align*}
    W_{62} &= \int_{\Gamma_h[\mathbf x^*]}\nabla_{\Gamma_h[\mathbf x^*]}\partial_{t,h}^\bullet \hat u^*_h \cdot D_{\Gamma_h[\mathbf x^*]}\hat e_y \nabla_{\Gamma_h[\mathbf x^*]}\hat \varphi_h^*\\
    & \quad\, - \int_{\Gamma_h[\mathbf x^*]}[\nabla_{\Gamma_h[\mathbf x^*]} \hat I_h^*v - n_{\Gamma_h[\mathbf x^*]}^*(n_{\Gamma_h[\mathbf x^*]}^*)^\top(\nabla_{\Gamma_h[\mathbf x^*]} \hat I_h^*v)^\top]\nabla_{\Gamma_h[\mathbf x^*]}\hat u_h^* \cdot D_{\Gamma_h[\mathbf x^*]}\hat e_y \nabla_{\Gamma_h[\mathbf x^*]}\hat \varphi_h^*\\
    & =: W_{621} + W_{622}.
\end{align*}
We can further approximate $W_{621}$ by $\int_{\Gamma} \nabla_{\Gamma} \partial_{t,h}^\bullet u \cdot D_{\Gamma}\hat e_y^l \nabla_{\Gamma} \varphi^l$ and use Newton--Leibniz rule, i.e., 
\begin{align*}
    W_{621} &= \int_0^1 \frac{\d}{\d \theta} \int_{\Gamma^\theta}  \nabla_{\Gamma^\theta} \partial_{t,h}^\bullet \hat u^\theta \cdot D_{\Gamma^\theta}\hat e_y^\theta\nabla_{\Gamma^\theta} \hat \varphi^\theta \,\d\theta + \int_{\Gamma} \nabla_{\Gamma} \partial_{t}^\bullet u \cdot D_{\Gamma}\hat e_y^l \nabla_{\Gamma} \varphi^l \\
    & = \int_0^1 \int_{\Gamma^\theta}\nabla_{\Gamma^\theta} \partial_{t,h}^\bullet \hat u^\theta \cdot D_{\Gamma^\theta}\hat e_y^\theta\nabla_{\Gamma^\theta}\hat \varphi^\theta \nabla_{\Gamma^\theta} \cdot e^\theta \,\d\theta+ \int_0^1 \int_{\Gamma^\theta} \partial_{\theta}^\bullet\nabla_{\Gamma^\theta} \partial_{t,h}^\bullet\hat u^\theta \cdot D_{\Gamma^\theta}\hat e_y^\theta\nabla_{\Gamma^\theta} \hat \varphi^\theta \,\d\theta\\
    & \quad \, + \int_0^1 \int_{\Gamma^\theta} \nabla_{\Gamma^\theta} \partial_{t,h}^\bullet\hat u^\theta \cdot \partial_{\theta}^\bullet D_{\Gamma^\theta}\hat e_y^\theta\nabla_{\Gamma^\theta} \hat \varphi^\theta \,\d\theta + \int_0^1 \int_{\Gamma^\theta} \nabla_{\Gamma^\theta} \partial_{t,h}^\bullet\hat u^\theta \cdot  D_{\Gamma^\theta}\hat e_y^\theta\partial_{\theta}^\bullet\nabla_{\Gamma^\theta} \hat \varphi^\theta \,\d\theta \\
    & \quad \, + \int_{\Gamma} \nabla_{\Gamma} \partial_{t}^\bullet u \cdot D_{\Gamma}\hat e_y \nabla_{\Gamma} \varphi^l = :W_{6211} + W_{6212} + W_{6213} + W_{6214} + W_{6215} . 
\end{align*}
By using the following estimates 
\begin{align*}
\|\partial_{t,h}^\bullet \hat u^\theta \|_{W^{1,6}(\Gamma^\theta)} &\le C &&\mbox{(as a result of \eqref{ind-H-K})} \\
\| \hat e_y^\theta \|_{W^{1,6}(\Gamma^\theta)} &\sim \| \hat y_h^* - {{{\rm id}_{\Gamma_h[\x^*]}}} \|_{W^{1,6}(\Gamma_h[\x^*])} \le Ch^k &&\mbox{(shown in Lemma \ref{Lemma:W1p})} \\
\| e^\theta \|_{W^{1,6}(\Gamma^\theta)} &\le Ch^k &&\mbox{(error of interpolating surface $\Gamma$)} . 
\end{align*}
we have 
$$
|W_{6211}|
\le 
\|\partial_{t,h}^\bullet \hat u^\theta \|_{W^{1,6}(\Gamma^\theta)} 
\| \hat e_y^\theta \|_{W^{1,6}(\Gamma^\theta)} \|\hat\varphi^\theta \|_{H^1(\Gamma^\theta)} \| e^\theta \|_{W^{1,6}(\Gamma^\theta)} \le 
Ch^{2k} \|\varphi^l\|_{H^1(\Gamma)} .
$$ 
By using formula \eqref{dt-grad-f}, we see that 
$\partial_{\theta}^\bullet\nabla_{\Gamma^\theta} \partial_{t,h}^\bullet\hat u^\theta $ is equal to $\nabla_{\Gamma^\theta} \partial_{\theta}^\bullet\partial_{t,h}^\bullet\hat u^\theta$ plus a term which is bounded by $\nabla_{\Gamma^\theta}e^\theta$ times $\nabla_{\Gamma^\theta} \partial_{t,h}^\bullet\hat u^\theta $. Therefore, 
\begin{align*}
|W_{6212}|
&\le 
( \|\nabla_{\Gamma^\theta} \partial_{\theta}^\bullet\partial_{t,h}^\bullet \hat u^\theta \|_{L^2(\Gamma^\theta)} 
+ \| e^\theta \|_{W^{1,3}(\Gamma^\theta)} \|\partial_{t,h}^\bullet \hat u^\theta \|_{W^{1,6}(\Gamma^\theta)} 
)
\| \hat e_y^\theta \|_{W^{1,3}(\Gamma^\theta)} \|\hat\varphi^\theta \|_{W^{1,6}(\Gamma^\theta)} \\
&\le 
Ch^{2k} \|\varphi^l\|_{H^2(\Gamma)} ,
\end{align*}
where the last inequality uses the Sobolev embedding $H^2(\Gamma)\hookrightarrow W^{1,6}(\Gamma)$, the following result that follows from \eqref{dtHh-dtH-H1-f}: 
$$
\|\nabla_{\Gamma^\theta} \partial_{\theta}^\bullet\partial_{t,h}^\bullet \hat u^\theta \|_{L^2(\Gamma^\theta)} 
  \le C\|(\partial_{\theta}^\bullet\partial_{t,h}^\bullet \hat u^\theta)^l \|_{H^1(\Gamma)} 
        = C \| (\partial_{t,h}^\bullet \hat u_h^*)^l - \partial_t ^\bullet u\|_{H^1(\Gamma)} \le Ch^k, 
$$ 
and the following result pertains to the interpolation error caused by the perturbation of the surface:
$$
\|\nabla_{\Gamma^\theta} e^\theta\|_{L^\infty(\Gamma^\theta)}\le Ch^k.
$$
Moreover, $W_{6213} $ and $ W_{6214}$ can be estimated similarly by using the following results: 
\begin{align*}
    \|\nabla_{\Gamma^\theta} \partial_\theta^\bullet \varphi_h^\theta\|_{L^2(\Gamma^\theta)} & \le \|\varphi^l - (\hat \varphi_h^*)^l\|_{H^1(\Gamma)} 
    = \|\varphi^l - (P_h\varphi)^l\|_{H^1(\Gamma)}  \le Ch\|\varphi^l \|_{H^2(\Gamma)} 
\end{align*}
This would yield the following results: 
\begin{align*}
    |W_{6213}| + |W_{6214}| \le Ch^{2k} \|\varphi^l \|_{H^2(\Gamma)} .
\end{align*}
Then, $W_{6215}$ can be estimated with integration by parts and the estimate of $\|\hat e_y\|_{L^2(\Gamma)}\sim \| \hat y_h^* - {{{\rm id}_{\Gamma_h[\x^*]}}} \|_{L^2(\Gamma_h[\x^*])}$ in Lemma \ref{Lemma:L2}, with $ |W_{6215}| \le Ch^{k+1} \|\varphi^l\|_{H^2(\Gamma)}$. 
This proves the following result (by summing up the estimates of $|W_{621j}|$, $j=1,\dots,5$): 
\begin{align*}
    |W_{621}| &\le Ch^{k+1} \|\varphi^l\|_{H^2(\Gamma)} .
\end{align*}
Note that $|W_{622}|$ can be estimated similarly as $|W_{621}|$, with $|W_{622}| \le Ch^{k+1} \|\varphi^l\|_{H^2(\Gamma)}$. By summing up the estimates of $|W_{621}|$ and $|W_{622}|$ we obtain 
\begin{align}\label{estimate-W62}
    |W_{62}| \le Ch^{k+1} \|\varphi^l\|_{H^2(\Gamma)}.
\end{align}
Now, substituting \eqref{estimate-W61} and \eqref{estimate-W62} into \eqref{def-W6} yields \eqref{estimate-W6}. Then, substituting the estimates in \eqref{estimate-W1}--\eqref{estimate-W6} into the right-hand side of \eqref{dth-w-elliptic} yields \eqref{dtw-remain} and therefore completes the proof of \eqref{L2-estimate-dtw}.

\subsection{Proof of the ${\bf L^2}$ estimates in Lemma \ref{Lemma:L2-dtH}} 

By substituting estimates \eqref{H1-estimate-g}, \eqref{L2-estimate-g}, \eqref{dth-w-H1-h} and \eqref{L2-estimate-dtw} into \eqref{dtHh-dtH-L2-1}, we obtain the following result: 
\begin{align*} 
\|(\partial_{t,h}^\bullet \hat u_h^*)^l - \partial_{t,h}^\bullet u \|_{L^2(\Gamma)} 
&\le
Ch^{k+1} + Ch 
\inf_{\chi_h\in S_h(\Gamma_h[\x^*])} \| \partial_{t,h}^\bullet u - \chi_h^l \|_{H^1(\Gamma)} \notag \\
&\le
Ch^{k+1}+ Ch 
\| \partial_{t,h}^\bullet u - (I_h^*\partial_t^\bullet u)^l  \|_{H^1(\Gamma)} \notag \\
&=
Ch^{k+1} + Ch 
 \| \partial_{t}^\bullet u + (\tilde v_h^*-v) \cdot\nabla u - (I_h^*\partial_t^\bullet u)^l \|_{H^1(\Gamma)} \notag \\
&\le
Ch^{k+1}+ Ch 
 \| \partial_{t}^\bullet u - (I_h^*\partial_t^\bullet u)^l  \|_{H^1(\Gamma)} ,
\end{align*} 
where the last inequality follows from the application of the triangle inequality and the estimate of $\|\tilde v_h^*-v\|_{H^1(\Gamma)}$ in \eqref{tilde_v_h-v}. The last inequality, together with a standard estimate of interpolation error $\| \partial_{t}^\bullet u - (I_h^*\partial_t^\bullet u)^l  \|_{H^1(\Gamma)} $, implies the following $L^2$ estimate:
\begin{align}\label{dtHh-dtH-L2-f}
\|(\partial_{t,h} ^\bullet \hat u_h^*)^l - \partial_{t}^\bullet u \|_{L^2(\Gamma)} 
&\le Ch^{k+1} \quad\mbox{for}\,\,\, t\in[0,t_*+\delta] . 
\end{align} 
This proves \eqref{L2-dtH*}--\eqref{L2-dtn*} for $t\in[0,t_*]$ and therefore completes the proof of Lemma \ref{Lemma:L2-dtH}, as  discussed at the end of Section \ref{section:B1}. 
\hfill$\square$

\section{Proof of Lemma \ref{Lemma:defects}}\label{appendix_D}
\renewcommand{\thelemma}{D.\arabic{lemma}}
\renewcommand{\theequation}{D.\arabic{equation}}
\renewcommand{\thesubsection}{D.\arabic{subsection}}

The proof of Lemma \ref{Lemma:defects} is similar as the proof for Lemma \ref{Lemma:L2-dtH}. For the readers' convenience, we present a complete proof and divide it into several sections.

For a finite element function $\chi_u \in S_h[\y^*]^4 $ we denote by $\hat\chi_u$ the finite element function on $\Gamma_h[\x^*]$ with the same nodal vector as $\chi_u$ and denote by $\hat\chi_u^l$ the lift of $\hat\chi_u$ from $\Gamma_h[\x^*]$ to $\Gamma$. Then equation \eqref{KLL} implies that 
\begin{equation}\label{exact_u}
  \begin{aligned}
 &\int_{\Gamma} \partial_t ^\bullet u \cdot\hat \chi_u^l 
    + \int_{\Gamma} \nabla_{\Gamma} u \cdot \nabla_{\Gamma} \hat \chi_u^l 
    = \int_{\Gamma}|\nabla_{\Gamma}n|^2 u\cdot \hat \chi_u^l   \qquad \forall \chi_u \in S_h[\y^*]^4 . 
  \end{aligned}
\end{equation}
The definition of the dynamic Ritz projection in \eqref{def-Ritz} implies that 
\begin{align}\label{Ritz_u}
&\int_{\Gamma_h[\y^*]}\nabla_{\Gamma_h[\y^*]} u_h^* \cdot \nabla_{\Gamma_h[\y^*]} \chi_u - \int_{\Gamma} \nabla_{\Gamma} u \cdot \nabla_{\Gamma} \hat \chi_u^l 
 = \int_{\Gamma}  u  \cdot \hat \chi_u^l - \int_{\Gamma_h[\y^*]}  u_h^*\cdot  \chi_u
\end{align}
By subtracting \eqref{exact_u} from \eqref{weak-Ritz-u-version} and utilizing \eqref{Ritz_u}, we have 
\begin{equation}\label{defect_H_identity}
  \begin{aligned}
   & \int_{\Gamma_h[\bf y^*]}{d}_u^*\cdot \chi_u \\
   &= \int_{\Gamma_h[\y^*]} \partial_{t,h}^\bullet u_h^* \cdot\chi_u - \int_{\Gamma} \partial_t ^\bullet u \cdot\hat \chi_u^l + \int_{\Gamma_h[\y^*]}\nabla_{\Gamma_h[\y^*]} u_h^* \cdot \nabla_{\Gamma_h[\y^*]} \chi_u - \int_{\Gamma} \nabla_{\Gamma} u \cdot \nabla_{\Gamma} \hat \chi_u^l \\
    & \quad \, + \int_{\Gamma}|\nabla_{\Gamma}n|^2 u\cdot \hat \chi_u^l - \int_{\Gamma_h[\y^*]}  |\nabla_{\Gamma_h[\y^*]}n_h^*|^2 u_h^*\cdot \chi_u \\
    & =  \int_{\Gamma_h[\y^*]} \partial_{t,h}^\bullet u_h^* \cdot \chi_u - \int_{\Gamma} \partial_t ^\bullet u\cdot  \hat \chi_u^l + \int_{\Gamma}  u  \cdot \hat \chi_u^l - \int_{\Gamma_h[\y^*]}  u_h^*\cdot  \chi_u \\
    & \quad\,+ \int_{\Gamma} |\nabla_{\Gamma}n|^2 u\cdot\hat \chi_u^l - \int_{\Gamma_h[\y^*]} |\nabla_{\Gamma_h[\y^*]}n_h^*|^2 u_h^* \cdot\chi_u\\
    & =: \int_{\Gamma_h[\bf y^*]}{d}_{u,1}^*\cdot \chi_u + \int_{\Gamma_h[\bf y^*]}{d}_{u,2}^*\cdot \chi_u +\int_{\Gamma_h[\bf y^*]}{d}_{u,3}^*\cdot \chi_u, \qquad \qquad \quad \forall \chi_u \in S_h[\y^*]^4 . 
  \end{aligned}
\end{equation}

The three terms on the right-hand side of \eqref{defect_H_identity} are estimated in the following lemmas.
 
\begin{lemma}\label{Lemma:defects12}
Under the assumptions of Theorem \ref{THM1}, the following estimates hold{\rm:}
\begin{align*}
\bigg|\int_{\Gamma_h[\bf y^*]}{d}_{u,1}^*\cdot \chi_u\bigg|
\le
Ch^{k+1}  \|\chi_u\|_{H^1(\Gamma_h[\bf y^*])} ,\\
\bigg|\int_{\Gamma_h[\bf y^*]}{d}_{u,2}^*\cdot \chi_u\bigg|
\le
Ch^{k+1}  \|\chi_u\|_{H^1(\Gamma_h[\bf y^*])}  . 
\end{align*}
\end{lemma}
\begin{proof}
By using the triangle inequality, we have
\begin{align}\label{estimate-du2}
&\bigg|\int_{\Gamma_h[\bf y^*]}{d}_{u,2}^*\cdot \chi_u\bigg| \\
&=
\bigg|\int_{\Gamma} u\cdot \hat \chi_u^l -  \int_{\Gamma_h[\y^*]}  u_h^*\cdot \chi_u\bigg| \notag\\
&\le
\bigg|\int_{\Gamma} u\cdot \hat \chi_u^l -  \int_{\Gamma_h[\x^*]} u^{-l}\cdot \hat\chi_u\bigg|
+ \bigg|\int_{\Gamma_h[\x^*]} u^{-l}\cdot \hat \chi_u -  \int_{\Gamma_h[\x^*]}  \hat I_h^*u \cdot \hat\chi_u\bigg| \notag\\
&\quad 
+\bigg| \int_{\Gamma_h[\x^*]}  \hat I_h^*u\cdot  \hat\chi_u -  \int_{\Gamma_h[\y^*]}  u_h^*\cdot \chi_u \bigg| \notag\\
&\le
Ch^{k+1} \|\chi_u\|_{L^2(\Gamma_h[\y^*])} + Ch^{k+1} \|\chi_u\|_{L^2(\Gamma_h[\y^*])}  
+\bigg| \int_{\Gamma_h[\x^*]} \hat I_h^*u \cdot\hat\chi_u -  \int_{\Gamma_h[{\y}^*]}  u_h^*\cdot\chi_u \bigg| , \notag
\end{align}
where the first term on the right-hand side is obtained by using the standard estimates for the geometric perturbation errors in \cite[Lemma 5.6]{Kovacs2018} (also see \cite[Lemma 7.4]{KLLP2017}), and the second term on the right-hand side of the inequality above is obtained by using error estimates for the Lagrange interpolation, i.e., $\|u - (\hat I_h^* u)^l\|_{L^\infty(\Gamma)}\le Ch^{k+1}$. In both terms we have converted $\|\chi_u\|_{L^2(\Gamma_h[\bf x^*])}$ to $\|\chi_u\|_{L^2(\Gamma_h[\y^*])}$ by using the norm equivalence between $\Gamma_h[\x^*]$ and $\Gamma_h[\y^*]$. The third term on the right-hand side of the inequality above can be estimated as follows. 

Let $ \y^\theta=(1-\theta)\x^*+\theta \y^*$,
let $\chi_u^{\theta}$, $u_h^{\theta}$ and $e_y^\theta$ be the finite element functions on $\Gamma_h[\y^\theta]$ with nodal vectors ${\bm \chi_{\bf u}}$, $(1-\theta){\bf I}_hu+\theta {\bf u}^*$ and $\bf e_\y: = \y^* -\x^*$, respectively. Then 
$$
\partial_\theta^{\bullet}u_h^{\theta} = e_u^{\theta}
\quad\mbox{and}\quad
\partial_\theta^{\bullet}\hat\chi_u^{\theta} =0 , 
$$
where $e_u^{\theta}$ is the finite element function on $\Gamma_h[\y^\theta]$ with nodal vector ${\bf u}^*-{\bf I}_hu$. The surface $\Gamma_h[\y^\theta]$ is moving with velocity $e_y^\theta$ as $\theta$ changes from $0$ to $1$. Then, Lemma \ref{Lemma:L2} and the equivalence of norms on $\Gamma_h[\y^\theta]$, $\theta\in[0,1]$, imply that
\begin{align}\label{ey-L2-copy}
\|e_y^\theta\|_{L^2(\Gamma_h[\y^\theta])} + \|e_u^\theta\|_{L^2(\Gamma_h[\y^\theta])} \le Ch^{k+1}.
\end{align}
By using \eqref{ey-L2-copy}, along with the equivalence of $L^p$ norm on $\Gamma_h[\bf y^\theta]$ and $\Gamma_h[\bf y^*]$, we have  
\begin{align*}
&\quad\bigg| \int_{\Gamma_h[\x^*]}  \hat I_h^*u \cdot\hat\chi_u -  \int_{\Gamma_h[\y^*]}  u_h^*\cdot\chi_u \bigg| \\
&=
\bigg|\int_0^1 \frac{\d}{\d\theta}\int_{\Gamma_h[\y^\theta]}  u_h^{\theta} \cdot \chi_u^\theta \d\theta\bigg| \\
&=\bigg|\int_0^1 \int_{\Gamma_h[\y^\theta]}  (e_u^{\theta} \cdot\chi_u^\theta + u_h^{\theta} \cdot\chi_u^\theta \nabla\cdot e_y^\theta)  \d\theta\bigg| \\
&\le C\int_0^1  \|e_u^{\theta}\|_{L^2(\Gamma_h[\y^\theta])} \|\chi_u^\theta\|_{L^2(\Gamma_h[\y^\theta])} \d\theta + \bigg| \int_0^1 \int_{\Gamma_h[\y^\theta]} u_h^{\theta} \cdot \chi_u^\theta \nabla_{\Gamma_h[\y^\theta]}\cdot e_y^\theta \d\theta\bigg| \\
&\le Ch^{k+1}\|\chi_u\|_{L^2(\Gamma_h[\bf y^*])} + \bigg| \int_0^1 \int_{\Gamma_h[\y^\theta]} u_h^{\theta} \cdot\chi_u^\theta \nabla_{\Gamma_h[\y^\theta]}\cdot e_y^\theta \d\theta\bigg| .
\end{align*}
The second term on the right-hand side of the last inequality can be estimated as follows (by using the triangle inequality): 
\begin{align*}
    \bigg| \int_0^1 \int_{\Gamma_h[\y^\theta]} u_h^{\theta} \cdot \chi_u^\theta \nabla_{\Gamma_h[\y^\theta]}\cdot e_y^\theta \d\theta\bigg| 
    &\le \bigg|\int_{\Gamma_h[\bf x^*]} \hat u_h^* \cdot \hat\chi_u^* \nabla_{\Gamma_h[\bf x^*]}\cdot e_y^0  - \int_{\Gamma} (\hat u_h^*)^l \cdot (\hat\chi_u^*)^l (\nabla_{\Gamma_h[\bf x^*]}\cdot e_y^0)^l \bigg|\\
    &\quad \,+ \bigg|\int_{\Gamma} (\hat u_h^*)^l \cdot (\hat\chi_u^*)^l [(\nabla_{\Gamma_h[\bf x^*]}\cdot e_y^0)^l  -  \nabla_{\Gamma}\cdot (e_y^0)^l ]\bigg|\\
    & \quad \,+ \bigg|\int_0^1 \int_0^\theta \frac{\d }{\d \alpha} \int_{\Gamma_h[\y^\alpha]} u_h^\alpha \cdot \chi_h^\alpha \nabla_{\Gamma_{h}[\bf y^\alpha]}\cdot e_y^\alpha \d\alpha\d\theta\bigg|\\
    & \quad \, + \bigg|\int_{\Gamma} (\hat u_h^*)^l \cdot(\hat\chi_u^*)^l \nabla_{\Gamma}\cdot (e_y)^l \bigg|\\
    &=: J_1 + J_2 + J_3 + J_4 
    \le  Ch^{k+1} \|{\chi_u}\|_{H^1(\Gamma_h[\bf y^*])},
\end{align*} 
where $J_1$ is estimated by applying the geometric perturbation errors (see \cite[Lemma 5.6]{Kovacs2018} and \cite[Lemma 7.4]{KLLP2017}); $J_2$ is estimated by applying the chain rule of partial differentiation and Lemma \ref{Lemma:W1p}, which imply that 
$$
J_2\le C\|(\hat\chi_u^*)^l\|_{L^2(\Gamma)} h^k \|\nabla_{\Gamma_h[\x^*]} e_y^0\|_{L^2(\Gamma_h[\x^*])} \le Ch^{2k} \|\hat\chi_u^*\|_{L^2(\Gamma_h[\y^*])} ; 
$$
$J_3$ is estimated by utilizing identities \eqref{M-Diff-0} and \eqref{dt-grad-f}; $J_4$ is estimated by applying integration by parts which transfers the gradient operator $\nabla_\Gamma$ from $(e_y)^l$ to $(\hat u_h^*)^l \cdot(\hat\chi_u^*)^l $. 

This proves the desired upper bound for the third term on the right-hand side of \eqref{estimate-du2}, and therefore completes the proof for the second result of Lemma \ref{Lemma:defects12}. The first result of Lemma \ref{Lemma:defects12} can be proved in the same way by using Lemma \ref{Lemma:L2-dtH}. 
\hfill\end{proof}

\begin{lemma}\label{Lemma:defect3}
Under the assumptions of Theorem \ref{THM1}, the following estimate holds{\rm:}
\begin{align*}
  \bigg|\int_{\Gamma_h[\bf y^*]}{d}_{u,3}^*\cdot \chi_u\bigg|
  \le
Ch^{k+1}  \|{\chi_u}\|_{H^1(\Gamma_h[\bf y^*])}. 
\end{align*}
\end{lemma}
\begin{proof}
  Observe that
\begin{align}\label{du3.phi}
  \bigg|\int_{\Gamma_h[\bf y^*]}{d}_{u,3}^*\cdot \chi_u\bigg|
&= 
\bigg| \int_{\Gamma} |\nabla_{\Gamma}n|^2 u \cdot \hat \chi_u^l -
\int_{\Gamma_h[\y^*]}
|\nabla_{\Gamma_h[\y^*]}n_h^*|^2 u_h^* \cdot \chi_u \bigg| \notag \\
&\le 
\bigg| \int_{\Gamma} |\nabla_{\Gamma}n|^2 u \cdot \hat \chi_u^l  -
\int_{\Gamma_h[\x^*]}
|\nabla_{\Gamma_h[\x^*]}\hat I_h^*n|^2 \hat I_h^*u\cdot \hat\chi_u \bigg|  \notag \\
&\quad
+
\bigg| \int_{\Gamma_h[\x^*]}
|\nabla_{\Gamma_h[\x^*]}\hat I_h^*n|^2 \hat I_h^*u\cdot \hat\chi_u 
-\int_{\Gamma_h[\y^*]}
|\nabla_{\Gamma_h[\y^*]}n_h^*|^2 u_h^* \cdot \chi_u  \bigg| .
\end{align}
The first term on the right-hand side of \eqref{du3.phi} can be estimated as follows:
\begin{align}\label{d3-12}
& \bigg| \int_{\Gamma} |\nabla_{\Gamma}n|^2 u \cdot \hat \chi_u^l  -
\int_{\Gamma_h[\x^*]}
|\nabla_{\Gamma_h[\x^*]}\hat I_h^*n|^2 \hat I_h^*u\cdot \hat\chi_u \bigg|  \notag \notag \\
&\le  \bigg| \int_{\Gamma} |\nabla_{\Gamma}n|^2 u\cdot\hat \chi_u^l -
\int_{\Gamma_h[\x^*]}
|\nabla_{\Gamma_h[\x^*]}n^{-l}|^2 u^{-l} \cdot \hat\chi_u\bigg| \notag \\
&\quad
+ \bigg|\int_{\Gamma_h[\x^*]}
|\nabla_{\Gamma_h[\x^*]}n^{-l}|^2 u^{-l} \cdot \hat\chi_u  -
\int_{\Gamma_h[\x^*]}
|\nabla_{\Gamma_h[\x^*]}\hat I_h^*n|^2 \hat I_h^*u\cdot \hat\chi_u\bigg| \notag \\ 
&=:F_1+F_2 .
\end{align}
To estimate $F_1$, we denote 
$$
x^{\theta} := (1-\theta)x^{l}+\theta x\,\,\,\mbox{for}\,\,\, x\in\Gamma_h[\x^*]
\quad\mbox{and}\quad 
\Gamma^{\theta}:=(1-\theta)\Gamma+\theta\Gamma_h[\x^*]
=\{x^\theta:x\in\Gamma_h[\x^*]\} .
$$ 
Let $u^\theta$, $n^\theta$, $e^\theta$ and $\chi_u^\theta$ be functions on $\Gamma^\theta$ defined by 
$$
u^\theta(x^\theta) = u(x^l),
\quad 
n^\theta(x^\theta) = n(x^l),
\quad
e^\theta(x^\theta) = x-x^l
\quad\mbox{and}\quad
\chi_u^\theta(x^\theta)=\hat \chi_u(x) 
\quad\mbox{for}\,\,\, x\in\Gamma_h[\bf x^*] . 
$$
Then
$$
\partial_{\theta}^{\bullet} u^{\theta} = \partial_{\theta}^{\bullet} \chi_u^{\theta}  =0
\quad\mbox{and}\quad
\partial_{\theta}^{\bullet} n^{\theta} =0 , 
$$
and
\begin{align*}
F_1 
&=\bigg|  \int_0^1 \frac{\d}{\d\theta} \int_{\Gamma^\theta}
|\nabla_{\Gamma^\theta}n^\theta|^2 u^\theta\cdot\chi_u^\theta \d\theta\bigg|  \\
&=  \bigg|  \int_0^1  \int_{\Gamma^\theta}
|\nabla_{\Gamma^\theta}n^\theta|^2 (\nabla_{\Gamma^\theta}\cdot e^\theta) u^\theta\cdot\chi_u^\theta\d\theta \\
&\quad\,\,
 + \int_0^1\int_{\Gamma^\theta} 
2 \nabla_{\Gamma^\theta}n^\theta 
\cdot\Big(\nabla_{\Gamma^{\theta}} \partial_{\theta}^{\bullet} n^{\theta}-\nabla_{\Gamma^{\theta}} e^{\theta} \nabla_{\Gamma^{\theta}} n^{\theta}+\nu^{\theta}(\nu^{\theta})^{\top}(\nabla_{\Gamma^{\theta}} e^{\theta})^{\top} \nabla_{\Gamma^{\theta}} n^{\theta}  \Big) u^\theta\cdot\chi_u^\theta \d\theta \bigg|  \\
&=:  |F_{11}+F_{12}| ,
\end{align*}
where $\nu^\theta$ is the unit normal vector on $\Gamma^\theta$. 
Notice that the following identity holds:
\begin{align*}
    F_{11} & =  \int_0^1  \int_{\Gamma^\theta}
    |\nabla_{\Gamma^\theta}n^\theta|^2 (\nabla_{\Gamma^\theta}\cdot e^\theta) u^\theta\cdot\chi_u^\theta\d\theta \\
    &  = \int_0^1 \big[ \int_{\Gamma^\theta}  |\nabla_{\Gamma^\theta}n^\theta|^2 (\nabla_{\Gamma^\theta}\cdot e^\theta) u^\theta\cdot\chi_u^\theta - \int_{\Gamma} |\nabla_{\Gamma}n|^2 (\nabla_{\Gamma} \cdot e^l) u\cdot\hat \chi_u^l\big] \,\d \theta\\
    & \quad \, + \int_{\Gamma} |\nabla_{\Gamma}n|^2 (\nabla_{\Gamma} \cdot e^l) u\cdot\hat \chi_u^l = : F_{111} + F_{112}.
\end{align*}
By employing the intermediate surface again and referencing the inequality $\|\nabla _{\Gamma^\theta} \cdot e^\theta\|_{L^\infty(\Gamma^\theta)} \le Ch^k$ induced from Lagrange interpolation error estimates, we obtain
\begin{equation*}
  \begin{aligned}
    |F_{111}| \le C\|\nabla _{\Gamma^\theta} \cdot e^\theta\|_{L^\infty(\Gamma^\theta)} ^2 \|\hat \chi_u^l\|_{L^2(\Gamma)}\le Ch^{2k} \|\hat \chi_u\|_{H^1(\Gamma_h[\bf x^*])}.
  \end{aligned}
\end{equation*}
By integration by parts and referencing the inequality $\|e^\theta\|_{L^\infty(\Gamma^\theta)} \le Ch^{k+1}$ from Lagrange interpolation error estimates, we have
\begin{equation*}
  \begin{aligned}
    |F_{112}| \le C\|e^\theta\|_{L^\infty(\Gamma^\theta)} \|\nabla_{\Gamma} \hat \chi_u^l\|_{L^2(\Gamma)}\le Ch^{k+1} \|\hat \chi_u\|_{H^1(\Gamma_h[\bf x^*])}.
  \end{aligned}
\end{equation*}
The term $F_{12}$ can be estimated in the same way as $F_{11}$ (split it into two parts, separately, and using integration by parts for the second part). Hence, 
\begin{align}
F_1 \le Ch^{k+1} \|\hat \chi_u\|_{H^1(\Gamma_h[\bf x^*])}.
\end{align}
By using the triangle inequality, we have 
\begin{align*}
F_2
&= \bigg|\int_{\Gamma_h[\x^*]}
|\nabla_{\Gamma_h[\x^*]}n^{-l}|^2 u^{-l}\cdot  \hat\chi_u -
\int_{\Gamma_h[\x^*]}
|\nabla_{\Gamma_h[\x^*]}\hat I_h^*n|^2 \hat I_h^*u \cdot \hat\chi_u \bigg| \\
&\le 
\bigg|\int_{\Gamma_h[\x^*]}
|\nabla_{\Gamma_h[\x^*]}n^{-l}|^2 (u^{-l}-\hat I_h^*u)\cdot  \hat\chi_u\bigg| \\
&\quad + \bigg|
\int_{\Gamma_h[\x^*]}
( |\nabla_{\Gamma_h[\x^*]}n^{-l}|^2- |\nabla_{\Gamma_h[\x^*]}\hat I_h^*n|^2) \hat I_h^*u\cdot\hat\chi_u \bigg| \\
&=:  F_{21} +F_{22}.
\end{align*}
By using error estimates of Lagrange interpolation, we have
\begin{align*}
F_{21}
&\le 
C\|u^{-l}-\hat I_h^*u\|_{L^2(\Gamma_h[\x^*])} \| \hat\chi_u \|_{L^2(\Gamma_h[\x^*])}
\le 
Ch^{k+1}\| \hat\chi_u \|_{L^2(\Gamma_h[\x^*])} \le Ch^{k+1}\| \hat\chi_u \|_{L^2(\Gamma_h[\x^*])}.
\end{align*}
Let $\eta=n^{-l}-\hat I_h^*n$, then 
\begin{align*}
F_{22}
&=
\bigg|
\int_{\Gamma_h[\x^*]}
(\nabla_{\Gamma_h[\x^*]}n^{-l}+\nabla_{\Gamma_h[\x^*]}\hat I_h^*n) \cdot \nabla_{\Gamma_h[\x^*]}\eta \big(\hat I_h^*u \cdot \hat\chi_u\big) \bigg|\\
&\le 
\bigg|
\int_{\Gamma_h[\x^*]}
(\nabla_{\Gamma_h[\x^*]}n^{-l}+\nabla_{\Gamma_h[\x^*]}\hat I_h^*n
- 2 (\nabla_{\Gamma}n) ^{-l} ) \cdot \nabla_{\Gamma_h[\x^*]}\eta \big(\hat I_h^*u \cdot \hat\chi_u\big) \bigg| \\
&\quad
+\bigg|
\int_{\Gamma_h[\x^*]}
2 (\nabla_{\Gamma}n) ^{-l}  \cdot \nabla_{\Gamma_h[\x^*]}\eta \big(\hat I_h^*u \cdot \hat\chi_u\big)  - 2\int_{\Gamma} \nabla_{\Gamma} n \cdot \nabla_{\Gamma}\eta^l \big(\hat I_h^*u \cdot \hat\chi_u\big)^l\bigg| \\
&\quad
+ \bigg|2\int_{\Gamma} \nabla_{\Gamma} n \cdot \nabla_{\Gamma}\eta^l \big(\hat I_h^*u \cdot \hat\chi_u\big)^l\bigg| =: F_{221} + F_{222} + F_{223}.
\end{align*}
Since the following inequality holds ({\mbox{Lagrange interpolation error estimates}})
\begin{align}\label{eta-error}
  \|\eta\|_{L^2(\Gamma_h[\bf x^*])} + h\|\nabla_{\Gamma_h[\bf x^*]} \eta\|_{L^2(\Gamma_h[\bf x^*])} \le Ch^{k+1},
\end{align}
along with error estimates from perturbation of surfaces, we have
\begin{align*}
  F_{221} + F_{222} &\le Ch^{2k} \| \hat\chi_u \|_{L^2(\Gamma_h[\x^*])}.
\end{align*}
By utilizing integration by parts to remove the gradient $\nabla_\Gamma$ from $\eta^l$, along with inequality \eqref{eta-error}, we have
\begin{align*}
  F_{223} \le C\|\eta^l\|_{L^2(\Gamma)} \|\hat\chi_u \|_{H^1(\Gamma_h[\x^*])} \le Ch^{k+1} \| \hat\chi_u \|_{H^1(\Gamma_h[\x^*])}.
\end{align*}
Then we obtain 
\begin{align}
F_2 &\le F_{21} + F_{221} + F_{222} + F_{223} \le
Ch^{k+1} \| \hat\chi_u \|_{H^1(\Gamma_h[\x^*])}. 
\end{align}
Substituting the estimates of $F_1$ and $F_2$ into \eqref{d3-12}, we obtain the following result for the first term on the right-hand side of \eqref{du3.phi}: 
\begin{align}\label{d3-1}
 \bigg| \int_{\Gamma} |\nabla_{\Gamma}n|^2 u \cdot \hat \chi_u^l  -
\int_{\Gamma_h[\x^*]}
|\nabla_{\Gamma_h[\x^*]}\hat I_h^*n|^2 \hat I_h^*u\cdot \hat\chi_u \bigg| 
&\le 
Ch^{k+1} \| \hat\chi_u \|_{H^1(\Gamma_h[\x^*])}\notag\\
& \le Ch^{k+1} \| \hat\chi_u \|_{H^1(\Gamma_h[\y^*])} , 
\end{align}
where norm equivalence of $W^{1,p}$ norm on $\Gamma_h[\bf x^*]$ and $\Gamma_h[\bf y^*]$ has been used in the last inequality.

The second term on the right-hand side of \eqref{du3.phi} can be estimated by employing the intermediate surface again between $\Gamma_h[\bf x^*]$ and $\Gamma_h[\bf y^*]$. Recall that $ \y^\theta=(1-\theta)\x^*+\theta \y^*$. 
Let $\chi_u^{\theta}$, $u_h^{\theta}$, $n_h^\theta$ and $e_y^\theta$ be the finite element functions on $\Gamma_h[\y^\theta]$ with nodal vectors ${\bm \chi_{\bf u}}$, $(1-\theta){\bf I}_hu+\theta {\bf u}^*$,$(1-\theta){\bf I}_h n+\theta {\bf n}^*$ and ${\bf e}_\y := \y^* -\x^*$, respectively. Then 
$$
\partial_\theta^{\bullet}u_h^{\theta} = e_u^{\theta}, \quad \partial_\theta^{\bullet}n_h^{\theta} = e_n^{\theta}
\quad\mbox{and}\quad
\partial_\theta^{\bullet}\hat\chi_u^{\theta} =0 , 
$$
where $e_u^{\theta}$ and $e_n^{\theta}$ are the finite element functions on $\Gamma_h[\y^\theta]$ with nodal vectors ${\bf u}^*-{\bf I}_hu$ and ${\bf n}^*-{\bf I}_h n$, respectively. The surface $\Gamma_h[\y^\theta]$ is moving with velocity $e_y^\theta$ as $\theta$ changes from $0$ to $1$. Then, Lemmas \ref{Lemma:W1p} and \ref{Lemma:L2} imply that
\begin{align}\label{ey-L2-copy-2}
\|e_y^\theta\|_{L^2(\Gamma_h[\y^\theta])} + h\|e_y^\theta\|_{H^1(\Gamma_h[\y^\theta])} \le Ch^{k+1},
\end{align}
\begin{align}\label{en-L2}
  \|e_n^\theta\|_{L^2(\Gamma_h[\y^\theta])}
  +h\|e_n^\theta\|_{H^1(\Gamma_h[\y^\theta])} \le Ch^{k+1},
  \end{align}
\begin{align}\label{eu-L2}
  \|e_u^\theta\|_{L^2(\Gamma_h[\y^\theta])}
  +h\|e_u^\theta\|_{H^1(\Gamma_h[\y^\theta])} \le Ch^{k+1}.
\end{align}
Then the second term on the right-hand side of \eqref{du3.phi} can be estimated as follows:
\begin{align}\label{rhs-du3}
&\quad \bigg| \int_{\Gamma_h[\x^*]}
|\nabla_{\Gamma_h[\x^*]}\hat I_h^*n|^2 \hat I_h^*u\cdot \hat\chi_u 
-\int_{\Gamma_h[\y^*]}
|\nabla_{\Gamma_h[\y^*]}n_h^*|^2 u_h^* \cdot \chi_u  \bigg| \notag\\
&=
\bigg| \int_0^1 \frac{\d}{\d\theta} \int_{\Gamma_h[\y^\theta]}
|\nabla_{\Gamma_h[\y^\theta]}n_h^\theta|^2 u_h^\theta \cdot\chi_u^\theta \,\d\theta \bigg| \notag\\
&=
\bigg| \int_0^1  \int_{\Gamma_h[\y^\theta]} 
( |\nabla_{\Gamma_h[\y^\theta]}n_h^\theta|^2 e_u^\theta \cdot \chi_u^\theta + |\nabla_{\Gamma_h[\y^\theta]}n_h^\theta|^2\big(u_h^\theta \cdot \chi_u^\theta\big) \nabla_{\Gamma_h[\y^\theta]}\cdot e_y^\theta)\,\d\theta \notag \\
&\quad
+\int_0^1  \int_{\Gamma_h[\y^\theta]}  2 \nabla_{\Gamma_h[\y^\theta]}n_h^\theta \cdot \Big(\nabla_{\Gamma_h[\y^\theta]} \partial_{\theta}^{\bullet} n_{h}^{\theta}-\nabla_{\Gamma_h[\y^\theta]} e_{y}^{\theta} \nabla_{\Gamma_h[\y^\theta]} n_{h}^{\theta}+\nu^{\theta}(\nu^{\theta})^{\top}(\nabla_{\Gamma_h[\y^\theta]} e_{y}^{\theta})^{\top} \nabla_{\Gamma_h[\y^\theta]} n_{h}^{\theta}  \Big) \big(u_h^\theta \cdot \chi_u^\theta\big)\,\d\theta \bigg| \notag\\
&\le 
\bigg| \int_0^1  \int_{\Gamma_h[\y^\theta]} |\nabla_{\Gamma_h[\y^\theta]}n_h^\theta|^2 e_u^\theta \cdot \chi_u^\theta\,\d\theta \bigg|  \notag\\
&\quad
+ \bigg| \int_0^1 \int_{\Gamma_h[\y^\theta]}  |\nabla_{\Gamma_h[\y^\theta]}n_h^\theta|^2 \big(u_h^\theta \cdot \chi_u^\theta\big)\nabla_{\Gamma_h[\y^\theta]}\cdot e_y^\theta \,\d\theta \bigg| \notag\\
&\quad
+ \bigg| \int_0^1  \int_{\Gamma_h[\y^\theta]}  2 \nabla_{\Gamma_h[\y^\theta]}n_h^\theta \cdot  \nabla_{\Gamma_h[\y^\theta]} e_n^{\theta} \big(u_h^\theta \cdot \chi_u^\theta\big)\,\d\theta \bigg| \notag\\
&\quad
+
\bigg| \int_0^1  \int_{\Gamma_h[\y^\theta]}   2 \big(\nabla_{\Gamma_h[\y^\theta]}n_h^\theta \cdot
\nabla_{\Gamma_h[\y^\theta]} e_{y}^{\theta} \nabla_{\Gamma_h[\y^\theta]} n_{h}^{\theta}\big) \big(u_h^\theta \cdot \chi_u^\theta\big)\,\d\theta \bigg|\notag \\
&\quad
+
\bigg| \int_0^1  \int_{\Gamma_h[\y^\theta]}  2 \big(\nabla_{\Gamma_h[\y^\theta]}n_h^\theta \cdot \nu_{h}^{\theta}(\nu_{h}^{\theta})^{\top}(\nabla_{\Gamma_h[\y^\theta]} e_{y}^{\theta})^{\top} (\nabla_{\Gamma_h[\y^\theta]} n_{h}^{\theta} )\big)\big(u_h^\theta \cdot \chi_u^\theta\big)\,\d\theta \bigg| \notag\\
&=: \sum_{j=1}^5 E_j,
\end{align}
where $\nu^\theta$ is the unit normal vector on $\Gamma_h[\y^\theta]$.
By using inequality \eqref{eu-L2}, we have  
\begin{align*}
E_1 
&= \bigg| \int_0^1  \int_{\Gamma_h[\y^\theta]} |\nabla_{\Gamma_h[\y^\theta]}n_h^\theta|^2 \big(e_u^\theta \cdot \chi_u^\theta\big)\,\d\theta \bigg| \\ 
&\le C \|e_u^\theta\|_{L^2(\Gamma_h[\y^\theta])} 
\|\nabla_{\Gamma_h[\y^\theta]}n_h^\theta\|_{L^{6}(\Gamma_h[\y^\theta])}^2 
\|\chi_u^\theta\|_{L^{6}(\Gamma_h[\y^\theta])} \\
&\le
C h^{k+1}
\|\chi_u\|_{H^1(\Gamma_h[\y^*])},
\end{align*}
where norm equivalence of $L^p$ and $W^{1,p}$ norms on $\Gamma_h[\y^\theta]$ and $\Gamma_h[\y^*]$ are used in the last inequality.
Observe that $E_2$ can be rewritten as follows:
  \begin{align}\label{rewriteE2}
    E_2 &= \bigg| \int_0^1 \int_{\Gamma_h[\y^\theta]} |\nabla_{\Gamma_h[\y^\theta]}n_h^\theta|^2 \big( u_h^\theta \cdot \chi_u^\theta\big) \nabla_{\Gamma_h[\y^\theta]}\cdot e_y^\theta \d\theta  \bigg|\notag \\ 
    & \le \bigg| \int_0^1 \big[\int_{\Gamma_h[\y^\theta]} |\nabla_{\Gamma_h[\y^\theta]}n_h^\theta|^2 \big(u_h^\theta\cdot\chi_u^\theta\big) \nabla_{\Gamma_h[\y^\theta]}\cdot e_y^\theta - \int_{\Gamma_h[\bf x^*]}  |\nabla_{\Gamma_h[\bf x^*]} \hat I_h^*n|^2 \big(\hat I_h^* u \cdot \hat \chi_u\big)\nabla_{\Gamma_h[\bf x^*]}\cdot e_y^0\big]\d\theta \bigg| \notag\\
    & \quad + \bigg| \int_{\Gamma_h[\bf x^*]} |\nabla_{\Gamma_h[\bf x^*]} \hat I_h^*n|^2 \big(\hat I_h^* u \cdot \hat \chi_u\big)\nabla_{\Gamma_h[\bf x^*]}\cdot e_y^0 - \int_{\Gamma} |\nabla_{\Gamma} n|^2 \big(u\cdot \hat \chi_u^l\big) \nabla_\Gamma\cdot (e_y^0)^l\bigg|\notag\\
    & \quad + \bigg| \int_{\Gamma} |\nabla_{\Gamma} n|^2 \big(u\cdot \hat \chi_u^l\big) \nabla_\Gamma\cdot (e_y^0)^l\bigg| = :E_{21} + E_{22} + E_{23}.
  \end{align}
By employing the intermediate surface between $\Gamma_h[\y^\theta]$ and $\Gamma_h[\x^*]$ again and referencing the inequalities \eqref{ey-L2-copy-2}, \eqref{en-L2} and \eqref{eu-L2}, we obtain
\begin{equation*}
  \begin{aligned}
    E_{21} \le Ch^{2k} \|\chi_u\|_{L^2(\Gamma_h[\bf y^*])}.
  \end{aligned}
\end{equation*}
By employing the intermediate surface between $\Gamma_h[\bf x^*]$ and $\Gamma$ throughout the lift map and utilizing the inequality $\|\nabla _{\Gamma_h[\bf x^*]} \cdot (x-x^l)\|_{L^p(\Gamma_h [\bf x^*])} \le Ch^k$ induced from Lagrange interpolation error estimates, we obtain
\begin{equation*}
  \begin{aligned}
    E_{22} \le Ch^{2k} \|\chi_u\|_{L^2(\Gamma_h[\bf y^*])}.
  \end{aligned}
\end{equation*}
By using integration by parts along with inequality \eqref{ey-L2-copy-2}, we derive 
\begin{equation*}
  \begin{aligned}
    E_{23} \le C\|e_y^0\|_{L^2(\Gamma_h[\bf x^*])} \|\chi_u\|_{H^1(\Gamma_h[\bf y^*])} \le Ch^{k+1} \| \chi_u\|_{H^1(\Gamma_h[\bf y^*])}.
  \end{aligned}
\end{equation*}
Substituting the estimates of $E_{2j}$ for $j=1,2,3$ into \eqref{rewriteE2}, we obtain 
\begin{equation*}
  \begin{aligned}
    E_2 \le Ch^{k+1} \| \chi_u\|_{H^1(\Gamma_h[\bf y^*])}.
  \end{aligned}
\end{equation*}
The terms $E_3$, $E_4$ and $E_5$ can be estimated similarly as $E_2$, then we have
\begin{align*}
  E_3 + E_4 + E_5 \le Ch^{k+1} \| \chi_u\|_{H^1(\Gamma_h[\bf y^*])}.
\end{align*}
Substituting the estimates of $E_{j}$ for $j=1,2,3,4,5$ into \eqref{rhs-du3}, we obtain the following result for the second term on the right-hand side of \eqref{du3.phi}:
\begin{align}\label{rhs-du3-result}
  \bigg| \int_{\Gamma_h[\x^*]}
|\nabla_{\Gamma_h[\x^*]}\hat I_h^*n|^2 \hat I_h^*u\cdot \hat\chi_u 
-\int_{\Gamma_h[\y^*]}
|\nabla_{\Gamma_h[\y^*]}n_h^*|^2 u_h^* \cdot \chi_u  \bigg| \le Ch^{k+1} \| \chi_u\|_{H^1(\Gamma_h[\bf y^*])}.
\end{align}
Then substituting inequalities \eqref{d3-1} and \eqref{rhs-du3-result} into inequality \eqref{du3.phi}, we obtain
\begin{align*}
  \bigg|\int_{\Gamma_h[\bf y^*]}d^*_{u,3} \cdot \chi_u\bigg| \le Ch^{k+1} \|{\chi_u}\|_{H^1(\Gamma_h[\bf y^*])},
\end{align*}
which proves Lemma \ref{Lemma:defect3}. 
\end{proof}
\hfill

Substituting results from Lemma \ref{Lemma:defects12} and Lemma \ref{Lemma:defect3} into \eqref{defect_H_identity}, we obtain
\begin{align*}
  \bigg|\int_{\Gamma_h[\bf y^*]}d^*_{u} \cdot \chi_u\bigg| \le Ch^{k+1} \|{\chi_u}\|_{H^1(\Gamma_h[\bf y^*])},
\end{align*}
which proves Lemma \ref{Lemma:defects}. 
\end{document}